\documentclass[10pt,oneside,english]{amsart}
\usepackage{lmodern}
\usepackage[T1]{fontenc}
\usepackage[latin9]{inputenc}
\usepackage{geometry}
\usepackage{amsbsy}
\usepackage{amsthm}
\usepackage{amssymb}
\usepackage{psfrag}
\usepackage{amstext}
\usepackage{color}
\usepackage{verbatim}
\usepackage{graphicx}
\usepackage{auto-pst-pdf}

\makeatletter
\numberwithin{equation}{section}
\numberwithin{figure}{section}
\theoremstyle{plain}
\newtheorem{thm}{\protect\theoremname}
  \theoremstyle{plain}
  \newtheorem{lem}[thm]{\protect\lemmaname}
\theoremstyle{plain}
\newtheorem{cor}[thm]{\protect\corollaryname}
\theoremstyle{plain}
\newtheorem{prop}[thm]{\protect\propositionname}
\theoremstyle{plain}
\newtheorem{remark}[thm]{\protect\remarkname}
\makeatother

\usepackage{babel}
\providecommand{\lemmaname}{Lemma}
\providecommand{\theoremname}{Theorem}
\providecommand{\corollaryname}{Corollary}
\providecommand{\propositionname}{Proposition}

\providecommand{\remarkname}{Remark}
\DeclareSymbolFont{matha}{OML}{txmi}{m}{it}
\newcommand{\ef}{{v}}

\makeatletter
\newsavebox\myboxA
\newsavebox\myboxB
\newlength\mylenA
\newcommand*\mybar[2][0.75]{%
    \sbox{\myboxA}{$\m@th#2$}%
    \setbox\myboxB\null
    \ht\myboxB=\ht\myboxA%
    \dp\myboxB=\dp\myboxA%
    \wd\myboxB=#1\wd\myboxA
    \sbox\myboxB{$\m@th\overline{\copy\myboxB}$}
    \setlength\mylenA{\the\wd\myboxA}
    \addtolength\mylenA{-\the\wd\myboxB}%
    \ifdim\wd\myboxB<\wd\myboxA%
       \rlap{\hskip 0.5\mylenA\usebox\myboxB}{\usebox\myboxA}%
    \else
        \hskip -0.5\mylenA\rlap{\usebox\myboxA}{\hskip 0.5\mylenA\usebox\myboxB}%
    \fi}
\makeatother

\newcommand{\oV}{{\text{\mybar{V}}}}
\newcommand{\oP}{\mybar{\mathcal{P}}}
\newcommand{\oPsi}{\mybar{\Psi}}

\newcommand{\nTAP}{\operatorname{TAP}}
\newcommand{\TAP}{\boldsymbol{\operatorname{T}}}
\newcommand{\oTAP}{\mybar{\boldsymbol{\operatorname{T}}}}

\newcommand{\la}{\langle}
\newcommand{\ra}{\rangle}

\newcommand{\e}{\mathbb{E}}
\newcommand{\p}{\mathbb{P}}
\newcommand{\Reals}{\mathbb{R}}

\global\long\def\bs{\boldsymbol{\sigma}}
\global\long\def\tbs{\tilde{\boldsymbol{\sigma}}}
\global\long\def\bb{\boldsymbol{b}}
\global\long\def\tbb{\tilde{\boldsymbol{b}}}
\global\long\def\br{\boldsymbol{\rho}}

\global\long\def\Es{E_{\star}}

\global\long\def\E{\mathbb{E}}
\global\long\def\P{\mathbb{P}}
\global\long\def\e{\mathbb{E}}
\global\long\def\p{\mathbb{P}}
\global\long\def\R{\mathbb{R}}

\global\long\def\indic{\mathbf{1}}
 
\global\long\def\ch{\cosh}
\global\long\def\supp{{\operatorname{supp}}}
\global\long\def\th{\tanh}
\global\long\def\PP{\mathcal{P}}

\newcommand{\eps}{{\varepsilon}}

\newcommand{\EA}{\operatorname{EA}}
\newcommand{\RS}{\operatorname{RS}}

\newcommand{\MM}{\mathcal{M}}

\begin{document}

\allowdisplaybreaks

		\title[Generalized TAP free energy]{The generalized TAP free energy}
		\author{Wei-Kuo Chen}\thanks{School of Mathematics, University of Minnesota. Email: wkchen@umn.edu. Partially supported by NSF grants DMS-16-42207 and DMS-17-52184.}
			\author{Dmitry Panchenko}\thanks{Department of Mathematics. University of Toronto. Email: panchenk@math.toronto.edu. Partially supported by NSERC}
			\author{Eliran Subag}\thanks{Courant Institute. Email: esubag@cims.nyu.edu. Supported by the Simons Foundation.}
		
		\begin{abstract} 
			We consider the mixed $p$-spin mean-field spin glass model with Ising spins and investigate its free energy in the spirit of the TAP approach, named after Thouless, Anderson, and Palmer \cite{TAP}. More precisely, we define and compute the generalized TAP correction, and establish the corresponding generalized TAP representation for the free energy. In connection with physicists' replica theory, we introduce the notion of generalized TAP states, which are the maximizers of the generalized TAP free energy, and show that their order parameters match the order parameter of the ancestor states in the Parisi ansatz. We compute the critical point equations of the  TAP free energy that generalize the classical TAP equations for pure states. Furthermore, we give an exact description of the region where the generalized TAP correction is replica symmetric, in which case it coincides with the classical TAP correction, and show that Plefka's condition is necessary for this to happen. In particular, our result shows that the generalized TAP correction is not always replica symmetric on the points corresponding to the  Edwards-Anderson parameter.
		\end{abstract}
		
	\maketitle		

		\section{Introduction}\label{sec1}
		

		How does a function on a high-dimensional space $\Reals^N$ (i.e. for large $N$) typically look like? For probabilists, ``typically'' means endowing some natural space of functions with a probability measure and understanding what occurs  with high probability. One natural family of functions consists of random homogeneous polynomial of degree $p\geq 1$ in the coordinates of $\bs=(\sigma_1,\ldots,\sigma_N)\in \R^N$,
		\begin{align}
			\label{hamp}
			H_{N,p}(\bs)=\frac{1}{N^{(p-1)/2}}\sum_{i_1,\ldots,i_p=1}^{N}g_{i_1,\ldots,i_p}\sigma_{i_1}\cdots\sigma_{i_p},
		\end{align}
		whose coefficients $g_{i_1,\ldots,i_p}$ are independent standard Gaussian variables. In this case, the domain is often restricted to the hypercube $\Sigma_N:=\{-1,+1\}^N$  or sphere $S_N:=\{\|\bs\|_2=\sqrt N\},$ depending on whether the motivation comes from a discrete or continuous setting; the scaling $N^{-(p-1)/2}$ in the definition of $H_{N,p}$ is chosen so that the maximum is typically of order $N$. More generally, assuming that the processes $H_{N,p}$ above are independent of each other for $p\geq 1,$ we will also consider their linear combinations
		\begin{align}
			\label{hamx}
			H_N(\bs)=\sum_{p\geq 1}\beta_p H_{N,p}(\bs),
		\end{align}
		for some  sequence $(\beta_p)_{p\geq 1}$ that decreases fast enough, for example, $\sum_{p\geq 1}2^p\beta_p^2<\infty.$ 
		
		In statistical physics, the random processes $H_N(\bs)$ are called spin glass models, or Hamiltonians. 
		A classical example is the Sherrington-Kirkpatrick (SK) model \cite{SK75}, defined by \eqref{hamp} with $p=2$ and $\bs\in \Sigma_N$. For general $p\geq 1$, $H_{N,p}(\bs)$ is called the pure $p$-spin model, and the linear combination $H_N(\bs)$ is called a mixed $p$-spin model. In this paper we will work with general mixed models with Ising spins, that is, when the domain of $H_N(\bs)$ is $\Sigma_N$. The spherical case $S_N$ will also be discussed occasionally to put things into a more general context.

		Going back to the question we started with, one may ask, for example, what is the maximal value of $H_N(\bs)$ over $\Sigma_N$ or $S_N$, or what is the structure of the set of all near maxima (on the right scale).
		More generally, one may wish to know what is the typical cardinality or volume of the set of points where $H_N(\bs)$ approximately takes a given value, or what is the structure of the same set. However, instead of tackling these questions directly, one often first studies ``smooth approximations'' of various quantities parametrized by the so-called  inverse-temperature parameter $\beta>0$ --- an idea common in statistical physics. For example, if for certainty we consider the Ising case $\Sigma_N$, relying on 		\eqref{eq:lvlsetvol} and \eqref{eq:lvlsetG} below,  instead of the cardinality of an approximate level  set and its geometric structure, one can first study the free energy 
		\begin{align}\label{add:eq0}
			F_N(\beta)&=\frac{1}{N}\log \sum_{\bs\in \Sigma_N}e^{\beta H_N(\bs)},
		\end{align}
and the Gibbs measure
		\begin{align}
			G_{N,\beta}(A)&=\frac{\sum_{\bs \in A}\exp (\beta H_N(\bs))}{ \sum_{\bs\in \Sigma_N}\exp (\beta H_N(\bs))}.
		\end{align}
In the spherical case, the summation is replaced by integration with respect to the Haar measure on $S_N$. It is well-known that the free energy concentrates around its expectation and that the expectation has a limit $F(\beta):=\lim_{N\to\infty}\e F_N(\beta)$ (\cite{GuerraToninelli}, \cite{Pan00}), which is differentiable in $\beta>0$ (\cite{TPMeasure, Talbook1,Talbook2}). Standard concentration of measure inequalities then imply that, asymptotically, the maximum of the Hamiltonian (also called the ground state energy, ignoring the minus sign) can be computed via the free energy as
		\begin{equation}
		\label{eq:GS}
		\Es:=\lim_{N\to\infty}\max_{\sigma\in\Sigma_N}\frac{H_N(\bs)}{N}=\lim_{\beta\to\infty}\frac{F(\beta)}{\beta}.
		\end{equation}
Moreover, it is known (\cite{AC18}) that given an energy level $E\in(0,\Es)$, if we choose $\beta$ so that $E= F'(\beta)$, the cardinality on the logarithmic scale of the corresponding approximate level set can be expressed via the free energy,
		\begin{equation}
		\label{eq:lvlsetvol}
		\lim_{N\to \infty}\frac1N\log \# \Big\{  \bs:  \Big|\frac1N H_N(\bs)- E \Big| < \eps \Big\} = F(\beta)-\beta F'(\beta)+O(\eps),
		\end{equation}
		and the Gibbs measure concentrates on the same set,
		\begin{equation}
		\label{eq:lvlsetG}
		\lim_{N\to \infty} G_{N,\beta}\Big\{  \bs:  \Big|\frac1N H_N(\bs)- E \Big| < \eps \Big\} = 1.
		\end{equation}
		In this paper, we will focus on results in the language of the free energies and Gibbs measures for all finite temperatures $\beta<\infty$ and in the follow-up paper \cite{CPSZT}, we translate these results to $\beta=\infty$ limit, which concerns near maximizers and their geometry.

	
		To motivate our main results and informally illustrate some of the ideas behind them, consider the following question. Can one identify (in some non-trivial way) points $m\in(-1,1)^N$ \emph{inside the cube} such that, for  small $\eps>0$,
		the \emph{narrow band} of configurations $\bs\in\Sigma_N$ close to the hyperplane perpendicular to $m,$ 
		\begin{equation}
		\label{eq:Band}
			B(m,\eps)=\Big\{\bs\in\Sigma_N:\,|R(\bs,m)-R(m,m)|=\frac{1}{N}|m\cdot (\bs-m)|<\eps \Big\},
		\end{equation}
contains a large number of points with some given energy $\frac1N H_N(\bs)\approx E$? 
As we mentioned above, studying this question means fixing $\beta$ as in \eqref{eq:lvlsetG} above and ``large number'' means that the Gibbs measure of the set of such points is not too small.
It turns out that, without additional structure, such points $m$ are too common to be interesting and, moreover, the ``measure'' of near maximizers in a band fluctuates too wildly to hope for a meaningful criterion. However, as was pointed out in \cite{SubagFEL}, if we add an additional constraint that there any many \emph{nearly orthogonal} directions $\bs-m$ inside the band with  $\frac1N H_N(\bs)\approx E$, then the answer is yes. In fact, for $E\in(0,\Es)$, such special points $m$ can be characterized through their energy $H_{N}(m)$ and location $m$ by checking that
		\begin{equation}
		\label{eq:H+T}
		\frac{\beta}{N} H_N(m)+\nTAP_\beta(\mu_m)\approx F(\beta),
		\end{equation}
		where $\nTAP_\beta(\mu_m)$, which we call the \emph{generalized TAP correction} (after Thouless, Anderson and Palmer), is a \emph{deterministic} function of 
		the empirical measure 
		\begin{equation}
		\mu_m=\frac1N\sum_{i\leq N}\delta_{m_i}.
		\label{eqEmpMu0}
		\end{equation} 	
		Moreover, with high probability over the choice of random coefficients in $H_N(\bs)$, this criterion can be applied \emph{simultaneously} to all $m\in(-1,1)^N$. As will be explained in the next section,  in addition to yielding such a surprisingly simple description of the ``special points'', the idea of looking at many nearly orthogonal directions has a clear motivation coming from the theory of spin glasses in physics.

		While (\ref{eq:H+T}) will be proved in the current work, in the follow-up work  \cite{CPSZT}  we will deal with the zero temperature $\beta=\infty$ analogue of these results, or, equivalently, the maximal energy value $E=\Es$. In this case, with appropriate deterministic function $\nTAP_\infty$, the special points $m$ whose bands have properties analogous to the above are characterized by 
		\begin{equation}
		\label{eq:H+T_GS}
		\frac{1}{N} H_N(m)+\nTAP_\infty(\mu_m)\approx \Es.
		\end{equation}

		In a recent paper \cite{SubagFEL}, a natural way to define 
		the correction $\nTAP_\beta(\mu_m)$ in the spherical-spin case $S_N$ was suggested, from which the following \emph{generalized TAP representation} follows rather quickly:
		for any $q\in[0,1)$ that belongs to the support of the so-called Parisi measure (see (\ref{eqParisiOrig}) below for definition), for large $N$,
		\begin{equation}
		\label{eq:TAPrep}
		F(\beta) \approx \max_{m:\,\|m\|^2=Nq}\Bigl(\frac{\beta}{N} H_N(m)+\nTAP_\beta(\mu_m)\Bigr).
		\end{equation} 
		While the definition in \cite{SubagFEL} also makes sense for models with Ising spins, it is not obvious at all that this correction can be computed explicitly for those models. In this paper, we introduce some new ideas to solve this problem, and explicitly express the correction via a Parisi-type \cite{Parisi79, Parisi80} variational formula. As a result, this yields the generalized TAP representation \eqref{eq:TAPrep} for the Ising case.

		Representations of the type \eqref{eq:H+T} started from the paper \cite{TAP}, where Thouless, Anderson and Palmer  derived (non-rigorously, using an expansion of  the partition function around the local magnetizations) a representation for the free energy of the SK model \cite{SK75} called the \emph{TAP free energy,}
		\begin{align}\label{eq-2}
			F_N(\beta)&\approx \frac{\beta}{N}H_N(m)-\frac{1}{N}\sum_{i=1}^N\Bigl(\frac{1+m_i}{2}\log \frac{1+m_i}{2}+\frac{1-m_i}{2}\log \frac{1-m_i}{2}\Bigr)+\frac{\beta^2}{2}\Bigl(1-\frac{\|m\|_2^2}{N}\Bigr)^2,
		\end{align}
		where $m\in [-1,1]^N$ is some critical point of the right-hand side, $\nabla \mathrm{RHS}=0$ or
		\begin{equation}
			\label{eq:TAPeq}
			m_i=\tanh\Bigl(\bigl(\beta\nabla H_N(m)\bigr)_{i}-2\beta^2 m_i\Bigl(1-\frac{\|m\|_2^2}{N}\Bigr)\Bigr),\quad \forall i\leq N.
		\end{equation}
		As the authors of \cite{TAP} explained, the problem of computing the free energy is then reduced to 
		finding the random solutions $m$ of \eqref{eq:TAPeq}, known as the \emph{TAP equations},  subject to a certain convergence condition proposed in \cite{TAP}, and applying \eqref{eq-2} to $m$ --- a problem ``not much easier'' than the original, in their own words. The representation \eqref{eq:TAPrep} we establish in this paper is a more general analogue of the TAP free energy representation (\ref{eq-2}), which is well-motivated and fully rigorous. 
		 Similarly, 
		 by computing the critical point equations for the right-hand side of \eqref{eq:TAPrep}, we will derive the \emph{generalized TAP equations} analogous to \eqref{eq:TAPeq}. As will be explained in the next section, the motivation comes from the picture that emerged in the subsequent work of physicists in the eighties.
		
		A few years after \cite{TAP}, a real breakthrough was made by Parisi in \cite{Parisi79, Parisi80}, who discovered the correct formula for the free energy by proposing a very special ansatz within the physicist's replica method. The Parisi solution, which was rather algebraic in nature, was reinterpreted in terms of the geometric structure of the Gibbs measure in the papers by M\'ezard, Parisi, Sourlas, Toulouse and Virasoro \cite{M1, M2}, where it was understood, for example, that the ultrametricity of the replica matrix corresponds to ultrametricity of the support of the Gibbs measure in the infinite-volume limit and that the Gibbs measure asymptotically splits into pure states, i.e., disjoint subsets whose structure is simple in an appropriate sense. The so-called order parameter in Parisi's solution, which also plays an important role in the current paper,  is a probability measure on $[0,1]$ called the Parisi measure;  for generic models (defined below) it coincides with the asymptotic law of the overlap $R(\bs^1,\bs^2):=\frac{1}{N}\sum_{i=1}^N\sigma_i^1\sigma_i^2$ of i.i.d. samples from the Gibbs measure.

		
		The connection between the Parisi ansatz and the  classical TAP free energy representation (\ref{eq-2}) was well-understood in the physics literature. In the setting of the SK model, some form of TAP equations for the ancestor states (see below) were derived by M\'ezard and Virasoro in \cite{M3}. However, rigorous mathematical results beyond the high temperature region (see, e.g., \cite{EB,chatt10,Talbook1}) started appearing only more recently. For example, it was confirmed in \cite{CPTAP17} that the TAP representation of the free energy holds at the level of pure states, that is, a formula of the form of \eqref{eq:TAPrep} holds for $q_{\EA}$, the right-most point in the support of the Parisi measure. Also, the M\'ezard-Virasoro equations \cite{M3} for mixed $p$-spin models were derived by Auffinger and Jagannath \cite{AufJag18,AJ19} (see also Remark \ref{LabRmkMV} below). Furthermore, recently, Belius and Kistler \cite{BelKist} developed a new method in the setting of the spherical $2$-spin model.	Lastly, the TAP representation was established at the level of pure states for the spherical pure $p$-spin models with $p\geq3$ and $\beta\gg 1$ in \cite{S17} and for some spherical mixed $p$-spin models in the 1-RSB regime with $\beta\gg 1$ by Ben Arous, Zeitouni and one of the authors in \cite{BSZ}. In the latter works \cite{S17,BSZ}, the calculations leading to the TAP representation also yielded certain explicit pure state decompositions, in which each state is centered around a local maximum of the Hamiltonian which also maximizes a certain free energy.  More generally, the barycenters, or the so-called local magnetizations, of the abstract pure states decompositions of Talagrand \cite{Talagrand10} and Jagannath \cite{Jagannath17} are approximate maximizers $m$ as in \eqref{eq:TAPrep}, which correspond to the rightmost point in the support of the Parisi measure. 

In this work, we mainly focused on positive temperature $\beta<\infty$ analysis, which deals with energy levels $\frac1N H_N(\bs)\approx  E$ strictly smaller than the ground state energy $\Es$.
However, the set of near maximizers $\frac1N H_N(\bs)\approx  \Es$  also has a rich and interesting geometry.
For example, it was understood, both in the physics literature \cite{MPV} and rigorously \cite{ChatBook,CHL, GEZ}, that $H_N$ has exponentially many (in $N$) near maximizers  that are nearly orthogonal to each other. For the spherical models, similar results about the highest critical points are known for the pure models \cite{ABAC,CLR03,CS95,Subag17,SZ17} and some mixed models, which are close to being pure \cite{ABAC,BSZ}.  

	As mentioned above, in  \cite{CPSZT}
we extend our analysis to the zero temperature case $\beta=\infty$.  One of the consequences of \cite{CPSZT} is that a large set of approximate maximizers of
$
\frac{1}{N} H_N(m)+\nTAP_\infty(\mu_m)
$, i.e., approximate generalized TAP solutions,
can be  arranged in a certain tree structure, whose root is the origin and leaves are points of $\Sigma_N$. Since $\nTAP_{\infty}(\mu_{\bs})$ is constant on $\Sigma_N$, the leaves approximately maximize $\frac1N H_N(\bs)$. 
This picture is particularly interesting when \emph{full replica symmetry breaking} (FRSB) occurs on the interval $[0,q_{\EA}]$, namely, the support of the Parisi measure is equal to the interval $[0,q_{\EA}]$ as $\beta$ tends to infinity, which is conjectured to be the case in the SK model, see \cite{MPV}. 
In this case, the normalized radii $\|m\|/\sqrt N$ of the inner vertices of the tree, which are points $m\in(-1,1)^N$, are asymptotically dense in $[0,1]$ and the tree is asymptotically continuous in an appropriate sense. 
In the spherical case, when the model is FRSB on $[0,q_{\EA}],$ similar insights from \cite{SubagFEL} inspired an optimization algorithm designed in \cite{ES18}, which outputs a configuration in $S_N$ that roughly maximizes $H_N(\bs)$ in polynomial time in $N$.
In the Ising case, Montanari \cite{AM18} achieved the same optimization result for the SK model by utilizing the {\it Approximate Message Passing} (AMP) algorithm based on the TAP equations.
Both algorithms start from the origin and iteratively move towards $S_N$ or $\Sigma_N$ using orthogonal updates, until reaching the approximate optimizer.
Montanari proved that his algorithm ends at an approximate TAP solution. In fact, we believe that in each iteration the algorithm jumps from one approximate TAP solution to another. An extension of the algorithm from \cite{AM18}, which optimizes models with FRSB on $[0,q_{\EA}]$ was constructed by El Alaoui, Montanari and Sellke in \cite{EMS20}. 	In addition to be useful in the optimizations of the mixed $p$-spin Hamilonians, the AMP algorithms driven by the TAP equations have also received great popularity in a number of Baysian inference problems, see, e.g., \cite{KKM+16,MR+16,MV+18,ZK+16}. Finally, we mention that  in another direction, when a certain overlap gap property holds, e.g. for the pure $p$-spin model with even $p\geq4$, it was proved  that a broad class of algorithms, such as Lipschitzian iteration schemes and low-degree methods fail to produce near ground states for $H_N(\bs)$ in polynomial time in $N,$ see \cite{GJ19,GJA20}. This property is expected to hold
generically if the model is not FRSB on $[0,q_{\EA}]$, for instance, when $\beta_2=0$ and $\beta_p>0$ for some $p\geq 3$ (see \cite{AC15.1,WKGPM,JT17}).

\section{\label{sec:results}Main results}
 
\subsection{The model.} 
In this paper we will consider the mixed $p$-spin Hamiltonian $H_N(\bs)$ defined in \eqref{hamx} with Ising spins, indexed by $\bs\in \Sigma_N = \{-1,+1\}^N$. The covariance of the Gaussian process $H_N(\bs)$ equals
		\begin{align}
		\e H_N(\bs^1)H_N(\bs^2)=N\xi\bigl(R(\bs^1,\bs^2)\bigr),
		\end{align} 
		where $R(\bs^1,\bs^2)=\frac{1}{N}\sum_{i=1}^N\sigma_i^1\sigma_i^2$ is called the overlap of $\bs^1$ and $\bs^2$, and where
		\begin{align}
		\xi(s)=\sum_{p\geq 1}\beta_p^2s^p.
		\end{align} 
In the Introduction we allowed the (random) free energy and Gibbs measure depend on an inverse-temperature parameter $\beta>0$. Of course, $\beta$ can be absorbed into the coefficients $\beta_p$. Hence, to simplify the notation, we redefine the 
free energy by
		\begin{align}
		F_N&=\frac{1}{N}\log \sum_{\bs\in \Sigma_N}\exp H_N(\bs),
		\end{align}
 and the Gibbs measure by
		\begin{align}
		G_N(\bs)&=\frac{\exp H_N(\bs)}{ \sum_{\bs\in \Sigma_N}\exp H_N(\bs)},
		\end{align}
and henceforth use these definitions which do not include $\beta$.
One can also add an external field term $h\sum_{i\leq N}\sigma_i$ to the Hamiltonian $H_N(\sigma)$, but, for simplicity of notation, we will usually omit it (see also Remark \ref{rmkMEF} below).		
		
		A special role will be played by the so called \emph{generic} mixed $p$-spin models that satisfy
\begin{align}
\mybar{\operatorname{Span}}\bigl\{x^p \,:\, \beta_p\not = 0\bigr\} = C\bigl([-1,1],\|\cdot\|_\infty\bigr),
\label{Generic}
\end{align}
which means that sufficiently many of the $p$-spin terms in the Hamiltonian (\ref{hamx}) are present in the model.
 
The limit of the free energy $F_N$ is given by the celebrated Parisi formula \cite{Parisi79, Parisi80} mentioned above, which was first proved in a seminal work of Talagrand in \cite{Tal03} (building upon a breakthrough by Guerra \cite{Guerra}), and later generalized to models with odd spin interactions in \cite{Pan00}. (The formula for $\Es$ in (\ref{eq:GS}) was derived in \cite{ChenAuffGSE}.) If $\mathcal{M}_{0,1}$ is the space of probability measures on $[0,1]$, for $\zeta\in \mathcal{M}_{0,1},$ let $\Phi_{\zeta}(t,x)$ be the solution of the Parisi PDE
\begin{equation}
\partial_t \Phi_{\zeta} = -\frac{\xi''(t)}{2}\Bigl(
\partial_{xx} \Phi_{\zeta} + \zeta(t)\bigl(\partial_x \Phi_{\zeta}\bigr)^2
\Bigr)
\label{ParisiPDEOrig}
\end{equation}
on $[0,1]\times \Reals$ with the boundary condition $\Phi_{\zeta}(1,x)=\log2\ch x.$ Here $\zeta(s):=\zeta([0,s]).$ Define the Parisi functional on $\mathcal{M}_{0,1}$ by
\begin{equation}
\PP(\zeta):= \Phi_{\zeta}(0,0)-\frac{1}{2}\int_0^{1}\!s\xi''(s)\zeta(s)\,ds.
\end{equation}
Then, the limit of the free energy is given by
\begin{equation}
\lim_{N\to\infty}\e F_N = 
\inf_{\zeta\in \mathcal{M}_{0,1}}\PP(\zeta).
\label{eqParisiOrig}
\end{equation}
The minimizer $\zeta_*$ is unique (see \cite{AC15}, also \cite{JT16}) and is called the \emph{Parisi measure}. The solution of the above PDE is usually constructed explicitly for discrete $\zeta$ and extended by continuity to all $\zeta$, but one can also show its uniqueness (see \cite{JT16}).

\subsection{Motivation via infinitary nature of the Parisi tree.}\label{sec2.2}
The Parisi ansatz, schematically depicted in Figure \ref{Fig1}, states that the Gibbs measure asymptotically decomposes into disjoint pure states, whose magnetizations (barycenters) are organized ultrametrically (see \cite{MPV}). For simplicity, we plotted only a finite-RSB scenario but, in principle, the overlap can take infinitely many values. 

The Parisi ansatz holds for any model that satisfies the Ghirlanda-Guerra identities (see \cite{ultrametricity}) and, in particular, it holds for generic mixed $p$-spin models (see Section 3.7 in \cite{SKmodel}). Since any mixed $p$-spin model can be approximated by generic models at the level of the free energy, all the results below will apply to non-generic models as well, and the Parisi ansatz for the generic models will be used as guiding our motivation.

\begin{figure}[t]
\centering
 \psfrag{QuM21}{ { $m$}}
 \psfrag{Que31}{ { $q$}}
 \psfrag{Que71}{ { $q_{EA}$}}
  \psfrag{Que53}{ { $1$}}
 \psfrag{Sigma43}{ { $\bs$}}
 \psfrag{AncestorS}{\small ancestor states}
 \psfrag{PureS}{\small pure states}
 \psfrag{ConfigS}{\small configurations}
 \hspace{-22mm}
 \includegraphics[width=0.8\columnwidth]{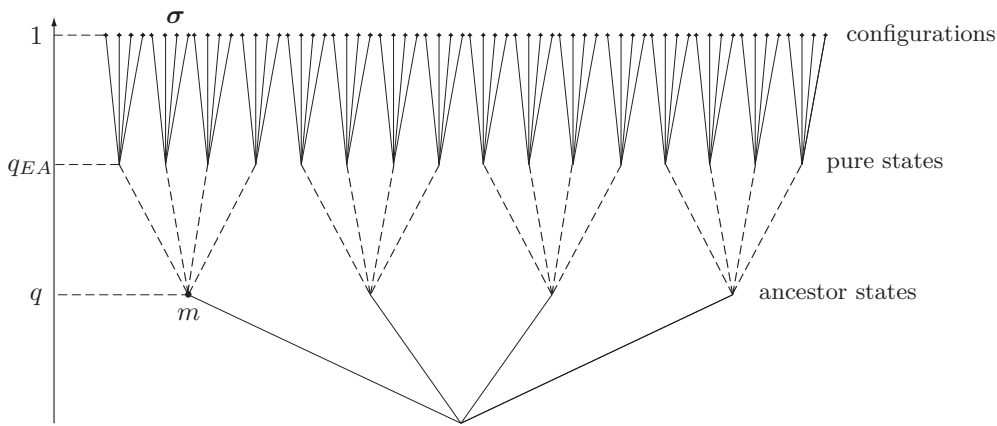}
 \caption{\label{Fig1} Ancestor state $m$ on the infinitary tree of states with self-overlap $\frac{1}{N}\|m\|^2=q\leq q_{EA}$.}
 \label{fig1}
 \end{figure}

Vertices in the tree in Figure \ref{Fig1} below the level of pure states are called \emph{ancestor states} and they represent branching points when clusters break into smaller subclusters as we zoom in on individual configurations. These ancestor states also have the physical meaning of points $m$ inside the cube
\begin{equation}
m \in [-1,1]^N.
\label{Mconstr}
\end{equation}
(In the spherical models the cube is replaced by the ball of radius $\sqrt{N}$.) The key feature of the Parisi ansatz is that this hierarchical tree of states is \emph{infinitary} in the thermodynamic limit, which means that at each branching point there are infinitely many edges (corresponding to subclusters of a bigger cluster). This infinitary property, in particular, means that all edges in this tree are orthogonal and, for example, a point $m$ corresponding to an ancestor state is perpendicular to $\bs-m$ for any configuration $\bs$ (on the scale $1/{N}$) coming from a pure state which is a descendant of $m$. This infinitary property of the Parisi ansatz was used implicitly or explicitly in many applications of ultrametricity (for example, in the proof of ultrametricity itself as well as chaos in temperature in  \cite{ultrametricity, PChaosT}, and in the proof of the synchronization mechanism in \cite{PMS15, panchenko2018}), and it is central to the main idea in \cite{SubagFEL} as well as the current paper, which we will explain next.
 
For $m$ as in (\ref{Mconstr}) and $\eps>0,$ we recall the definition of the band centered at $m$ from the introduction,
\begin{equation}
B(m,\eps)=\Big\{\bs\in\Sigma_N:\,|R(\bs,m)-R(m,m)|=\frac{1}{N}|m\cdot (\bs-m)|<\eps \Big\}.
\end{equation}
If $\eps\geq 2/\sqrt{N}$ then all the bands are non-empty, which can be seen, for example, from Bernstein's inequality: if $\bs$ comes from the product measure on $\Sigma_N$ with mean $m$ then
$$
\p\bigl(|m\cdot (\bs-m)|\geq N\eps\bigr)\leq e^{-\frac{N\eps^2}{1+\sqrt{N}\eps/3}}
\leq e^{-\frac{4}{1+2/3}}<1.
$$

Given $\delta>0$ and $n\geq 1,$ let us consider a set consisting of $n$ configurations in this narrow band $\bs^{1},\ldots,\bs^{n}\in B(m,\eps)$ such that all 
\begin{equation}
\tbs^i = \bs^i-m
\end{equation}
are almost orthogonal to each other,
\begin{equation}
\label{eq:Bn}
B_n(m,\eps,\delta)
= \Big\{ (\bs^{1},\ldots,\bs^{n})\in B(m,\eps)^n:\,\forall i\neq j,\,\,\big|R({\bs}^{i},{\bs}^{j})-R(m,m)\big|<\delta\Big\}.
\end{equation}
Here, when $n=1$, this is understood as $B_1(m,\eps,\delta)=B(m,\eps).$
Heuristically, if $m$ corresponds to an ancestor state then all the descendant pure states are in the band $B(m,\eps)$, so the band carries some non-negligible weight of the Gibbs measure. Moreover, by the infinitary nature of the tree, we can choose many nearly orthogonal configurations relative to $m$ with non-negligible Gibbs probability. This means that, for such $m$, the inequalities
\begin{equation}
F_N 
\geq \frac{1}{N}\log\sum_{B(m,\eps)}e^{H_{N}(\bs)}
\geq \frac{1}{nN}\log\sum_{B_{n}(m,\eps,\delta)}e^{\sum_{i=1}^{n}H_{N}(\bs^{i})}
\end{equation}
are approximate equalities. Let us introduce the quantity 
\begin{equation}
\label{eq:TAPn}
\nTAP_{N,n}(m,\eps,\delta) :=\frac{1}{nN}\log\sum_{B_{n}(m,\eps,\delta)}e^{\sum_{i=1}^{n}\big[H_{N}(\bs^{i})-H_N(m)\big]},
\end{equation}
which, for simplicity of notation, will often be written with $\eps$ and $\delta$ omitted,
\begin{equation}
\nTAP_{N,n}(m):=\nTAP_{N,n}(m,\eps)=\nTAP_{N,n}(m,\eps,\delta).
\end{equation}
Then the above inequalities can be rewritten as
\begin{equation}
F_N 
\geq \frac{H_N(m)}{N} + \nTAP_{N,1}(m,\eps)
\geq \frac{H_N(m)}{N} + \nTAP_{N,n}(m,\eps,\delta).
\end{equation}
Again, for the ancestor states with the self-overlap $\frac{1}{N}\|m\|^2=q$ corresponding to some $q$ in the support of the Parisi measure we expect these to be approximate equalities. Moreover, for $q$ in the support of the Parisi measure, one can show that such ancestor states exist, which will imply that
\begin{equation}
F_N \approx \max_{\frac{1}{N}\|m\|^2=q}\Bigl( \frac{H_N(m)}{N} + \nTAP_{N,n}(m,\eps,\delta)\Bigr).
\end{equation}
What do we gain by appealing to the infinitary nature of the tree of states in this way?

Given $a,b\in\Reals,$ let $\MM_{a,b}$ denote the space of probability measures on $[a,b]$, equipped with the topology of weak convergence. We will always implicitly identify a probability measure $\mu$ with its c.d.f. and, for simplicity of notation, write $\mu(x):=\mu((-\infty,x])$. For $\mu,\mu'\in \MM_{a,b}$, we will work with the metric
\begin{equation}
d_1(\mu,\mu')=\int_a^b\! |\mu(x)-\mu'(x)|\,dx,
\label{eqMd1}
\end{equation}
which metrizes weak convergence. We will keep the dependence of $d_1$ on $a$ and $b$ implicit.

Recall the definition \eqref{eqEmpMu0} of the empirical measure 
\begin{equation*}
	\mu_m=\frac1N\sum_{i\leq N}\delta_{m_i}\in \MM_{-1,1}
\end{equation*} 
of $m=(m_i)_{i\leq N}\in[-1,1]^N$. The key point will be that, for small $\eps$ and $\delta$ and large $n$, we can write, with high probability,
\begin{equation}
\nTAP_{N,n}(m,\eps,\delta) \approx \nTAP(\mu_m)
\label{eqUnif}
\end{equation}
\emph{uniformly over all $m$} in (\ref{Mconstr}), for some specific {non-random} functional $\nTAP\colon \MM_{-1,1}\to\Reals$. This functional will be our \emph{generalized TAP correction term} and
\begin{equation}
F_N \approx \max_{\frac{1}{N}\|m\|^2=q}\bigl( \frac{H_N(m)}{N} + \nTAP(\mu_m)\bigr)
\end{equation}
for $q$ in the support of the Parisi measure is the \emph{generalized TAP representation} of the free energy. Moreover, heuristically, the ancestor states $m$ in the Parisi tree are among the TAP states (near maximizers of the right hand side), and we will show that these states very much resemble the ancestor states. Let us describe the generalized TAP correction and state our main results precisely.

\subsection{Generalized TAP correction and representation}

For $\zeta\in \mathcal{M}_{0,1}$, recall the Parisi PDE solution $\Phi_{\zeta}(q,x)$ from \eqref{ParisiPDEOrig}. Denote the concave conjugate of $\Phi_\zeta(q,\cdot)$ by
\begin{equation}
\Lambda_\zeta(q,a):=\inf_{x\in\Reals}\Bigl(\Phi_\zeta(q,x)-ax\Bigr), \,\, a\in [-1,1].
\label{eqCCPar}
\end{equation}
It is well-known that $\Phi_\zeta(q,x)$ is a strictly convex function in $x$ and it goes to $\pm 1$ as $x\to\pm\infty$, see \cite{AC15}. Hence, for each $a\in (-1,1),$ the variational problem defined in  $\Lambda_\zeta(q,a)$ has a unique minimizer $\oPsi(q,a,\zeta)$, which satisfies
\begin{equation}
\partial_x \Phi_{\zeta}\bigl(q,\oPsi(q,a,\zeta)\bigr) = a,\,\, a\in (-1,1).
\label{eqDefPsiefO}
\end{equation}
With this notation, we can also write
\begin{equation}
\Lambda_\zeta(q,a)=\Phi_\zeta\bigl(q,\oPsi(q,a,\zeta)\bigr)-a\oPsi(q,a,\zeta), \,\, a\in (-1,1).
\label{eqCCParPsi}
\end{equation}
We will see that the infimum in (\ref{eqCCPar}) is finite for $a\in \{-1,1\}$, so the function $\Lambda_\zeta(q,a)$ is continuous on $[-1,1]$. Moreover, since $x\to\Phi_\zeta(q,x)$ is even, so is $a\to\Lambda_\zeta(q,a).$

For $\mu\in \MM_{-1,1}$ such that $\int\! a^2\, d\mu(a)=q\in [0,1]$, we define
\begin{equation}
\nTAP(\mu,\zeta):=
\int\! \Lambda_{\zeta}(q,a)\,d\mu(a)-\frac{1}{2}\int_q^{1}\!s\xi''(s)\zeta(s)\,ds.
\label{eqTAPfirstZeta}
\end{equation}
Note that, when $\mu=\delta_1$ (and $q=1$), the functional is identically equal to zero. Let
\begin{equation}
\nTAP(\mu):=
\inf_{\zeta\in \mathcal{M}_{0,1}} \nTAP(\mu,\zeta).
\label{eqTAPfirst}
\end{equation}
Notice that, for a fixed $q$, this definition depends only on the values of $\zeta(s)$ on the interval $[q,1]$. This means that we could, equivalently, write
\begin{equation}
\nTAP(\mu):=
\inf_{\zeta\in \mathcal{M}_{q,1}} \nTAP(\mu,\zeta),
\label{eqTAPfirstQ}
\end{equation}
where $\mathcal{M}_{q,1}$ is the space of probability distributions on $[q,1]$. We will show in Theorem \ref{lem:TAPnew} below that the infimum is achieved and the minimizer is unique in $\mathcal{M}_{q,1}$. Because of this, whenever we use the representation (\ref{eqTAPfirst}), it will be convenient to use the convention that we minimize over $\zeta\in \mathcal{M}_{0,1}$ such that 
\begin{equation}
\zeta(s)=0 \,\mbox{ for }\, s\in [0,q).
\label{ConventionQ}
\end{equation}
The following is our main result. Recall $\mu_m$ in (\ref{eqEmpMu0}).
\begin{thm}[Generalized TAP correction]
\label{thm:TAPcorrection}
For any $c, t>0$, if $\eps,\delta>0$ are small enough and $n\geq 1$ is large enough then, for large $N$,
\begin{equation}
\label{eq:TAPunifconv}
\p \Bigl( \forall m \in [-1,1]^N : \ \big|\nTAP_{N,n}(m,\eps,\delta)  -  \nTAP(\mu_m) \big| < t \Bigr)>1-e^{-cN}.
\end{equation}
\end{thm}
In particular, we can let $t=t_N$ go to zero slowly with $N$ if we let $\eps=\eps_N$ and $\delta=\delta_N$ go to zero and $n=n_N$ go to infinity slowly enough. Once we computed the TAP correction, we get the TAP representation for the free energy.

\begin{thm}[Generalized TAP representation]\label{thm:GenTAP}
For any $q$ in the support of the Parisi measure of the original model (\ref{hamx}), in probability,
\begin{equation}
\label{eq:GenTAP}
\lim_{N\to\infty}\Bigl|F_N - \max_{\frac{\|m\|^2}{N}= q}\Bigl( \frac{H_N(m)}{N} + \nTAP(\mu_m)\Bigr)\Bigr|=0.
\end{equation}
\end{thm}

Recall that in \cite{CPTAP17}, it was proved that if $q_{\EA}$ is the largest point in the support of the Parisi measure for the original Parisi formula of $F_N$, then 
\begin{align}\label{add:eq-1}
	\lim_{N\rightarrow\infty}F_N&=\lim_{\varepsilon\downarrow 0}\lim_{N\rightarrow\infty}\sup_{m\in [-1,1]^N:\frac{\|m\|_2^2}{N}\in [q_{\EA}-\varepsilon,1]}\Bigl(\frac{H_N(m)}{N}-\int I(a)d\mu_m(a)+C(q)\Bigr),
\end{align}
where
\begin{equation}
\label{eq:IC}
\begin{aligned}
	I(a)&:=\frac{1+a}{2}\log \frac{1+a}{2}+\frac{1-a}{2}\log \frac{1-a}{2},
	\\
	C(q)&:=\frac{1}{2}\bigl(\xi(1)-\xi(q)-\xi'(q)(1-q)\bigr).
\end{aligned}
\end{equation}
Under an appropriate condition on the empirical measure $\mu_m$, we shall see that the variational formula defined in $\nTAP(\mu_m)$ is solved by the replica symmetric solution in the sense that the minimizer is the Dirac measure at the origin and moreover, our TAP correction term coincides with the sum of the entropy and correction terms in \eqref{add:eq-1}, that is, $\nTAP(\mu_m)=-\int I(a)\mu_m(da)+C(q),$ see Proposition \ref{add:prop2} below. This  allows us to conclude the following version of the classical TAP representation for the free energy from the general representation in Theorem \ref{thm:GenTAP}.
We will explain this in more detail and discuss the relation with Plefka's condition \cite{Plefka} in Section \ref{subsec:classical}.

\begin{cor}[Classical TAP representation] \label{add:cor1}The following equation holds almost surely
	\begin{align*}
		\lim_{N\rightarrow\infty}F_N&=\lim_{\varepsilon\downarrow 0}\lim_{N\rightarrow\infty}\max\Bigl(\frac{H_N(m)}{N}-\int I(a)d\mu_m(a)+C(q)\Bigr),
	\end{align*}
	where the maximum is taken over all $m\in [-1,1]^N$ satisfying  $
	\sup_{0\leq s\leq 1-q}\Gamma_{\mu_m}(s)\leq \varepsilon$, where $q=\|m\|_2^2/N$ and $\Gamma_\mu(s)$ is defined in \eqref{eq:Gamma_mu}.
\end{cor}

In what follows, we call the  near maximizers of the functional 
$$
\frac{H_N(m)}{N} + \nTAP(\mu_m)
$$
the \emph{generalized TAP states}.
Our definition of the TAP correction was motivated by the fact that ancestor states in the Parisi tree of states should be among the TAP states, if $\frac{1}{N}\|m\|^2$ is close to the support of the Parisi measure. Next, we will see that the TAP states have the properties one expects from the ancestors states.

\subsection{Properties of generalized TAP states.}

Let us denote by $\zeta_m$ the minimizer in (\ref{eqTAPfirst}) or (\ref{eqTAPfirstQ}) (recall our convention (\ref{ConventionQ})) corresponding to $m\in [-1,1]^N$ with $q=\frac{1}{N}\|m\|^2.$ We will see below that $\zeta_m$ has the meaning of the distribution of the overlap for the model on the narrow band $B(m,\eps)$ with its own random external field removed and with a new non-random external field added that forces $0$ in the support of this distribution $\zeta_m$ (see next section for details). We will show that, if $m$ is a generalized TAP state, then $\zeta_m(s)\approx \zeta_*(s)$ for $s\in [q,1]$, so the order parameters on the band around TAP state agrees with the Parisi measure of the original model on the interval $[q,1]$. To show this, we will upper bound the TAP correction by
\begin{equation}
\label{eq:TAPzetastar}\nTAP(\mu) = \inf_{\zeta\in \mathcal{M}_{0,1}} \nTAP(\mu,\zeta)\leq \nTAP(\mu,\zeta_*),
\end{equation}
and obtain the following.
\begin{thm}[TAP states are ancestral]\label{Thm1label}
For any $q$ in the support of the Parisi measure of the original model (\ref{hamx}), in probability,
\begin{equation}
\label{eqTAPAS}
\lim_{N\to\infty}
\Bigl|F_N - \max_{\frac{\|m\|^2}{N}= q}\Bigl( \frac{H_N(m)}{N} + \nTAP(\mu_m,\zeta_*)\Bigr)\Bigr|=0.
\end{equation}
\end{thm}
This together with the representation (\ref{eq:GenTAP}) implies that, if $m$ is a TAP state with $q=\frac{1}{N}\|m\|^2\in \supp(\zeta_*)$ then 
$$
\nTAP(\mu_m,\zeta_{m})\approx \nTAP(\mu_m,\zeta_*).
$$ 
By continuity properties of the Hamiltonian and the functional $\nTAP(\mu,\zeta)$ proved below, this also holds for states with $q=\frac{1}{N}\|m\|^2$ close to the support of $\zeta_*$. We will see (in the proof of Theorem \ref{lem:TAPnew} below) that the functional $\zeta\to\nTAP(\mu,\zeta)$ is $d_1$-Lipschitz uniformly over $\mu$ and has a unique minimizer $\zeta_{\mu}\in \mathcal{M}_{q,1}$ if $q=\int\! a^2 d\mu(a)$, which qualitatively means that 
 \begin{equation}
 \zeta_m\approx \zeta_*,
 \end{equation}
 i.e. the order parameter $\zeta_m$ in the TAP states follows the Parisi measure. This approximation can be quantified, but we do not pursue it here.
 
Next, in order to describe the critical point equations for the TAP states,
\begin{equation}
\frac{1}{N}\nabla H_N(m)=-\nabla \nTAP(\mu_m),
\label{eqTAPstates1}
\end{equation} 
we need to compute the gradient of $\nTAP(\mu_m).$ Recall the definition of $\oPsi(q,a,\zeta)$ in (\ref{eqDefPsiefO}) and let
\begin{equation}
\Psi(q,a,\zeta) :=\oPsi(q,a,\zeta)+ a\int_q^1\!\xi''(s)\zeta(s)\,ds.
\label{eqDefPsiefF}
\end{equation}
The gradient is given by the following formula.
\begin{thm}[Gradient of TAP correction]\label{ThmGTElab}
For any $m\in (-1,1)^N$ with $\frac{1}{N}\|m\|^2=q,$ if we denote
\begin{equation}
R(m):=
\xi''(q)\int_{q}^1\! \zeta_{m}(s)\,ds
-\int_{q}^{1}\!\xi''(s)\zeta_m(s)\,ds
\end{equation}
then
\begin{align}
\nabla \nTAP(\mu_m) 
&= -\frac{1}{N}\Bigl(\Psi(q,m_i,\zeta_m)+R(m)m_i\Bigr)_{i\leq N}
\nonumber
\\
&=
-\frac{1}{N}\Bigl(\oPsi(q,m_i,\zeta_m)+m_i\xi''(q)\int_{q}^1\!\zeta_{m}(s)\,ds\Bigr)_{i\leq N}.
\label{eqTAPstates12}
\end{align}
\end{thm}

\begin{remark}[Generalized TAP equations]\label{LabRmkMV}\rm
Let us show how (\ref{eqTAPstates1}) and (\ref{eqTAPstates12}) lead to the generalized TAP equations. If we combine (\ref{eqTAPstates1}) and (\ref{eqTAPstates12}), we can write
$$
(\nabla H_N(m))_i-m_i \xi''(q)\int_{q}^1\!\zeta_{m}(s)\,ds=\oPsi(q,m_i,\zeta_m).
$$
If we plug both sides into $\partial_x\Phi_{\zeta_{m}}(q,\cdot)$ and recall the definition of $\oPsi$, we get
\begin{equation}
\partial_x\Phi_{\zeta_m}\Bigl(q,(\nabla H_N(m))_i-m_i \xi''(q)\int_q^1\! \zeta_m(s)\, ds\Bigr)
= m_i.
\label{exactTAP}
\end{equation}
These are the TAP equations for generalized TAP states. To compare them with classical equations, we can use that TAP states with $\frac{1}{N}\|m\|^2=q\in \supp(\zeta_*)$ (or close to the support) must have the order parameter $\zeta_m\approx \zeta_*$, which yields the approximate TAP equations,
\begin{equation}
\partial_x\Phi_{\zeta_*}\Bigl(q,(\nabla H_N(m))_i-m_i \xi''(q)\int_q^1\! \zeta_*(s)\, ds\Bigr)
\approx m_i.
\label{eqMVas}
\end{equation}
We will discuss the replica symmetric case of TAP correction in the next section in much more detail, but notice that, when $\zeta_*=1$ for $s\in [q,1]$, (\ref{eqMVas}) reduces to
\begin{equation}
\tanh\Bigl((\nabla H_N(m))_i-m_i \xi''(q)(1-q)\Bigr)
\approx m_i.
\label{eqMVasTAP}
\end{equation}
For the SK model, these are the classical TAP equations, which also  appeared in the physics literature for the pure $p$-spin model in \cite{Rieger} (see also \cite{CLR}). The equations (\ref{exactTAP}), (\ref{eqMVas}) are, thus, an extension of the classical TAP equations to all generalized TAP states, for all mixed $p$-spin models.

Let us remark that (\ref{exactTAP}), (\ref{eqMVas}) above are self-consistent TAP equations in the sense that they relate state magnetization to itself. They are different from the M\'ezard-Virasoro equations for the ancestor states derived rigorously in Theorem 1.4 in \cite{AufJag18}, which relate magnetization to the local field, although for the $2$-spin SK model considered in M\'ezard-Virasoro \cite{M3} they happen to coincide with (\ref{eqMVas}). The main reason is because, in those results, in place of the term $\nabla H_N(m)$ in (\ref{exactTAP}) one has the cavity field process at an ancestor state $m$, which represents the average of $\nabla H_N(\bs)$ over many mutually orthogonal directions around $m$. For example, if the entire system is in a pure state, the cavity field process corresponds to $\la\nabla H_N(\bs)\ra$, which coincides with $\nabla H_N(\la\sigma\ra)=\nabla H_N(m)$ only for the pure $2$-spin model, when the gradient $\nabla H_N(\bs)$ does not include interaction terms. This is why in those results one has the term $m_i \int_q^1\! \xi''(s)\zeta_*(s)\, ds$ instead of $m_i \xi''(q)\int_q^1\! \zeta_*(s)\, ds$ in (\ref{eqMVas}), and correspondingly $m_i(\xi'(1)-\xi'(q))$ instead of $m_i \xi''(q)(1-q)$ in (\ref{eqMVasTAP}), which appears in self-consistent TAP equations, as e.g. in \cite{Rieger}.
\qed
\end{remark}

\begin{remark}[Spherical gradient]\rm
The formula (\ref{eqTAPstates12}) implies that in the spherical directions,
\begin{equation}
\nabla \nTAP(\mu_m)\cdot v = -\frac{1}{N}\bigl(\Psi(q,m_i,\zeta_m)\bigr)_{i\leq N}\cdot v \mbox{ for all } v\perp m.
\label{eqTAPstates12Sph}
\end{equation}
There is a physical argument for the formula (\ref{eqTAPstates12Sph}) to be satisfied by the ancestor states $m$ in the Parisi ansatz. We will see below that $(\Psi(q,m_i,\zeta_m))$ has the meaning of the unique external field that forces the model on the narrow band $B(m,\eps)$ (with its own random external field removed) to have many orthogonal pure states relative to $m$. On the other hand, $\nabla H_N(m)$ is the external field of the original model restricted to the band and, when $m$ is an ancestor state, we know that there exist many orthogonal states on the band. This suggests that $\nabla H_N(m)\cdot v=(\Psi(q,m_i,\zeta_m))\cdot v$ for such $m$ and $v\perp m$, which agrees with (\ref{eqTAPstates1}) and (\ref{eqTAPstates12Sph}).
\qed
\end{remark}

\begin{remark}[Model with external field]\label{rmkMEF}\rm
One can include an external field to the original model and consider the model with the Hamiltonian
\begin{equation}
H_N^{\mathrm{ef}}(\bs)=H_N(\bs)+h\sum_{i=1}^{N}\sigma_i.
\end{equation}
It will be clear from the discussion below that the TAP correction $\nTAP(\mu_m)$ is the same whether or not the external field is present. In fact, this will be, in some sense, a big part of the motivation for our definition of the generalized TAP correction. The only difference will be at the level of the TAP representation (\ref{eq:GenTAP}), which will become
\begin{equation}
\label{eq:GenTAPef}
F_N \approx \max_{\frac{\|m\|^2}{N}= q}\Bigl( \frac{H_N(m)}{N} + h\sum_{i=1}^{N}m_i+\nTAP(\mu_m)\Bigr),
\end{equation}
for $q$ in the support of the Parisi measure of the model with external field $h$. For simplicity of notation, we will work without the external field, because only trivial modifications are necessary in the case with external field.
\qed
\end{remark}

In the next section, we will give an outline of the main ideas in the proof and state further results. For example, we will show that the replica symmetric case of the above generalized TAP correction (when the minimizer in (\ref{eqTAPfirst}) equals $\delta_0$) reduced to the classical TAP correction and give a precise characterization for when that happens. We will show how this implies the necessity of Plefka's condition and, in particular, the generalized TAP correction is not always replica symmetric on the points corresponding to the Edwards-Anderson parameter (see Remark \ref{LabRmkPlefka}).

\section{General outline and further results}

\subsection{Utilizing many orthogonal directions.}  

Let us explain the main ideas from \cite{SubagFEL} that allow us to make the uniform claim (\ref{eq:TAPunifconv}) and at the same time compute things explicitly despite the dependence on large $n$. This will also allow us to introduce some necessary definitions and notation. We will see that for large $n$ and small $\eps$ and $\delta$, the following properties hold:
\begin{enumerate}
\item[(a)] with high probability, $\nTAP_{N,n}(m)$ in (\ref{eq:TAPn}) is close to its expectation, uniformly over all $m$ in (\ref{Mconstr});

\item[(b)] adding or removing an external field term in $\e \nTAP_{N,n}(m)$ has a negligible effect.
\end{enumerate}
Let us explain what these properties mean and sketch why they hold. First of all, let us compute the covariance of the process $H_{N}(\bs)-H_N(m)$ for $\bs$ in the narrow band $B(m,\eps),$
\begin{align}
&\frac{1}{N}\e\bigl(H_{N}(\bs^{i})-H_N(m)\bigr)\bigl(H_{N}(\bs^{j})-H_N(m)\bigr)
\label{eqConBm}
\\
&=
\xi(R(\bs^i,\bs^j))-\xi(R(\bs^i,m))-\xi(R(\bs^j,m))+\xi(R(m,m)).
\nonumber
\end{align}
For $(\bs^1,\ldots,\bs^n)\in B_n(m,\eps,\delta)$, the covariance for $i\not =j$ is small, by the definition of $B(m,\eps)$ and $B_n(m,\eps,\delta).$ Therefore, the variance of $\sum_{i=1}^{n}[H_{N}(\bs^{i})-H_N(m)]$ in (\ref{eq:TAPn}) is roughly of the order $nN$ and, by the Gaussian concentration, the fluctuations of $\nTAP_{N,n}(m)$ are of order $(nN)^{-1/2}$ (Lemma \ref{lem:concentration} contains a precise statement). The extra factor $n^{-1/2}$ with large $n$ will allow us to discretize and apply a union bound uniformly over $m$, implying the first property (a).

To explain the second property (b), it is convenient to think of $H_{N}(\bs)-H_N(m)$ as a new mixed $p$-spin model on the narrow band, as follows. If $q=R(m,m)=\|m\|^2/N$ then, for $\bs^1,\bs^2\in B(m,\eps),$
\begin{equation}
|R(\tbs^j,m)| = |R(\bs^j,m)- q|<\eps
\end{equation}
and
\begin{equation}
|R(\bs^1,\bs^2)- (R(\tbs^1,\tbs^2)+q)| 
=|R(\tbs^1,m)+R(\tbs^2,m)|< 2\eps.
\end{equation}
Therefore, up to the error of order $O(\eps)$, the covariance in (\ref{eqConBm}) is approximated by
\begin{equation}
\xi(R(\tbs^1,\tbs^2) + q) - \xi(q) = \hat{\xi}_q(R(\tbs^1,\tbs^2)),
\end{equation}
where
\begin{equation}
\hat{\xi}_q(s) := \xi(s+q)-\xi(q)=\sum_{k\geq 1} \beta_k(q)^2 s^k
\label{eqXiq}
\end{equation}
and where
\begin{equation}
\beta_k(q)^2 = \sum_{p\geq k}\binom{p}{k}\beta_p^2 q^{p-k}.
\label{eqBkq}
\end{equation}
As a result, if $\hat{H}_N^{m}(\tbs)$ is the mixed $p$-spin Hamiltonian indexed by $\tbs=\bs-m$ for $\bs\in B(m,\eps)$ with the covariance given by 
\begin{equation}
\e \hat{H}_N^{m}(\tbs^1) \hat{H}_N^{m}(\tbs^2)= N \hat{\xi}_q(R(\tbs^1,\tbs^2)),
\end{equation}
one can show that (see Lemma \ref{LemHtoHm} below)
\begin{equation}
\label{eq:ETAPn}
\e \nTAP_{N,n}(m,\eps,\delta) =  \frac{1}{nN} \e \log\sum_{B_{n}(m,\eps,\delta)}e^{\sum_{\ell=1}^{n}\hat{H}_N^{m}(\tbs^{\ell})}+O\Bigl(\frac{\eps}{\sqrt{q}}\Bigr)
\end{equation}
uniformly over $n$, if $q=R(m,m)=\|m\|^2/N>0.$ The case when $q$ is small and $m$ is close to zero will be handled slightly differently, by working with the original model with the external field removed, without any recentering. 

The external field term that we mentioned in the property (b) is present in the model $\hat{H}_N^{m}(\tbs)$, because $\beta_1(q)\not = 0$ in (\ref{eqXiq}). Define the function $\xi_q$ similarly to \eqref{eqXiq}, only with the summation starting from $k=2$,
\begin{equation}
{\xi}_q(s) := \xi(s+q)-\xi(q)-\xi'(q)s = \sum_{k\geq 2} \beta_k(q)^2 s^k,
\label{eqXiq11}
\end{equation}
and let $H_{N}^{m}(\tbs)$ be the Hamiltonian with the covariance 
\begin{equation}
\e H_{N}^{m}(\tbs^1)H_{N}^{m}(\tbs^2)= N\xi_q(R(\tbs^1,\tbs^2)),
\end{equation}
with $q=\|m\|^2/N.$ Then, in distribution,
\[
\hat{H}_N^{m}(\tbs) = H_{N}^{m}(\tbs) + \beta_1(q) \tbs\cdot g,
\]
where $g$ is a standard Gaussian vector. Hence, using the pairwise near-orthogonality of $(\tbs^1,\ldots,\tbs^{n})$ for $(\bs^1,\ldots,\bs^{n})\in B_{n}(m,\eps,\delta)$, namely 
$$
\sup_{B_{n}(m,\eps,\delta)}    \Big\| \frac{1}{n\sqrt N}\sum_{\ell=1}^{n}  \tbs^\ell \Big\| 
\leq \Bigl(\frac{1}{n}+\delta\Bigr)^{1/2},
$$
we get that
\begin{align*}
	& \frac{1}{n N}\Big| \e \log\sum_{B_{n}(m,\eps,\delta)}e^{\sum_{\ell=1}^{n}\hat{H}_N^{m}(\tbs^{\ell})} 
	-\e \log\sum_{B_{n}(m,\eps,\delta)}e^{\sum_{\ell=1}^{n}H_{N}^{m}(\tbs^{\ell})}\Big|
\leq \beta_1(q)\Bigl(\frac{1}{n}+\delta\Bigr)^{1/2}.
\end{align*}
Together with (\ref{eq:ETAPn}), this shows that
\begin{equation}
\label{eq:TAPmAg}
\e \nTAP_{N,n}(m,\eps,\delta) = \e F_{N,n}(m,\eps,\delta) 
+O\Bigl(\frac{\eps}{\sqrt{q}}+\Bigl(\frac{1}{n}+\delta\Bigr)^{1/2}\Bigr),
\end{equation}
where we introduce the notation
\begin{equation}
F_{N,n}(m,\eps,\delta) := 
\frac{1}{nN} \log\sum_{B_{n}(m,\eps,\delta)}e^{\sum_{\ell=1}^{n}H_{N}^{m}(\tbs^{\ell})}
\end{equation}
for the free energy in the replicated band, with the external field removed. 

By the same argument, we may also add a deterministic external field term $h=(h_i)_{i\leq N}$, as long as we keep $\|h\|/\sqrt{N}$ bounded. Namely, if we define
\begin{equation}
\label{eq:ETAPn0Is}
F_{N,n}^h(m,\eps,\delta) := \frac{1}{nN}  \log\sum_{B_n(m,\eps,\delta)}e^{\sum_{\ell=1}^{n}\bigl[H_{N}^{m}(\tbs^{\ell})+\sum_{i=1}^N h_i \tilde\sigma_i^\ell\bigr]}
\end{equation}
then
\begin{equation}
\bigl|F_{N,n}(m,\eps,\delta)-F_{N,n}^h(m,\eps,\delta)\bigr|
\leq \frac{\|h\|}{\sqrt{N}}\Bigl(\frac{1}{n}+\delta\Bigr)^{1/2}.
\label{eqFtoFh}
\end{equation}
In other words, 
$$
\e \nTAP_{N,n}(m,\eps,\delta)\approx \e F_{N,n}^h(m,\eps,\delta),
$$ 
when $\eps,\delta$ are small, $n$ is large, and $\|h\|/\sqrt{N}$ stays bounded. As with $ \nTAP_{N,n}(m)$, for simplicity of notation, we will  often omit $\eps$ and $\delta$ and write
\begin{equation}
F_{N,n}^h(m):= F_{N,n}^h(m,\eps)= F_{N,n}^h(m,\eps,\delta).
\end{equation}
Let us sketch how the properties (a) and (b) lead to an explicit calculation of the generalized correction term $\nTAP(\mu_m)$, and also contrast what happens in the spherical models vs. models with Ising spins.

\subsection{Spherical vs. Ising spin models.}
Let us start with an overview of the spherical model that was considered in \cite{SubagFEL}, where of course $\Sigma_N$ should be replaced by $\sqrt{N}S^{N-1}$ and the sums over configurations should be replaced by integrals. Using the above heuristics of introducing many orthogonal constraints (or infinitary nature of the tree of states), we get that, uniformly over $m$,
\begin{equation}
\nTAP_{N,n}(m)\approx \e \nTAP_{N,n}(m)\approx \e F_{N,n}(m).
\label{eqInfSum}
\end{equation}
Since a narrow band on the sphere looks the same for all $m$ with $\frac{1}{N}\|m\|^2=q,$ the right hand side depends on $m$ only through $q$. We can see that the constraints $(\bs^1,\ldots,\bs^n)\in B(m,\eps,\delta)$ in
\begin{equation}
\label{eq:ETAPn0}
\e F_{N,n}(m) = \frac{1}{nN} \e \log\int_{B_{n}(m,\eps,\delta)}e^{\sum_{\ell=1}^{n}H_{N}^{m}(\tbs^{\ell})} d\bs^1\ldots d\bs^n
\end{equation}
can be expressed by saying that 
$$
\frac{1}{N}\|\tbs^i\|^2 = \frac{1}{N}\|\bs^i-m\|^2\approx 1-q
$$ 
and all $\tbs^1,\ldots,\tbs^n$ are almost orthogonal to each other, $R(\tbs^i,\tbs^j)\approx 0$ for $i\not =j.$ Moreover, the narrow band is, essentially, a sphere in one dimension less, so we can think of $H_{N}^{m}(\tbs)$ as a new spherical model. The fact that the external field has been removed implies a well-known fact that zero is in the support of the Parisi measure of this new spherical model and, in particular, the overlap constraints $R(\tbs^i,\tbs^j)\approx 0$ can not have a free energy cost. In other words, one can show that
\begin{equation}
\e F_{N,n}(m) \approx \e F_{N,1}(m),
\label{eqRemn}
\end{equation}
which is the free energy of a spherical model that can be written as a spherical analogue of the Parisi formula, the Crisanti-Sommers formula. Thus, the generalized TAP correction has a particularly simple form in the spherical models, and this has important consequences, as was demonstrated in \cite{SubagFEL}. 

In the Ising spin models, the situation is quite different. First of all, the narrow band depends on $m$ in a complicated way and the constraint $\bs\in B(m,\eps)$ can be viewed as a constraint on $\tbs=\bs-m$ of the form
\begin{equation}
\begin{aligned}
\tbs \in \tilde B(m,\eps) &:= B(m,\eps)-m\\
				&\,\approx (\Sigma_N -m) \bigcap\, \bigl\{\frac{1}{N}\|\tbs\|^2\approx 1-q \bigr\}.
\end{aligned}
\label{eqBandIsing}
\end{equation}
So, the first question is: for $n=1$, can we compute the analogue of the Parisi formula for the free energy $\e F_{N,1}(m)$ on the narrow band? The answer is yes, but this will require some work. However, the bigger issue is that, even if we can compute this free energy, removing the external field term will not result in zero being in the support of the Parisi measure, because of the inherent asymmetry of the band (\ref{eqBandIsing}), and will not allow us to make the step (\ref{eqRemn}). The solution to this will be to \emph{add a new external field} to balance out the asymmetry of the band. In other words, using (\ref{eqFtoFh}), we will introduce an external field $h=(h_i)$ at the step (\ref{eqInfSum}),
\begin{equation}
\e \nTAP_{N,n}(m)\approx \e F_{N,n}(m)\approx \e F_{N,n}^h(m),
\label{eqInfSumI}
\end{equation}
and, with the right choice of $h$, we will show that zero is in the support of the Parisi measure of the model on the band and, therefore,
\begin{equation}
\e F_{N,n}^h(m) \approx \e F_{N,1}^h(m).
\label{eqRemnIs}
\end{equation}
The analogue of the Parisi formula for the right hand side will be our TAP correction $\TAP(\mu_m).$ The ideas behind finding the right choice of $h$ will  be explained below.

\begin{remark}\rm
Notice that the functional $a\to \Lambda_{\zeta}(q,a)$ in (\ref{eqCCPar}) is even and, therefore, the functional in (\ref{eqTAPfirstZeta}) has the symmetry 
$$
\nTAP(\mu_m,\zeta)=\nTAP(\mu_{|m|},\zeta),
$$
where
$
\mu_{|m|}=\frac{1}{N}\sum_{i\leq N}\delta_{|m_i|}.
$
Under the transformation $\sigma_i\to\mathrm{sgn}(m_i)\sigma_i$, the overlap between two configurations does not change, while the covariance of all the Hamiltonians as well as definition of the bands depend only on the overlaps. Furthermore, throughout the paper, we will always work with external fields of the form $h_i=\mathrm{sgn}(m_i)\ef(|m_i|)$ that depend on the coordinates $m_i$ in an anti-symmetric fashion, which means that the external field $h_i\tilde\sigma_i$ will also be invariant under this transformation. Because of this, from now on, we can and will assume that 
\begin{equation}
m\in [0,1]^N,
\end{equation}
that is, all the coordinates $m_i\geq 0.$ In particular, we assume that 
\begin{equation}
\mu_m=\frac1N\sum_{i\leq N}\delta_{m_i}\in \MM_{0,1}.
\label{eqEmpMu}
\end{equation} 
For the rest of the paper, we will work with this definition of $\MM_{0,1}$.
\end{remark}

\subsection{Parisi formula on the band.}
Let us now state the analogue of the Parisi formula on the band $B(m,\eps)$ for $m\in [0,1]^N$ with the general external fields of the form $h_i=\ef(m_i),$ for $\ef\in C([0,1])$, which will be sufficient for our purposes.  Since the self-overlap of the configurations $\tbs=\bs-m$ on the narrow band is close to $1-q$, it will be natural to work with the space $\mathcal{M}_{0,1-q}$ of all distributions on $[0,1-q]$. Recall the function $\xi_q$ in (\ref{eqXiq11}) and, for each $a\in[0,1]$ and $\zeta\in \mathcal{M}_{0,1-q},$ denote by $\Phi_{a,\zeta}(t,x)$ the solution of the Parisi PDE on $[0,1-q]\times \Reals,$
\begin{equation}
\partial_t \Phi_{a,\zeta} = -\frac{\xi_q''(t)}{2}\Bigl(
\partial_{xx} \Phi_{a,\zeta} + \zeta(t)\bigl(\partial_x \Phi_{a,\zeta}\bigr)^2
\Bigr)
\label{ParisiPDE}
\end{equation}
with the boundary condition 
\begin{equation}
\Phi_{a,\zeta}(1-q,x)=\log 2-ax +\log\ch x =\log\sum_{\sigma=\pm 1}e^{(\sigma-a)x}.
\label{ParisiPDEBoundary}
\end{equation}
Let us define the function $\Psi(a,\zeta)$ by
\begin{equation}
\partial_x \Phi_{a,\zeta}\bigl(0,\Psi(a,\zeta)\bigr) = 0,\,\, a\in [0,1).
\label{eqDefPsief}
\end{equation}
Note that $\Psi(a,\zeta)$ is well-defined as $\partial_x \Phi_{a,\zeta}\bigl(0,\cdot\bigr)$ is strictly increasing with $\partial_x \Phi_{a,\zeta}\bigl(0,\pm \infty\bigr)=\pm 1-a$. Also, note that $\Phi_{a,\zeta}$ depends on $q$, but we will keep this dependence implicit for simplicity of notation.

For $\mu\in \MM_{0,1}$, $\ef\in C([0,1])$, and $(\lambda,\zeta)\in \mathbb{R}\times\mathcal{M}_{0,1-q},$ let
\begin{equation}
\mathcal{P}_{\mu}^\ef(\lambda,\zeta)=\int\! \Phi_{a,\zeta}(0,\lambda a+\ef(a))\,d\mu(a)-\frac{1}{2}\int_0^{1-q}\! s\xi_q''(s)\zeta(s)\,ds.
\label{eqParisi1ab}
\end{equation}
Set 
 \begin{equation}\label{TAP}
\PP^{\ef}_\mu:=\inf_{\lambda\in\Reals,\,\zeta\in\mathcal{M}_{0,1-q}}\mathcal{P}_{\mu}^\ef(\lambda,\zeta).
\end{equation}
This will be the Parisi formula for the limit of $\e F_{N,1}^{h}(m)$ when $h_i=\ef(m_i)$ and the empirical measure $\mu_{m}$ in (\ref{eqEmpMu}) converges weakly to $\mu$. 
\begin{thm}[Parisi formula on the band]\label{thm1}
Assume that $\ef\in C([0,1])$ and $\mu\in \MM_{0,1}$. Let $m^N\in [0,1]^N$ be any sequence so that $\mu_{m^N}$ in (\ref{eqEmpMu}) converges to $\mu$ weakly and let $h_i=\ef(m_i)$ for $i\leq N$. Then
\begin{equation}
\lim_{N\to\infty}\e F_{N,1}^h(m^N,\eps_N)= \PP^\ef_\mu,
\label{eqPBmin}
\end{equation}
provided that $\eps_N$ goes to zero slowly enough.
\end{thm}

\begin{remark}\label{add:remark1}\rm
	In Proposition \ref{eqPropATO} below, we will establish a connection between $\Phi_{a,\zeta}$ and $\Phi_\zeta$ (recall \eqref{ParisiPDEOrig}), which states that these PDE solutions are essentially the same up to a transformation. An important consequence of this connection is that the function $\Psi(a,\zeta)$ in (\ref{eqDefPsief}) coincides with the one defined in (\ref{eqDefPsiefF}) up to a shift. With a properly chosen external field (see \eqref{add:eq12} below), this allows us to simplify the above Parisi formula on the band  and naturally gives rise to the the desired TAP correction (defined in \eqref{eqTAPfirstZeta}). 
\end{remark}

\subsection{Finding effective external field}

Next, in order to obtain \eqref{eqRemnIs}, we aim to find an external field such that the model with the Hamiltonian $H_{N}^{m}(\tbs)+\sum_{i=1}^N h_i \tilde\sigma_i$ on the band $B(m,\eps)$ has zero in the support of its Parisi measure. Given a function $\ef\in C([0,1]),$ suppose that $(\lambda^*,\zeta^*)$ is the minimizer in the Parisi formula $(\ref{TAP}).$ In the proof, we will deal with cases when $\lambda^*=+\infty$, but for the purpose of this discussion let us assume that it is finite. For zero to be in the support of $\zeta^*$, it is necessary and sufficient that 
\begin{equation}
\partial_x \Phi_{a,\zeta^*}(0,\lambda^*a + \ef(a)) = 0 
\label{eqLefBG}
\end{equation}
for all $a$ in the support of $\mu$. We will not prove this (standard) statement, because we will only need an implication in one direction (that will appear in Theorem~\ref{lem:TAPnew} $(v)$ below), but, again, let us use it as a motivation for what we do next. If we include the term $\lambda^*a$ into the field $\ef(a)$ then the field must satisfy $\partial_x \Phi_{a,\zeta^*}(0,\ef(a)) = 0$ for all $a$ in the support of $\mu$. If we recall the definition (\ref{eqDefPsief}) above, this means that our only hope to force zero in the support of the Parisi measure is to restrict our attention to external fields generated by functions of the form 
\begin{align}\label{add:eq12}
\ef_\zeta(a):=\Psi(a,\zeta)
\end{align} 
for $\zeta\in \mathcal{M}_{0,1-q},$ and, moreover, with such choice of $\ef_\zeta,$ the Parisi measure $\zeta^*$ must coincide with $\zeta,$ so that (\ref{eqLefBG}) holds \begin{equation}
\partial_x \Phi_{a,\zeta}(0,\ef_\zeta(a)) = 0.
\end{equation}
Actually, this will automatically force $\lambda^*=0$, so this equation matches (\ref{eqLefBG}). This raises two issues. 
\begin{enumerate}
\item
First of all, we will see that $\Psi(a,\zeta)$ goes to $+\infty$ as $a$ approaches $1$ uniformly over $\zeta$ (see Lemma \ref{ex:lem2}) while $\ef$ in the Parisi formula in Theorem \ref{thm1} was continuous on $[0,1]$, so we can not apply the Parisi formula with $v=v_\zeta$ when $1\in\supp(\mu).$
\item The minimizer $\zeta^*$ depends on the external field $v_\zeta$, and we want it to coincide with $\zeta$ in the definition of $v_\zeta.$ How can we find such `good' choice of $\zeta$?
\end{enumerate}

The second issue can be solved via an implicit fixed point problem, using the Schauder fixed point theorem; however, we will give a more direct and explicit way to find such good $\zeta$. The first issue will be handled by an approximation argument, which, at a crucial step, will allow us to work with measures $\mu$ with the support separated from $1.$ Because these two issues present very different obstacles, it will be convenient to work with two intermediate definitions of the functional $\nTAP(\mu)$ in (\ref{eqTAPfirst}), which will be shown to coincide with it. 

Let us consider the following growth condition on functions $\ef\colon[0,1]\to\Reals$,
\begin{equation}
v(a)\leq c_1+c_2\tanh^{-1}(a) \mbox{ for } a\in [0,1),
\label{vGrowthCond}
\end{equation}
for some absolute constants $c_1,c_2>0$ that depend only on the model $\xi$ and can be found explicitly from the proof of Lemma \ref{ex:lem2} below. Set
\begin{equation}
\label{eq:V}
\begin{aligned}
&V=\big\{\ef\in C([0,1]):\, \ef\geq 0, \ef \mbox{ is non-decreasing, and (\ref{vGrowthCond}) holds}\big\},
\\
&\oV=\big\{\ef\in C([0,1)):\, \ef\geq 0, \ef \mbox{ is non-decreasing, and (\ref{vGrowthCond}) holds}\big\}.
\end{aligned}
\end{equation}
Note that the only difference between $V$ and $\oV$ is on the right boundary of $v.$
We will prove that the functions $v_\zeta(a)=\Psi(a,\zeta)$ belong to $\oV$ (see Lemma \ref{add:lem}), which is the real reason behind these definitions. Let us recall the functional $\PP_\mu^\ef(\lambda,\zeta)$ defined in (\ref{eqParisi1ab}) for $v\in V$ and define a new functional $\oP_\mu^\ef(\lambda,\zeta)$ for $v\in \oV,$
\begin{equation}
\label{add:eq10}
\begin{aligned}
&
\PP_\mu^\ef(\lambda,\zeta):=\int_{[0,1]}\! \Phi_{a,\zeta}(0,\lambda a+\ef(a))\,d\mu(a)-\frac{1}{2}\int_0^{1-q}\! s\xi_q''(s)\zeta(s)\,ds,\,\mbox{ for } v\in V,
\\
&
\oP_\mu^\ef(\lambda,\zeta):=\int_{[0,1)}\Phi_{a,\zeta}(0,\lambda a+\ef(a))d\mu(a)-\frac{1}{2}\int_0^{1-q} s\xi_q''(s)\zeta(s)ds,\, \mbox{ for }v\in \oV.
\end{aligned}
\end{equation}
For $\mu\in \MM_{0,1}$ such that $\int a^2d\mu(a)=q,$ we define 
\begin{equation}
\label{eq:TAPmu}
\begin{aligned}
\TAP(\mu) &:= \inf_{\ef\in V,\,\zeta\in \mathcal{M}_{0,1-q}} \PP_\mu^\ef(0,\zeta),
\\
\oTAP(\mu) &:=\inf_{\ef\in \oV,\,\zeta\in \mathcal{M}_{0,1-q}}\oP_\mu^\ef(0,\zeta).
\end{aligned}
\end{equation}
In the case that $\supp(\mu)\subset [0,1),$ it is evident that these coincide, and we will in fact show that they always coincide. The first representation in (\ref{eq:TAPmu}) will be convenient when working with the Parisi formula on the band, and the second representation will be convenient for analytical reasons and because it allows us to find a good choice of $\zeta$ directly. We summarize all these properties in the following theorem. 

We define by 
\begin{equation}
\theta_q \zeta(t) = \zeta(t+q) \mbox{ for } t\in[0,1-q]
\label{eqShiftOper}
\end{equation}
the shift operator $\theta_q\colon \mathcal{M}_{0,1}\to \mathcal{M}_{0,1-q}$.
\begin{thm}\label{lem:TAPnew} The following statements hold:
	\begin{itemize}
		\item[$(i)$] If $\mu\in \MM_{0,1}$ and $q=\int\!x^2\,d\mu(a)$ then 
		$$
		\nTAP(\mu,\zeta') = \oP_\mu^{\ef_\zeta}(0,\zeta)
		$$
for all $\zeta'\in \mathcal{M}_{0,1}$ and $\zeta = \theta_q\zeta'$, and
		\begin{equation}
		\nTAP(\mu)=\TAP(\mu)=\oTAP(\mu)=\inf_{\zeta\in \mathcal{M}_{0,1-q}}\oP_\mu^{\ef_\zeta}(0,\zeta).
		\label{eqTAP4repr}
		\end{equation}
		\item[$(ii)$] The functional $\nTAP(\mu)$ is continuous on $(\MM_{0,1},d_1)$.
		\item[$(iii)$]  The right-hand side in (\ref{eqTAP4repr}) has a unique minimizer $\zeta_0$. 
		\item[$(iv)$] The minimizer $\zeta_0$ in $(iii)$ satisfies
		$$
		\oP_\mu^{\ef_{\zeta_0}}(0,\zeta_0)=\inf_{\lambda\in \mathbb{R},\zeta\in\mathcal{M}_{0,1-q}}\oP_\mu^{\ef_{\zeta_0}}(\lambda,\zeta)\,\,(=:\oP_\mu^{\ef_{\zeta_0}}).
		$$
		\item[$(v)$] The minimizer $\zeta_0$ in $(iii)$ has zero in its support, $0\in\supp(\zeta_0)$.		
		\item[$(vi)$] If there exists some $\zeta_1\in \mathcal{M}_{0,1-q}$ such that $(0,\zeta_1)$ is a minimizer of $$\inf_{\lambda\in \mathbb{R},\,\zeta\in \mathcal{M}_{0,1-q}}\oP_\mu^{\ef_{\zeta_1}}(\lambda,\zeta),$$ then $\zeta_0=\zeta_1.$
	\end{itemize}
\end{thm}

Notice that if $\mu=\delta_1$ then $q=1$, so $\oP_\mu^\ef(\lambda,\zeta)\equiv 0,$ $\mathcal{M}_{0,1-q}=\{\delta_0\}$, and the claims $(iii)$-$(vi)$ hold trivially. If $\supp(\mu)\subseteq \{0,1\}$ then $\oP_\mu^\ef(\lambda,\zeta)$ does not depend on $\lambda$. In Remark \ref{rmk1} below, we will see that in all other cases the minimizer $(\lambda_\mu^\ef,\zeta_\mu^\ef)$ exists and is unique. For convenience, we will use the convention that  $(\lambda_\mu^\ef,\zeta_\mu^\ef)=(0,\delta_0)$ when $\mu=\delta_1,$ and $(\lambda_\mu^\ef,\zeta_\mu^\ef)=(0,\zeta_\mu^\ef)$ when $\supp(\mu)\subseteq \{0,1\}$.

\subsection{\label{subsec:classical}Classical TAP correction.}
Finally, let us explain how the generalized TAP correction defined above leads to the classical TAP correction and the Plefka condition in \cite{Plefka,TAP}. In this section, we will consider non-trivial case $\mu\neq \delta_1,$ and will describe when the minimizer $\zeta_0$ in Theorem \ref{lem:TAPnew} is \emph{replica symmetric}, $\zeta_0=\delta_0.$ Let us denote the corresponding replica symmetric external field by  
\begin{align}\label{rs:eq1}
\ef_{\RS}(a):=\ef_{\delta_0}(a)=\Psi(a,\delta_0),\,\,\forall a\in [0,1).
\end{align} 
Using the Cole-Hopf transformation yields
\begin{align}
\begin{split}\label{add:eq81}
\Phi_{a,\delta_0}(s,x)&=\log 2+\log \E\ch(a+t(s)g)e^{-a(x+t(s)g)}\\
&=\frac{1+a^2}{2}t^2(s)-ax+\log 2\ch(x-at^2(s)),
\end{split}
\end{align}
where $g$ is a standard Gaussian random variable and, for $s\in [0,1-q]$, we denoted
$$
t(s):=(\xi_q'(1-q)-\xi_q'(s))^{1/2}.
$$
From this formula, we can express $\ef_{\RS}$ explicitly as
\begin{align}\label{add:eq8}
\ef_{\RS}(a):=\th^{-1}(a)+a\xi_q'(1-q),\,\,\forall a\in [0,1).
\end{align}
By Theorem \ref{lem:TAPnew} $(iv)$, if $\delta_0$ is the minimizer of the functional in (\ref{eqTAPfirst}), we must have
\begin{equation}
\oP_\mu^{\ef_{\RS}}(0,\delta_0)=\inf_{\lambda,\zeta\in\mathcal{M}_{0,1-q}}\oP_\mu^{\ef_{\RS}}(\lambda,\zeta).
\label{eqRSvC}
\end{equation}
On the other hand, by Theorem \ref{lem:TAPnew} $(vi)$, if this holds then $\delta_0$ is the minimizer. Hence, in order to characterize the replica symmetric TAP correction, it is enough to describe when (\ref{eqRSvC}) holds. Recall (\ref{add:eq81}) and define
\begin{align}
\gamma_\mu(s)
&:=\int_{[0,1)}\!\E\Bigl[\partial_x\Phi_{a,\delta_0}(s,g_{a}(s))^2
\exp\bigl(\Phi_{a,\delta_0}(s,g_{a}(s))-\Phi_{a,\delta_0}(0,g_{a}(0))\bigr)\Bigr]\,d\mu(a),
\end{align}
where $g_{a}(s)=\ef_{\RS}(a)+\xi_q'(s)^{1/2}g.$ For $s\in [0,1]$, define an auxiliary function $\Gamma_\mu$ by
\begin{align}
\label{eq:Gamma_mu}
\Gamma_\mu(s)=\int_0^s\xi_q''(r)(\gamma_\mu(r)-r)dr.
\end{align}
This function is derived through the directional derivative of the functional $\oP_\mu^{\ef_{\RS}}(0,\cdot)$ (see the derivation in \eqref{add:eq9} below). Recall $I(a)$ and $C(q)$ in \eqref{eq:IC}. The following holds.
\begin{prop} \label{add:prop2}
Assume that $\delta_1\neq \mu\in \MM_{0,1}$ and $q=\int\! a^2\,d\mu(a)$. Then $\delta_0$ is the minimizer of (\ref{eqTAPfirst}) (i.e. (\ref{eqRSvC}) holds) if and only if 
	\begin{align}\label{add:prop2:rs}
	\Gamma_\mu(s)\leq 0,\,\,\forall s\in[0,1-q].
	\end{align}
	Furthermore, in this case, $\nTAP(\mu)$ is given by the classical TAP correction,
	\begin{align}\label{add:eq11}
		\nTAP(\mu)=-\int I(a)d\mu(a)+C(q).
	\end{align}
\end{prop}

The above proposition will be used to conclude Corollary \ref{add:cor1} from the general TAP representation of Theorem \ref{thm:GenTAP}. The proposition
 also naturally leads to the so-called Plefka condition in the SK model, that is, $\xi(s)=\beta^2 s^2/2.$ 

\begin{prop}\label{add:prop4}
	Let $\xi(s)=\beta^2s^2/2$. Assume that $\delta_1\neq \mu\in \MM_{0,1}$. If $\delta_0$ is the minimizer of (\ref{eqTAPfirst}), i.e. (\ref{add:prop2:rs}) holds, then the so called \emph{Plefka's condition} holds,
	\begin{align}
	\beta^2\int (1-a^2)^2d\mu(a)\leq 1.
	\end{align}
\end{prop}

\begin{remark}\label{LabRmkPlefka}\rm
If we consider pure SK model with $\xi(s)=\beta^2 s^2/2$ with non-zero external field $h$, one can see that the generalized TAP correction is not always replica symmetric on the sphere $\frac{1}{N}\|m\|^2=q_{\EA},$ where the Edwards-Anderson parameter
$q_{\EA}$ is the largest point in the support of the Parisi measure.  Let us now consider $(\beta,h)$ below the AT line
$$
\beta^2 \e \frac{2}{\cosh^{4}(\beta z\sqrt{q}+h)} \leq 1,
$$
where $q$ is the unique solution of $q=\e \tanh^2(\beta z\sqrt{q}+h).$ It is well-known based on simulations, and in fact proved in some region of parameters in \cite{JT17}, that below the AT line the original model is replica symmetric, and $q_{\EA}=q.$ 
Let us now take $m$ that has $Nq$ coordinates equal to $1$ and $N(1-q)$ coordinates equal to $0$. For such $m,$ Plefka's condition becomes
$$
\beta^2\int (1-a^2)^2d\mu_m(a) = \beta^2(1-q) = \beta^2 \e \frac{1}{\cosh^2(\beta z\sqrt{q}+h)} \leq 1.
$$
Clearly, we can choose $(\beta,h)$ below the AT line such that Plefka's condition is violated, which means that the generalized TAP correction does not always coincide with the classical TAP correction on the sphere $\frac{1}{N}\|m\|^2=q_{\EA}.$
\end{remark}

In Plefka \cite{Plefka}, it was conjectured that in the SK model,
\begin{align*}
\lim_{N\rightarrow\infty}F_N&=\lim_{\varepsilon\downarrow 0}\lim_{N\rightarrow\infty}\sup\Bigl(\frac{H_N(m)}{N}-\int I(a)d\mu_m(a)+C(q)\Bigr),
\end{align*}	
where the supremum is taken over all $m\in [-1,1]^N$ satisfying that $q=\|m\|_2^2/N$ and the Plefka condition,
$
\beta^2\int (1-a^2)^2d\mu_m(a)\leq 1.
$
While Plefka obtained this condition through the consideration of convergence criterion for the series expansion of the SK free energy, we discover the same condition from an analogous study of the so-called Almeida-Thouless line for the Parisi formula on the band with external field $\ef_{\RS}$, namely, it is determined by the second derivative of $\Gamma_\mu$ at $s=0.$ We anticipate that it is not always true that when the Plefka condition is satisfied, \eqref{add:prop2:rs} is automatically valid. In addition, we mention that it looks possible that one can find a $\mu$ such that the Plefka condition is satisfied, but 
$$
\oP_\mu^{\ef_{\RS}}(0,\delta_0)=-\int I(a)d\mu(a)+C(q)
$$ 
is strictly larger than $\nTAP(\mu).$ It is however not clear to us how to compare the corresponding TAP free energy of such $\mu$ with our expression \eqref{eqTAPfirstZeta}, so the validity of Plefka's conjecture remains unclear.

Let us describe the structure of the rest of the paper. 
In the next section we derive the uniform statement in Theorem \ref{thm:TAPcorrection}, conditional on Theorem \ref{lem:TAPnew} and a representation of the limiting  replicated free energy on a band at the level of expectation. Using Theorem \ref{thm:TAPcorrection}, in Section \ref{SecGTAPrepr} we prove the generalized TAP representation in Theorem \ref{thm:GenTAP}.
In Section \ref{sec:TAPtoF}, we reduce the TAP correction to the free energy of the replicated model on the narrow band, and in Section \ref{sec:ParisiOnBand} we prove the Parisi formula for this model stated in Theorem \ref{thm1}. In Section \ref{SecOptimal} we prove the properties of various TAP representations in Theorem \ref{lem:TAPnew}. Section \ref{SecOptimizingEF} contains the key step, which combines the Parisi formula with the choice of the optimal external field from Theorem \ref{lem:TAPnew} to derive the first representation $\TAP(\mu)$ in (\ref{eq:TAPmu}) at the level of expectations. In Section \ref{SecTAPAnc}, we prove Theorem \ref{Thm1label}, and we compute the gradient of $\nTAP(\mu_m)$ in Theorem \ref{ThmGTElab} in Section \ref{SecTAPGenEE}. In Section \ref{SecTAPclass} we prove Corollary \ref{add:cor1} and Propositions \ref{add:prop2},  \ref{add:prop4}, about the classical TAP correction. Finally, in Section \ref{SecAnalytical}, we prove various technical results used throughout the paper.

 \section{Uniform TAP correction}\label{SecOptimizingUnif}
 
 In this section, we will prove our main result in Theorem \ref{thm:TAPcorrection}. Our proof is conditional on the continuity of the mapping $\mu\to\nTAP(\mu)$ that we stated in Theorem \ref{lem:TAPnew} $(ii)$ above, which we will prove in Section \ref{SecOptimal}, and the following two lemmas relating the limit of the expectations $\e \nTAP_{N,n}(m,\eps,\delta)$ to the functional $\nTAP(\mu)$, which we will prove in Section \ref{SecOptimizingEF}. 
 
 \begin{lem}\label{lem:TAP+} 
 	For any sequence $m=m^N\in [0,1]^N$ such that $\mu_{m}\to \mu\in \MM_{0,1}$,
 	\begin{equation}
 	\label{eq:lemTAP+}
 	\inf_{\eps,\delta,n}\limsup_{N\to\infty}\e \nTAP_{N,n}(m,\eps,\delta)
 	\leq\nTAP(\mu).
 	\end{equation}
 \end{lem}
 
 \begin{lem}
 	\label{lem:TAP-} Take any $\eta\in (0,1).$ For any sequence $m=m^N\in[0,1- \eta]^N$ such that $\mu_{m}\to \mu\in \MM_{0,1}$ and $\supp(\mu)\subseteq [0,1- \eta]$, 
 	\begin{equation}
 	\inf_{\eps,\delta,n}\liminf_{N\to\infty}\e \nTAP_{N,n}(m,\eps,\delta)
 	\geq \nTAP(\mu).
 	\label{lem:TAP-:eq2}
 	\end{equation}
 \end{lem}

 To move from the statements about expectations above to the actual random TAP free energies, we will need the following basic, but crucial, concentration result. This is another key way in which  we utilize many orthogonal directions.
 \begin{lem}
 	\label{lem:concentration}
 	For some constant $c_\xi>0$ depending only on $\xi$, 
 	\begin{equation}
 	\label{eq:LipschitzBound}
 	\P\big\{  
 	|\nTAP_{N,n}(m,\eps,\delta)-\E \nTAP_{N,n}(m,\eps,\delta)| > t
 	\big\}
 	<2\exp\left(  -\frac{Nt^2c_\xi}{1/n + \delta+\eps}     \right).
 	\end{equation}
 \end{lem}
 \begin{proof}
 	Note that, up to normalization by $n$, $\nTAP_{N,n}(m,\eps,\delta)$ is the free energy of the Gaussian process $\tilde H_N(\bs^1,\ldots,\bs^n):=\sum_{i=1}^{n}(H_{N}(\bs^{i})-H_N(m))$ on $B_n(m,\eps,\delta)$, which, using \eqref{eqConBm} and \eqref{eq:Bn}, has variance bounded by
 	\[
 	\frac{1}{N}\e \left(\tilde H_N(\bs^1,\ldots,\bs^n)\right)^2  \leq 4n\xi(1) +n(n-1)\xi'(1) (\delta +2\eps) \leq c_\xi(n+n^2\delta+n^2\eps).
 	\]
 	The lemma therefore follows from a canonical Gaussian concentration result, see e.g. \cite[Theorem 1.2]{SKmodel}. 
 \end{proof}
 
 We will also need the following consequence of the above concentration.
 \begin{lem}\label{lem:net2}
 	For any $k\in \mathbb{N}$ there exists $A_N\subseteq (0,1)^N$ with at most $e^{N\log k}$ elements such that letting $\eta=1/k,$ for every $m\in A_N,$
 	\begin{align}\label{lem:net2:eq1}
 		\max_{1\leq i\leq N}m_i\leq 1- \frac{\eta}{2},
 	\end{align}
 	and, for every $m\in [0,1]^N$, there exists $m'\in A_N$ such that 
 	\begin{align}\label{lem:net2:eq3}
 		\max_{i\leq N}|m_i-m_i'| \leq \frac{\eta}{2}. 
 	\end{align}
 	Furthermore, there exists a constant $c_\xi>0$ such that for any $\eps,\delta>0$ and $t,c>0$ satisfying  
 	\begin{align}\label{lem:net2:eq4}
 		\frac{-t^2c_\xi}{1/n+\delta+\eps}+\log\frac{2}{\eta} <-c\,,
 	\end{align}
 	we have that
 	\begin{equation}
 	\mathbb{P}\Big(\max_{m\in A_{N}}\big|\nTAP_{N,n}(m,\eps,\delta)-\E\nTAP_{N,n}(m,\eps,\delta)\big|>t\Big)<2e^{-Nc}.\label{eq:net2:concAN}
 	\end{equation}
 \end{lem}
 \begin{proof}
 	Let $P$ be a regular partition of $[0,1]^N$ using cubes of edge length $\eta$, and define $A_N$ as the collection of all center points of the cubes in $P$.  Of course, \eqref{lem:net2:eq1} and \eqref{lem:net2:eq3} hold, and $A_N$ has $(1/\eta)^N=e^{N\log k}$ elements. Finally, \eqref{eq:net2:concAN} follows from  Lemma \ref{lem:concentration} by a union bound.
 \end{proof}

 \subsection{Proof of the upper bound of \eqref{eq:TAPunifconv}}
 Let $t,\,c>0$ and assume towards the contradiction that, for some $\eps_{N},\,\delta_{N}\to0$ and $n_{N}\to\infty$, on some subsequence in $N$,
 \begin{equation}
 \mathbb{P}\left( \exists m \in [-1,1]^N :\, \nTAP_{N,n_{N}}(m,\eps_{N},\delta_{N})\geq\nTAP(\mu_{m})+t\right) \geq e^{-Nc}.\label{eq:TAPub1}
 \end{equation}
 
 Let us choose $\eta_N\downarrow 0$ in such a way that	
 \begin{align}\label{lem:net2:eq4ag}
 	\frac{-(t/4)^2c_\xi}{1/n_N+\delta_N+\eps_N+3\eta_N}+\log\frac{2}{\eta_N} <-2c\,,
 \end{align}
 Let us take the set $A_N$ in Lemma \ref{lem:net2}, so that
 \begin{equation}
 \mathbb{P}\Big(\max_{m\in A_{N}}\big|\nTAP_{N,n_{N}}(m,\eps_{N}+3\eta_{N},\delta_{N})-\mathbb{E}\nTAP_{N,n_{N}}(m,\eps_{N}+3\eta_{N},\delta_{N})\big|>t/4\Big)<2e^{-2Nc}.
 \label{eq:TAPub2}
 \end{equation}
 Suppose that $m'$ satisfies the inequality of (\ref{eq:TAPub1}), and let $m\in A_{N}$ be such that \eqref{lem:net2:eq3} holds.	 Since $\|m-m'\|_{2}<\sqrt{N}\eta_N$, we have 
 $$
 B(m,\eps_N+3\eta_N)\supseteq B(m',\eps_N)
 $$ 
 and, moreover, $d_{1}(\mu_{m},\mu_{m'})<\eta_N$. From the (uniform) continuity of $\mu\to\nTAP(\mu)$ in Theorem \ref{lem:TAPnew} $(ii)$, 
 \[
 \nTAP_{N,n_{N}}(m,\eps_{N}+3\eta_{N},\delta_{N})\geq\nTAP(\mu_{m'})+t\geq\nTAP(\mu_{m})+t/2
 \]
 for large $N$ and, on the subsequence as above,
 \[
 \mathbb{P}\Bigl( \exists m\in A_{N}:\,\nTAP_{N,n_{N}}(m,\eps_{N}+3\eta_{N},\delta_{N})\geq\nTAP(\mu_{m})+t/2\Bigr) \geq e^{-Nc}.
 \]
 From (\ref{eq:TAPub2}), (deterministically)
 \begin{equation}
 \exists m\in A_{N}:\,\,\mathbb{E}\nTAP_{N,n_{N}}(m,\eps_{N}+3\eta_{N},\delta_{N})\geq\nTAP(\mu_{m})+t/4.\label{eq:TAPub3}
 \end{equation}
 Since $(\MM_{-1,1},d_1)$ is a compact space, there exists a subsequence $m=m^{N}\in A_{N}$ such that $\mu_{m^{N}}\to\mu$ for some $\mu\in \MM_{-1,1}$ and (along this subsequence)
 \begin{equation}
 \limsup_{N\to\infty}\mathbb{E}\nTAP_{N,n_{N}}(m,\eps_{N}+3\eta_{N},\delta_{N})\geq\nTAP(\mu)+t/4.\label{eq:TAPub4}
 \end{equation}
 This contradicts Lemma  \ref{lem:TAP+}, which finishes the proof of the upper bound.
 \qed

 \subsection{Proof of the lower bound of \eqref{eq:TAPunifconv}}
 The proof is a variation of that of the upper bound.  Let $t,c>0$ and assume towards contradiction that for some $\eps,\,\delta$ as small as we wish and $n$ as large as we wish, there exists a subsequence in $N$ such that
 \begin{equation}
 \mathbb{P}\left( \exists m\in [-1,1]^N:\,\nTAP_{N,n}(m,\eps,\delta)\leq\nTAP(\mu_{m})-t\right) \geq e^{-Nc}.\label{eq:TAPlb1}
 \end{equation}
 Let us make sure that $\eps,\delta,n$ are such that we can choose $\eta<\eps/6$ satisfying 
 \[
 \frac{-(t/4)^2c_\xi}{1/n+\delta+\eps/2}+\log\frac{2}{\eta} <-2c,
 \]
 and such that $d_1(\mu,\mu')\leq \eta$ implies that $|\nTAP(\mu)-\nTAP(\mu')|\leq t/2.$ By Lemma \ref{lem:net2}, there exists $A_{N}$ such that \eqref{lem:net2:eq1} and \eqref{lem:net2:eq3} hold and 
 \begin{equation}
 \mathbb{P}\Big(\max_{m\in A_{N}}\big|\nTAP_{N,n}(m,\eps/2,\delta)-\mathbb{E}\nTAP_{N,n}(m,\eps/2,\delta)\big|>t/4\Big)<2e^{-2Nc}.\label{eq:TAPlb2}
 \end{equation}
 Suppose that $m'$ satisfies the inequality of (\ref{eq:TAPlb1}), and let $m\in A_{N}$ be such that \eqref{lem:net2:eq3} holds. Since $\|m-m'\|_{2}<\sqrt{N}\eta$ and $\eta<\eps/6$, we have 
 $$
 B(m',\eps)\supseteq B(m',\eps/2+3\eta)\supseteq B(m,\eps/2)
 $$ 
 and, moreover, $d_{1}(\mu_{m},\mu_{m'})<\eta$. Therefore,						\[
 \nTAP_{N,n}(m,\eps/2,\delta)\leq\nTAP_{N,n}(m',\eps,\delta)\leq \nTAP(\mu_{m'})-t\leq \nTAP(\mu_{m})-t/2.
 \]
 This implies that, on the subsequence as in (\ref{eq:TAPlb1}),
 \[
 \mathbb{P}\Bigl( \exists m\in A_{N}:\,\nTAP_{N,n}(m,\eps/2,\delta)\leq\nTAP(\mu_{m})-t/2\Bigr) \geq e^{-Nc},
 \]
 and, therefore, from (\ref{eq:TAPlb2}), (deterministically)
 \begin{equation}
 \exists m\in A_{N}:\,\,\mathbb{E}\nTAP_{N,n}(m,\eps/2,\delta)\leq\nTAP(\mu_{m})-t/4.\label{eq:ETAPub3}
 \end{equation}
 Since $(\MM_{-1,1},d_1)$ is a compact space, there exists a subsequence $m=m^{N}\in A_{N}$
 such that $\mu_{m^{N}}\to\mu$ for some $\mu\in \MM_{-1,1}$ and (along this subsequence)			
 \begin{equation}
 \liminf_{N\to\infty}\mathbb{E}\nTAP_{N,n}(m,\eps/2,\delta)\leq\nTAP(\mu)-t/4.\label{eq:TAPlb4}
 \end{equation}
 However, the condition \eqref{lem:net2:eq1} implies that $\supp(\mu)\subseteq [0,1-\eta/2]$, and the above inequality contradicts Lemma \ref{lem:TAP-}. This finishes the proof.
 \qed

 \section{Generalized TAP representation}\label{SecGTAPrepr}

 In this section, we prove the generalized TAP representation in Theorem \ref{thm:GenTAP} using the concentration of Theorem \ref{thm:TAPcorrection}. The basic idea is similar to the proof of \cite[Lemma 17]{SubagFEL}. We note that our proof of Theorem \ref{thm:TAPcorrection} in Section \ref{SecOptimizingUnif} is conditional on the results stated in the beginning of that section, and thus so is the current proof. 
 
 Let $\eps,\delta>0$ and $n\geq1$. By definition, for any $m$, 
 \[
 F_N\geq \frac{H_N(m)}{N}+\nTAP_{N,1}(m,\eps,\delta) \geq \frac{H_N(m)}{N}+\nTAP_{N,n}(m,\eps,\delta).
 \]
 In particular, from the uniform convergence of Theorem \ref{thm:TAPcorrection},
 \begin{equation*}
 	\lim_{N\to\infty}\p\Bigl(\,
 	F_N > \max_{\frac{\|m\|^2}{N}= q}\Bigl( \frac{H_N(m)}{N} + \nTAP(\mu_m)\Bigr)
 	- t\,
 	\Bigr) = 1.
 \end{equation*}
 
 Fix some value $q$ that belongs to the support of the Parisi measure of the model (\ref{hamx}).
 It is well-known that the free energy $F_N$ concentrates at exponential rate around its mean for large $N$, see e.g. \cite[Theorem 1.2]{SKmodel}. By the Borell-TIS inequality so does the maximum in \eqref{eq:GenTAP}. 
 Combining the above with Theorem \ref{thm:TAPcorrection}, we conclude that to complete the proof of Theorem \ref{thm:GenTAP} it will be enough to show that for any small $t,\eps,\delta>0$ and large $n\geq 1$, \begin{equation}
 \label{eq:goodpt}
 \begin{aligned}
 \lim_{N\to\infty}\frac1N\log\P\Big(\,\exists m\in&[-1,1]^N,\, \frac1N\|m\|^2=q:\\
 &F_N < \frac{H_N(m)}{N} + \nTAP_{N,n}(m,\eps,\delta)+t\,\Bigr)=0.
 \end{aligned}
 \end{equation}
 
 From \cite[Lemma  4.8]{SKbonus}, for any $\tau,t>0$ and $n\geq1$ we have that
 \begin{equation}
 \lim_{N\to\infty}\frac1N\log\P\bigg(\frac1{N}\log G_{N}^{\otimes 2n}\Big\{ \forall i\neq j,\, \big|R( \bs^i,\bs^j) - q   \big| < \tau  \Big\}>-t\bigg)=0.
 \end{equation}

 By conditioning on  $\bs^{n+1},\ldots,\bs^{2n}$, we conclude that, with probability not exponentially small in $N$, (w.r.t. the disorder only) there exist $\bar \bs^{n+1},\ldots,\bar \bs^{2n}\in \Sigma_N$ such that for any $n+1\leq i\neq j \leq 2n$, $\big|R(\bar  \bs^i, \bar \bs^j) - q   \big| < \tau$, and such that 
 \begin{equation}
 \label{eq:rev2}
 \begin{aligned}
 &\frac1{nN}\log G_{N}^{\otimes 2n}\Big\{ \forall i\neq j\leq 2n,\, \big|R( \bs^i,\bs^j) - q   \big| < \tau  \,\Big|\,  \bs^{n+i} = \bar \bs^{n+i},\,i=1,\ldots,n \Big\}\\
 &= \frac1{nN}\log \sum_{A_n(\tau)}
 \exp\Big\{\sum_{i=1}^{n}H_N(\bs^i)\Big\}-F_N
 >-t,
 \end{aligned}
 \end{equation}
 where we define $A_n(\tau)$ as the set of points $(\bs^{1},\ldots, \bs^{n})\in \Sigma_N^{n}$ such that  $|R( \bs^i,\bs^j) - q  |<\tau$ and $|R( \bs^i,\bar \bs^k) - q  |<\tau$, for any $i,j\leq n$ and $n+1\leq k\leq 2n$.
 
 Define $\bar{\bs}:=n^{-1}\sum_{i=n+1}^{2n}\bar{\bs}^i$ and set $m=0$ if $\bar{\bs}=0$, and 
 $m=\sqrt{Nq}\bar{\bs}/\|\bar{\bs}\|$ otherwise. Given $\delta,\eps>0$, it is straightforward to check that $A_n(\tau)\subset B_n(m,\eps,\delta)$, provided that $\tau>0$ is small enough and $n\geq1$ is large enough. 
 Assuming this inclusion and assuming that \eqref{eq:rev2} holds, $m$ is a point as in \eqref{eq:goodpt}, since
 \[
 \frac{H_N(m)}{N} + \nTAP_{N,n}(m,\eps,\delta)=\frac1{nN}\log \sum_{B_n(m,\eps,\delta)}
 \exp\Big\{\sum_{i=1}^{n}H_N(\bs^i)\Big\}.
 \]
 This completes the proof. \qed

\section{\label{sec:TAPtoF}Reduction to a model on the band}

In this section, we will justify approximations of $\nTAP_{N,n}(m,\eps,\delta)$ by $F_{N,n}(m,\eps,\delta)$. We will need the following technical lemma, which follows from a more general result in \cite[Corollary 59]{BSZ}, but which we prove here for convenience.
\begin{lem}
For a mixed $p$-spin model in (\ref{hamx}),
\begin{equation}
\e\max_{\|\bs\|\leq \sqrt{N}}\|\nabla H_N(\bs)\|
\leq c_\xi\sqrt{N},
\end{equation}
where one can take $c_\xi=2(\sum_{p\geq 1} \beta_p^2 p^3)^{1/2}.$
\end{lem}
\begin{proof}
Let us represent the maximum above as
\begin{equation}
\max_{\|\bs\|\leq \sqrt{N}}\|\nabla H_N(\bs)\|
=
\max_{\|u\|\leq 1}\max_{\|\bs\|\leq \sqrt{N}}u\cdot \nabla H_N(\bs)
\end{equation}
and denote
$$
G(u,\bs):=u\cdot \nabla H_N(\bs)
=
\sum_{p\geq 1} \frac{\beta_p}{N^{(p-1)/2}}\sum_{i_1,\ldots,i_p}g_{i_1,\ldots,i_p}(u_{i_1}\cdots\sigma_{i_p}+\ldots+\sigma_{i_1}\cdots u_{i_p}).
$$
Using that $(a_1+\ldots+a_p)^2\leq p(a_1^2+\ldots+a_p^2)$ and the fact that $\|u\|\leq 1$ and $\|\bs^j\|^2\leq N$, one can easily check that
$$
\e\bigl(G(u^1,\bs^1)-G(u^2,\bs^2)\bigr)^2
\leq
c_\xi^2 \bigl(\|u^1-u^2\|^2+N^{-1}\|\bs^1-\bs^2\| 
\bigr),
$$
where $c_\xi^2 = \sum_{p\geq 1} \beta_p^2 p^3<\infty.$ If we define, for i.i.d. standard Gaussian $g_i$ and $g_i',$
$$
\tilde G(u,\bs):=c_\xi\Bigl(\sum_{i=1}^{N}g_i u_i+N^{-1/2}\sum_{i=1}^{N}g_i' \sigma_i
\Bigr),
$$
then the right hand side above equals $\e(\tilde G(u^1,\bs^1)-\tilde G(u^2,\bs^2))^2.$ By the Sudakov-Fernique inequality,  
$$
\e \max_{u,\bs} G(u,\bs) \leq \e \max_{u,\bs} \tilde G(u,\bs)
\leq 2c_\xi\sqrt{N},
$$
and this finishes the proof.
\end{proof}

Let us start by proving the approximation in the equation (\ref{eq:ETAPn}).
\begin{lem}\label{LemHtoHm}
If $q=R(m,m)=\|m\|^2/N>0,$ the equation (\ref{eq:ETAPn}) holds, 
\begin{equation}
\label{eq:TAP(m)}
\e \nTAP_{N,n}(m,\eps,\delta) =  \frac{1}{nN} \e \log\sum_{B_{n}(m,\eps,\delta)}e^{\sum_{i=1}^{n}\hat{H}_N^{m}(\tbs^{i})}+O\Bigl(\frac{\eps}{\sqrt{q}}\Bigr),
\end{equation}
with an implicit constant in the error term that depends only on $\xi$.
\end{lem}
\begin{proof}
Given $\bs$ and $\tbs=\bs-m$, define
$$
\tbb=\tbs-\frac{\tbs\cdot m}{m\cdot m}m,\,
\bb=\bs-\frac{\tbs\cdot m}{m\cdot m}m = \tbs-\frac{\tbs\cdot m}{m\cdot m}m +m.
$$
Notice that $\bb$ is the projection of $\bs$ on the hyperplane perpendicular to $m$ passing through $m$, and $\tbb=\bb-m.$ We view $\bb$ as a function of $\bs$ but, to simplify notation, we will keep the function implicit. The projections satisfy $R(\bb^1,\bb^2)=R(\tbb^1,\tbb^2)+q$ and the Hamiltonians $H_{N}(\bb)-H_N(m)$ and $\hat{H}_N^{m}(\tbb)$ are equal in distribution. As a result,
$$
\frac{1}{nN}\e\log\sum_{B_{n}(m,\eps,\delta)}e^{\sum_{\ell=1}^{n}\big[H_{N}(\bb^{\ell})-H_N(m)\big]}
=
 \frac{1}{nN} \e \log\sum_{B_{n}(m,\eps,\delta)}e^{\sum_{\ell=1}^{n}\hat{H}_N^{m}(\tbb^{\ell})}.
$$
Since, for $\bs\in B(m,\eps),$
$$
\|\bs-\bb\|=\|\tbs-\tbb\|=\frac{\|\bs\cdot m-m\cdot m\|}{\|m\|}<\frac{\eps}{\sqrt{q}}\sqrt{N},
$$
in order to prove  (\ref{eq:ETAPn}), it is enough to use that, by the above lemma,
\begin{equation}
\e\max_{\|\bs^1-\bs^2\|\leq \eps\sqrt{N/q}}\bigl|H_N(\bs^1)-H_N(\bs^2)\bigr|
\leq c_\xi(\eps/\sqrt{q}) N,
\end{equation}
and that a similar statement holds for $\hat{H}_N^{m}$, which can be proved in exactly the same way.
\end{proof}

As we explained in the introduction, this implies that
\begin{equation}
\label{eq:TAPmAg2}
\e \nTAP_{N,n}(m,\eps,\delta) =\e F_{N,n}(m,\eps,\delta) 
+O\Bigl(\frac{\eps}{\sqrt{q}}+\Bigl(\frac{1}{n}+\delta\Bigr)^{1/2}\Bigr).\end{equation}
When $q=R(m,m)=\|m\|^2/N$ is small and the above approximation is not good enough, let us say when $q<\eps/2,$ we will use a more straightforward reduction. We will not add and subtract the term $H_N(m)$, which in this case is small (on the scale $1/N$). Instead, the only modification we will make is to remove the external field in the original model. We will consider the Hamiltonian
		\begin{align}
		\label{hamx2}
		H_N^0(\bs)=\sum_{p\geq 2}\beta_p H_{N,p}(\bs),
		\end{align}
which is the original Hamiltonian in (\ref{hamx}) with external field removed, so that	
\begin{equation}
\e H_{N}^{0}(\bs^1)H_{N}^{0}(\bs^2)= N\xi_0(R(\bs^1,\bs^2)),
\end{equation}
where $\xi_0(x)=\xi(x)-\xi'(0)x$ is defined exactly as $\xi_q(x)$ in (\ref{eqXiq11}) for $q=0$. If we define
\begin{equation}
F_{N,n}^*(m,\eps,\delta) := \frac{1}{nN}
 \log\sum_{B_{n}(m,\eps,\delta)}e^{\sum_{\ell=1}^{n}H_{N}^{0}(\bs^{\ell})}
\end{equation}	
then the argument leading to (\ref{eq:TAPmAg}) also gives in this case that	
\begin{equation}
\label{eq:TAPmAg22}
\e \nTAP_{N,n}(m,\eps,\delta) =\e F_{N,n}^*(m,\eps,\delta) 
+O\Bigl(\Bigl(\frac{1}{n}+\delta\Bigr)^{1/2}\Bigr).
\end{equation}	
In this case of small $q$, the analogue of the Parisi formula in Theorem \ref{thm1} is the following.
\begin{thm}\label{thm1ag2}
Let $m^N\in [-1,1]^N$ be any sequence such that $\mu_{m^N}\to\delta_0.$ Then
\begin{equation}
\lim_{N\to\infty}\e F_{N,1}^*(m^N,\eps_N)
= \inf_{\zeta\in \MM_{0,1}}\Bigl(
\Phi_{0,\zeta}(0,0)-\frac{1}{2}\int_0^{1}s\xi_0''(s)\zeta(s)ds
\Bigr).
\label{eqPBminAg2}
\end{equation}
provided that $\eps_N$  goes to zero slowly enough.
\end{thm}
This is, essentially, a classical Parisi formula for the original model without external field, only now we have the constraint $\bs\in B(m,\eps_N).$ If $q=\|m\|^2/{N}\leq \eps/2$ then, denoting $u=m/\|m\|,$
$$
B(m,\eps)\supseteq \bigl\{\bs: |\bs\cdot u|\leq \sqrt{\eps/2}\bigr\},
$$
and the last set is a constraint on the magnetization in some direction $u$.  Since the Hamiltonian $H_N^0$ does not contain an external field, one can show that this constraint has no free energy cost, which explains why the above formula coincides with the Parisi formula for the unconstrained model. We are not going to give a proof of this for two reasons. First reason is that the proof is straightforward. The second reason is that the proof of Theorem \ref{thm1} in the next section does this in a more complicated case, and following the same argument in this case would only significantly simplify the details.

\section{\label{sec:ParisiOnBand}Proof of the Parisi formula on the band}

We establish the proof of Theorem \ref{thm1} in this section. The argument is essentially the same as the treatment for the classical Sherrington-Kirkpatrick model, by utilizing Guerra's replica symmetry breaking scheme and the Aizenman-Sims-Starr scheme as implemented, for example, in \cite{SKmodel}. The added complication here is that the spin configurations are re-centered and the external field varies for each site, depending on the function $\ef\in C([0,1]).$ These will require extra care of the uniform convergence of the free energy in the variable $m$. To this end, in Section \ref{sec3.1}, we first give a version of Theorem \ref{thm1} in Proposition \ref{prop:Flim1} in terms of the Ruelle probability cascades followed by a set of lemmas that are devoted to establishing uniform controls of various functionals. Sections \ref{sec3.2} and \ref{sec3.3} establish the upper and lower bounds in Proposition \ref{prop:Flim1}.

\subsection{Ruelle probability cascades, and continuity results.} \label{sec3.1}
Let $q\in [0,1)$. Denote by $\mathcal{M}_{0,1-q}^d$ the collection of all atomic $\zeta\in \mathcal{M}_{0,1-q}$ satisfying that, for some integer $r\geq 1$ and some sequences
\begin{align}
&0<\zeta_0<\ldots<\zeta_{r-1}<\zeta_r=1, \label{eq:zeta}\\
&0=q_0<q_1<\ldots<q_r= 1-q, \label{eq:q}
\end{align}
we have 
\begin{equation}
\label{eq:zeta_fop}
\zeta(s) = \zeta_p \,\, \mbox{if}\,\, q_{p}\leq s<q_{p+1}\,\,\mbox{for some $0\leq p\leq r-1$.}
\end{equation}
Let $(v_\alpha)_{\alpha\in \mathbb N^r}$ be the weights of the Ruelle probability cascade \cite{RPC} corresponding to the sequence \eqref{eq:zeta} (see e.g. Section 2.3 in \cite{SKmodel} for the definition). For $\alpha^1,\alpha^2 \in \mathbb N^r$, denote 
\[
\alpha^1 \wedge\alpha^2 = \max \big\{ 0\leq p\leq r: \alpha^1_1 = \alpha^2_1,\ldots, \alpha^1_p = \alpha^2_p \big\}.
\]
Let $f(x)$ be a function of the form $\sum_{p\geq 2} c_p^2 x^p$ such that $f(x_0)<\infty$ for some $x_0>1.$ This ensures that all derivatives of $f$ are well defined on $(-x_0,x_0).$ The main choice of $f$ we have in mind is $\xi_q$, but it is convenient to keep it as a parameter in the following definitions.
Let $\theta_f(x) = xf'(x)-f(x)$. Let $g_{f'}(\alpha)$ be a centered Gaussian process on $\mathbb N^r$ with the covariance given by
\begin{equation}
\label{eq:g_f}
\E g_{f'}(\alpha^1)g_{f'}(\alpha^2) = f'(q_{\alpha^1\wedge \alpha^2}),
\end{equation}
and define $g_{\theta_f}(\alpha)$ similarly. For $m\in [0,1]^N$ and $\zeta\in \mathcal{M}_{0,1-q}^d$, we set
\begin{align}
\label{eq:Psif}
\Psi_N(m,\eps,f,\zeta) &= \frac1N \E \log \sum_{\alpha\in \mathbb N^r}v_\alpha 
\sum_{\bs \in B(m,\eps)}\exp \sum_{i\leq N} (g_{f',i}(\alpha)+\ef(m_i))\tilde\sigma_i,
\end{align}
where $\tilde \bs = \bs-m$ and $g_{f',i}$ are i.i.d. copies of $g_{f'}$. For $\zeta\in \mathcal{M}_{0,1-q}^d$, also set
\begin{align}
\label{eq:Upsilonf}
\Upsilon(f,\zeta) 
&=\frac1N \E \log \sum_{\alpha\in \mathbb N^r}v_\alpha \exp
\sqrt Ng_{\theta_f}(\alpha)
\\
&= \E \log \sum_{\alpha\in \mathbb N^r}v_\alpha \exp
g_{\theta_f}(\alpha)
=\frac{1}{2}\int_0^{1-q}\zeta(s)sf''(s)\,ds.
\nonumber
\end{align}
The second and third equalities in this equation are well-known; see e.g. \cite[Eq. (2.60)]{SKmodel}. For $\zeta\in \mathcal{M}_{0,1-q}^d$ and $f=\xi_q$, we will denote
\begin{align}
\label{eq:Psi}
\Psi_N(m,\eps,\zeta) &:= \Psi_N(m,\eps,\xi_q,\zeta)
\end{align}
and
\begin{align}
\label{eq:Upsilon}
\Upsilon(\zeta) &:= \Upsilon(\xi_q,\zeta).
\end{align}

Recall that 
\begin{equation}
F_{N,1}^h(m,\eps) = \frac{1}{N}  \log\sum_{\bs\in B(m,\eps)}e^{H_{N}^{m}(\tbs)+\sum_{i=1}^N h_i \tilde\sigma_i}.
\end{equation}
Throughout the section we will assume that 
\begin{equation}
h_i = \ef(m_i)
\end{equation}
for some $\ef\in C([0,1])$ fixed once and for all. Recall the notation $\mu_m=\frac{1}{N}\sum_{i=1}^{N}\delta_{m_i}.$

\begin{prop}
	\label{prop:Flim1}
	Given $\mu\in \MM_{0,1}$ and arbitrary sequence $m=m^N\in [0,1]^N$ such that  $\mu_{m^N}\to \mu$ weakly, if $q=\int\! x^2d\mu(x)$ then
	\begin{equation}
	\label{prop:Flim1:eq1}
	\lim_{\eps\downarrow 0}\lim_{N\to\infty} \E F_{N,1}^h(m^N,\eps)=\inf_{\zeta\in \mathcal{M}_{0,1-q}^d} \big(\Psi(\mu,\zeta) - \Upsilon(\zeta)  \big),
	\end{equation}
where the limit
	\begin{align}
		\label{prop:Flim1:eq2}
	\Psi(\mu,\zeta):=\lim_{\eps\downarrow 0}\lim_{N\rightarrow\infty}\Psi_N(m^N,\eps,\zeta)
	\end{align}
	exists and does not depend on the choice of the sequence $(m^N)$.
\end{prop}
More precisely, in (\ref{prop:Flim1:eq1}) we mean that
$$
\lim_{\eps\downarrow 0}\liminf_{N\to\infty} \E F_{N,1}^h(m^N,\eps)
=
\lim_{\eps\downarrow 0}\limsup_{N\to\infty} \E F_{N,1}^h(m^N,\eps),
$$
and both are given by the right hand side of (\ref{prop:Flim1:eq1}). In particular, we can choose $\eps_N\to 0$ slowly enough so that
$ \E F_{N,1}^h(m^N,\eps_N)$ converges to the same limit. 

Also, it is important to note that in $\Psi_N(m,\eps,\zeta)$ the variables $m$ and $q$ are two independent parameters and we will show that the quantity $\Psi(\mu,\zeta)$ in (\ref{prop:Flim1:eq2}) is well-defined for any $\mu\in \MM_{0,1}$, not necessarily satisfying the constraint $q=\int\! x^2\,d\mu(x)$. On the other hand, this constraint is crucial in (\ref{prop:Flim1:eq1}), and that is why we minimize over $\zeta\in \mathcal{M}_{0,1-q}^d$.

For the rest of this subsection, we establish the convergence in \eqref{prop:Flim1:eq2}, while the proof of \eqref{prop:Flim1:eq1} is deferred to the next two subsections. We begin with some basic continuity properties of the functionals defined above.
\begin{lem}\label{zetacont}
	For $\zeta\in \mathcal{M}_{0,1-q}^d$ and $\zeta'\in \mathcal{M}_{0,1-q'}^d$, 
		\begin{equation}
	\label{eq:zetacont}
	|\Upsilon(f,\zeta)  - \Upsilon(f,\zeta')|  \leq   \frac12 f''(1)\Bigl( \int_{0}^{1} |\zeta(s)-\zeta'(s)|ds+2|q-q'|\Bigr).
	\end{equation}
	and
	\begin{equation}
	\label{eq:zetacont2}
	|\Psi_N(m,\eps,f,\zeta) - \Psi_N(m,\eps,f,\zeta')|  \leq   2f''(1)\Bigl( \int_{0}^{1} |\zeta(s)-\zeta'(s)|ds+2|q-q'|\Bigr).
	\end{equation}
For any $f_1$ and $f_2$ as above and $\zeta\in \mathcal{M}_{0,1-q}^d$,
	\begin{equation}
	\label{eq:zetacontF}
	|\Psi_N(m,\eps,f_1,\zeta) - \Psi_N(m,\eps,f_2,\zeta)|  \leq \|f_1'-f_2'\|_\infty.
	\end{equation}
\end{lem}
The last sup-norm is defined on $[0,1]$.
\begin{proof}
The first inequality is clear. The proof of the second inequality is almost identical to the one in the classical SK model (due to Guerra \cite{Guerra}) and will be omitted. We only mention that the factor $2$ instead of the usual $1/2$ is due to the fact that $|\tilde \bs|\leq 2$ instead of $1$ and we have the additional term $2|q-q'|$,  because $\zeta$ and $\zeta'$ are defined on different intervals. If $q'<q,$ we first need to interpolate all the parameters $q_p'>1-q$ down to $1-q$ before we start the usual argument and one can check that all error along this interpolation is controlled by $4f''(1)|q-q'|.$

To prove (\ref{eq:zetacontF}), we replace the terms $g_{f_1',i}(\alpha)$ and $g_{f_2',i}(\alpha)$ by the usual Gaussian interpolation between the two, $\sqrt{t}g_{f_1',i}(\alpha)+\sqrt{1-t}g_{f_1',i}(\alpha)$. The derivative along the interpolation will be controlled by  the maximum of $|f_1'(q_{\alpha^1\wedge \alpha^2})-f_2'(q_{\alpha^1\wedge \alpha^2})|$, and this finishes the proof.
\end{proof}

One immediate consequence of the above continuity properties that will be useful is the following.
\begin{lem}\label{zetacont2}
	If $\zeta\in \mathcal{M}_{0,1-q}^d$ and $\zeta'\in \mathcal{M}_{0,1-q'}^d$ are such that $\int_{0}^{1} |\zeta(s)-\zeta'(s)|ds\leq |q-q'|$ then, for any $f_1,f_2$, 
	\begin{equation}
	\label{eq:zetacontUs}
|\Psi_N(m,\eps,f_1,\zeta) - \Psi_N(m,\eps,f_2,\zeta')|  \leq   
6f_1''(1)|q-q'|+\|f_1'-f_2'\|_\infty.
	\end{equation}
\end{lem}
\begin{proof}
This follows from Lemma \ref{zetacont}.
\end{proof}

Next, we will study continuity properties with respect to $m$. Recall the metric $d_1$ defined in (\ref{eqMd1}) and note that if the coordinates of $m,m'\in[0,1]^N$ are arranged in the non-decreasing order then
\begin{equation}
d_1(\mu_m,\mu_{m'})=\frac{1}{N}\sum_{i=1}^{N}|m_i-m_i'|=\frac{1}{N}\|m-m'\|_1.
\label{eqD1muInc}
\end{equation}
The following observation will be convenient.
\begin{lem}\label{lemObsConv}
If $\mu,\mu'\in \MM_{0,1}$ then, for $k\geq 1$,
\begin{equation}
\bigl|\int\! x^kd\mu(x)-\int\! x^kd\mu'(x)\bigr| 
\leq kd_1(\mu,\mu').
\end{equation}
\end{lem}
\begin{proof}
If $\mu^{-1}$ denotes the quantile transform of $\mu$ then
\begin{align*}
\bigl| \int\! x^kd\mu(x)-\int\! x^kd\mu'(x) \bigr|
&=
\bigl| \int\! \mu^{-1}(x)^kdx -\int{\mu'}^{-1}(x)^k dx\bigr|
\\
&\leq
k\int\! \bigl|\mu^{-1}(x)-{\mu'}^{-1}(x)\bigr| dx
=
k\int\! \bigl|\mu(x)-{\mu'}(x)\bigr| dx,
\end{align*}
which finishes the proof.
\end{proof}
For example, this implies that, for $m\in [0,1]^N$ and $\mu\in \MM_{0,1},$
$$
\Bigl|\frac{1}{N}\|m\|_2^2 - \int\! x^2d\mu(x)\Bigr|
=
\Bigl|\int\! x^2d\mu_m(x) - \int\! x^2d\mu(x)\Bigr|\leq 2d_1(\mu_m,\mu).
$$
If $\bs\in B(m,\eps)$, then the self-overlap of the recentered configuration $\tilde \bs = \bs-m$ can be rewritten as
$$
\frac{1}{N}\tilde \bs\cdot\tilde\bs=1-\frac{1}{N}m\cdot m -\frac{2}{N}(\bs-m)\cdot m
$$
and, therefore, if $q=\int\! x^2d\mu(x),$
\begin{equation}
\Bigl|\frac{1}{N}\tilde \bs\cdot\tilde\bs-(1-q)\Bigr|\leq 2(\eps + d_1(\mu_m,\mu)).
\label{eqSOonnbaq}
\end{equation}
In particular, if $\mu_m\to\mu$, then the self-overlap on the narrow band is approximately $1-q$, and the choice of $q_r=1-q$ in (\ref{eq:q}) is designed to match this.

Let us introduce one more notation. For $\ef\in C([0,1]),$ define
\begin{equation}
\Delta_\ef(x)=\inf_{t>0}\Bigl(\max_{|a-b|\leq t}|\ef(a)-\ef(b)|+2\|\ef\|_\infty\frac{x}{t}\Bigr).
\end{equation}
By the uniform continuity of $\ef,$ $\lim_{x\downarrow 0}\Delta_\ef(x)=0.$ We will use this quantity to control
\begin{align*}
\frac{1}{N}\sum_{i=1}^{N}|\ef(m_i)-\ef(m_i')|
&\leq
\max_{|a-b|\leq t}|\ef(a)-\ef(b)| + 2\|\ef\|_\infty \frac{1}{N}\sum_{i=1}^{N} I(|m_i-m_i|'>t)
\\
&\leq
\max_{|a-b|\leq t}|\ef(a)-\ef(b)| + 2\|\ef\|_\infty \frac{1}{t}\frac{1}{N}\sum_{i=1}^{N} |m_i-m_i|'
\\
&=\max_{|a-b|\leq t}|\ef(a)-\ef(b)| + 2\|\ef\|_\infty \frac{1}{t}d_1(\mu_m,\mu_{m'}),
\end{align*}
which implies that
\begin{equation}
\frac{1}{N}\sum_{i=1}^{N}|\ef(m_i)-\ef(m_i')|\leq \Delta_\ef(d_1(\mu_m,\mu_{m'})).
\label{eqDeltavd1}
\end{equation}
The next lemma contains a key result that will later allow us to approximate general vectors $m$ by nice `discrete' ones.
\begin{lem}
	\label{lem:Psidiff}
For $m,m'\in[0,1]^N$, if $N\eps>1$ and $d_1(\mu_m,\mu_{m'})<\eps^2/6$ then
\begin{equation}
	\bigl|\Psi_N(m,\eps,f,\zeta) - \Psi_N(m',\eps,f,\zeta)\bigr|<\Delta_f\bigl(d_1(\mu_m,\mu_{m'}),\eps\bigr)
	\label{lemPsidiffeq1}
\end{equation}
	and
\begin{equation}	
	\bigl|\E F_{N,1}^h(m,\eps) -\E F_{N,1}^h(m',\eps)\bigr|< \Delta_\xi\bigl(d_1(\mu_m,\mu_{m'}),\eps\bigr),
		\label{lemPsidiffeq2}
\end{equation}	
where
\begin{equation}
\Delta_f(d,\eps):= 2\Delta_v(d)+ \bigl(c_f+\|\ef\|_\infty\bigr) 2d\bigl(1+\frac{3}{\eps}\bigr) + \frac{3d}{\eps}\log\frac{\eps}{3d}
\end{equation}
for some constant $c_f$ that depends only on $f.$
\end{lem}
Note that $v\in C([0,1])$, which implies that $\lim_{d\downarrow 0}\Delta_v(d)=0$ and thus, $\lim_{d\downarrow 0}\Delta_f(d,\eps)=0$ for any $\eps>0.$ Let us also clarify that we will use $c_\xi$ to denote various constants that depend on $\xi_q$ for $q\in [0,1]$. The reason for this is that $\xi_q=\xi(x+q)-\xi(q)-\xi'(q)x$, and the derivatives of $\xi_q$ can be controlled in terms of derivatives of $\xi$ uniformly over $q.$

\begin{proof}
Since the order of the coordinates of $m$ does not affect $\Psi_N$, we can assume that they are arranged in the non-decreasing order. By (\ref{eqD1muInc}) and the triangle inequality, for $\bs\in B(m',\eps)\setminus  B(m,\eps)$,
\begin{equation}
	 \eps < \frac1N|\bs\cdot m - m\cdot m|
	 \leq \frac1N|\bs\cdot m' - m'\cdot m'|+ \frac3N\|m-m'\|_1
	 <\eps+3d_1(\mu_m,\mu_{m'}).
	 \label{eqMMpec}
\end{equation}
Since
	 \[
	 \sum_{\sigma_i=1}m_i+\sum_{\sigma_i=-1}m_i=\|m\|_1,\quad \sum_{\sigma_i=1}m_i-\sum_{\sigma_i=-1}m_i=m\cdot \bs,
	 \]
	 we have that
	 \begin{equation*}
	 \sum_{\sigma_i=1}m_i = (\|m\|_1+m\cdot \bs)/2, \quad\sum_{\sigma_i=-1}m_i= (\|m\|_1-m\cdot \bs)/2.
	 \end{equation*}
By (\ref{eqMMpec}), if $\bs\in B(m',\eps)\setminus  B(m,\eps)$ then there are two possibilities:
\begin{align}
N\eps&< \bs\cdot m -m\cdot m < N(\eps+3d_1(\mu_m,\mu_{m'})),
\label{eqTpp1}
\\
N\eps&< m\cdot m - \bs\cdot m < N(\eps+3d_1(\mu_m,\mu_{m'})).
\label{eqTpp2}
\end{align}
In the first case, 
$$
\sum_{\sigma_i=1}m_i =(\|m\|_1+m\cdot \bs)/2 > (\|m\|_1+\|m\|_2^2+N\eps)/2\geq N\eps/2
$$
and, in the second case, since $\|m\|_1\geq \|m\|_2^2$,
$$
\sum_{\sigma_i=-1}m_i =(\|m\|_1-m\cdot \bs)/2 > (\|m\|_1-\|m\|_2^2+N\eps)/2\geq N\eps/2.
$$
In both cases, the number of summands on the left hand side is at least $N\eps/2$ and, if we set
$$
d:= \frac{3d_1(\mu_m,\mu_{m'})}{\eps}
$$ 
then, by our assumption, $d<\eps/2.$ This means that the number of summands is greater than $Nd.$ In the first case, the sum of the largest $Nd$ values $m_i$ corresponding to $\sigma_i=1$ must be at least 
$$
\frac{Nd\eps}{2}=\frac{3Nd_1(\mu_m,\mu_{m'})}{2}
$$ 
and, therefore, if we flip the sign of $\sigma_i$ corresponding to these largest values from $+$ to $-$, the value of $m\cdot \bs=\sum_{\sigma_i=1}m_i-\sum_{\sigma_i=-1}m_i$ will decrease by at least $3Nd_1(\mu_m,\mu_{m'}).$ By the upper bound in (\ref{eqTpp1}), if we flip them consecutively, somewhere along the way we will have a configuration $\bs'$ such that  $-N\eps< \bs'\cdot m -m\cdot m < N\eps$ (since $N\eps> 1$, in one step we can not jump from $N\eps$ to $-N\eps$). In other words, $\bs'\in B(m,\eps)$ and, by construction, $\|\bs-\bs'\|_1\leq 2dN$. The second case is similar (we flip the sign of $\sigma_i$ corresponding to the largest values from $-$ to $+$), and we showed that
	 \begin{equation}
	\label{eq:B'toB}
	 \forall \bs\in B(m',\eps),\ \ \exists  \bs'\in B(m,\eps):\ \|\bs'-\bs\|_1\leq 2dN.
	 \end{equation} 
This implies that
		\begin{equation}
		\label{eq:Psiineq}
		\Psi_N(m',\eps,f,\zeta) \leq \frac1N \E \log \sum_{\alpha\in \mathbb N^r}v_\alpha 
		\sum_{\bs \in B(m,\eps)}\sum_{\|\bs'-\bs\|_1\leq 2dN}\exp \sum_{i\leq N} (g_{f',i}(\alpha)+\ef(m_i'))\tilde\sigma'_i,
		\end{equation}
		where we denote $\tilde\bs' = \bs'-m'$. First of all, using (\ref{eqDeltavd1}) and the fact that $|\tilde\sigma'_i|\leq 2$, we can bound this by
$$
2\Delta_v(d_1(\mu_m,\mu_{m'}))+
\frac1N \E \log \sum_{\alpha\in \mathbb N^r}v_\alpha 
		\sum_{\bs \in B(m,\eps)}\sum_{\|\bs'-\bs\|_1\leq 2dN}\exp \sum_{i\leq N} (g_{f',i}(\alpha)+\ef(m_i))\tilde\sigma'_i.
$$	
The second term is equal to $\varphi(1)$ if we define
	\[
	\varphi(s)=\frac1N \E \log \sum_{\alpha\in \mathbb N^r}v_\alpha 
	\sum_{\bs \in B(m,\eps)}\sum_{\|\bs'-\bs\|_1\leq 2dN}\exp \sum_{i\leq N} (g_{f',i}(\alpha)+\ef(m_i))(\tilde\sigma_i+s(\tilde\sigma'_i-\tilde\sigma_i)).
	\]
When $\|\bs'-\bs\|_1\leq 2dN$, we have
$$
\|\tilde\bs'-\tilde\bs\|_1\leq \|\bs'-\bs\|_1+\|m'-m\|_1\leq (d_1(\mu_m,\mu_{m'})+2d)N.
$$ 
Therefore, differentiating $\varphi(s)$ and using Gaussian integration by parts, we get
	\begin{align*}
	|\varphi'(s)| &=\big|\frac1N \E \big\langle  \sum_{i\leq N} (g_{f',i}(\alpha)+\ef(m_i))(\tilde\sigma'_i-\tilde\sigma_i)  \big\rangle_s \big|
	\\
	&\leq \big|\frac1N \E \big\langle  \sum_{i\leq N} g_{f',i}(\alpha)(\tilde\sigma'_i-\tilde\sigma_i)  \big\rangle_s \big|+\|\ef\|_\infty (d_1(\mu_m,\mu_{m'})+2d)
	\\
	&\leq \bigl(c_f+\|\ef\|_\infty\bigr) (d_1(\mu_m,\mu_{m'})+2d),
\end{align*}
	where $\langle\,\cdot\,\rangle_s$ denotes the Gibbs average along this interpolation. Lastly, if  $\mathcal{I}(t)$ is the rate function of a Rademacher random variable and $\bs_0$ is any fixed vector in $\Sigma_N$ then
\begin{align*}
	|\Psi_N(m,\eps,\zeta)-\varphi(0)|
	&\leq \frac1N \log \operatorname{card}\{\bs: \|\bs-\bs_0\|_1\leq 2dN\}
	\\
	&\leq \log 2 -{\mathcal I}(1-2d)
	\leq  d\log\frac 1d.
\end{align*}
Combining the above, we bounded $\Psi_N(m',\eps,f,\zeta)$ by
\begin{align*}
&\Psi_N(m,\eps,f,\zeta) + 2\Delta_v(d_1(\mu_m,\mu_{m'}))+ \bigl(c_f+\|\ef\|_\infty\bigr) \bigl(d_1(\mu_m,\mu_{m'})+2d\bigr) +d\log\frac 1d
\\
&= \Psi_N(m,\eps,f,\zeta)+ \Delta_f(d_1(\mu_m,\mu_{m'}),\eps).
\end{align*}
Since the same inequality holds with $m$ and $m'$ interchanged, this proves (\ref{lemPsidiffeq1}).

The proof of (\ref{lemPsidiffeq2}) is similar. Using (\ref{eq:B'toB}), we can write	
	\begin{equation}
	\label{eq:TAPhcont1}
	\E F_{N,1}^h(m',\eps) \leq  \frac{1}{N}\E  \log\sum_{\bs\in B(m,\eps)}\sum_{\|\bs'-\bs\|_1\leq 2dN}\exp\Big( H_N^{m'}(\tilde \bs')+\sum_{i=1}^{N}\ef(m_i')\tilde\sigma_i'\Big).
	\end{equation}
We can handle the external field term as above and bound this by
\begin{align*}
&2\Delta_\ef(d_1(\mu_m,\mu_{m'}))+\|\ef\|_\infty (d_1(\mu_m,\mu_{m'})+2d) 
\\
&+
\frac{1}{N}\E  \log\sum_{\bs\in B(m,\eps)}\sum_{\|\bs'-\bs\|_1\leq 2dN}\exp\Big( H_N^{m'}(\tilde \bs')+\sum_{i=1}^{N}\ef(m_i)\tilde\sigma_i\Big).
\end{align*}
Next, we will replace $H_N^{m'}(\tilde \bs')$ by $H_N^{m}(\tilde \bs)$ by using the interpolation $$
H_s(\tilde \bs',\tilde \bs):=\sqrt{s}H_N^{m'}(\tilde \bs') + \sqrt{1-s}H_N^{m}(\tilde \bs).
$$ 
By Gaussian integration by parts, the error of this interpolation will be controlled by (twice) the maximum of the covariance
$$
\frac{1}{N}\e \frac{\partial H_s(\tilde \bs',\tilde \bs)}{\partial s}H_s(\tilde \br',\tilde \br)
=\frac{1}{2}\bigl(\xi_a(\tilde R_{1,2}') - \xi_b(\tilde R_{1,2})\bigr),
$$
where 
$$
a=\frac{1}{N}\|m'\|_2^2, b=\frac{1}{N}\|m\|_2^2, \bs,\br\in B(m,\eps), \|\bs'-\bs\|_1\leq 2dN, \|\br'-\br\|_1\leq 2dN,
$$ 
and
$$
\tilde R_{1,2}'=\frac{1}{N}(\bs'-m')\cdot (\br'-m'), \tilde R_{1,2}=\frac{1}{N}(\bs-m)\cdot (\br-m).
$$	
Let us rewrite this as
$$
\frac{1}{2}\bigl(\xi_a(\tilde R_{1,2}') - \xi_b(\tilde R_{1,2})\bigr)
=
\frac{1}{2}\bigl(\xi_a(\tilde R_{1,2}') - \xi_a(\tilde R_{1,2})\bigr)
+
\frac{1}{2}\bigl(\xi_a(\tilde R_{1,2}) - \xi_b(\tilde R_{1,2})\bigr),
$$
and recall that $\xi_q(x)=\xi(x+q)-\xi(q)-\xi'(q)x.$ The second term can be bounded by
$$
\frac{1}{2}\bigl|\xi_a(\tilde R_{1,2}) - \xi_b(\tilde R_{1,2})\bigr| \leq \frac{c_\xi}{2} |a-b|\leq 
c_\xi d_1(\mu_m,\mu_{m'}).
$$
To bound the first term, by the triangle inequality,
$$
|\tilde R_{1,2}'-\tilde R_{1,2}|\leq \frac{2}{N}\bigl(\|\bs'-\bs\|_1+\|\br'-\br\|_1+2\|m'-m\|_1\bigr) \leq 4(d_1(\mu_m,\mu_{m'})+2d).
$$	
Therefore, the first term can be bounded by $c_\xi(d_1(\mu_m,\mu_{m'})+2d).$ The rest of the argument is identical, so the proof of (\ref{lemPsidiffeq2}) is complete.
	\end{proof}
We prove the following lemma by adapting an idea from Lemma 4 of \cite{panchenko2018}. 
\begin{lem}
	\label{lem:Psimueps}
	For any $\mu\in \MM_{0,1}$ and $m=m^N$ satisfying $\mu_{m^N}\to\mu$, the limit
	\begin{equation}
	\label{eq:limNPsi}
   \Psi(\mu,\eps,f,\zeta) := \lim_{N\to\infty}\Psi_N(m^N,\eps,f,\zeta)
	\end{equation}
	exists and does not depend on the choice of the sequence $(m^N)$. 
\end{lem}
We will denote
	\begin{equation}
	\label{eq:limNPsi2}
	\Psi(\mu,f,\zeta):=\lim_{\eps\downarrow 0}\Psi(\mu,\eps,f,\zeta).
	\end{equation}
Consistently with (\ref{eq:Psi}), for $\zeta\in \mathcal{M}_{0,1-q}^d$ and $f=\xi_q$, we will denote
\begin{align}
\label{eq:PsiL}
\Psi(m,\eps,\zeta) := \Psi(m,\eps,\xi_q,\zeta),\,
\Psi(\mu,\zeta):= \Psi(\mu,\xi_q,\zeta).
\end{align}
\begin{proof}
	First, assume that $\mu$  is an atomic measure with rational weights. 	
	Suppose that for some $K$, $m_1,\ldots,m_K$ is a sequence such that $\mu=\frac1K \sum_{i\leq K}\delta_{m_i}$. For $i>K$, define $m_i$ periodically $m_i=m_{i-K}$  and, for any $N$, define $m=m^N=(m_1,\ldots,m_N)$, so that, clearly, $\mu_m\to\mu$. 
	
	Let $N_1,N_2\geq 1$ be multiples of $K$ and set $N=N_1+N_2$. 	
	Note that
	\begin{equation}
	\label{eq:Bsupadtv}
	B(m^{N_1},\eps)\times  B(m^{N_2},\eps)\subseteq B( m^{N},\eps).
	\end{equation}
	Combined with standard properties of Ruelle probability cascades (see pp. 51--52 of \cite{SKmodel}), this implies that 
	\begin{equation}
	\label{eq:NPsi}
	N \Psi_N(m^N,\eps,f,\zeta)
	\end{equation}
	is a super-additive function of $N$ on multiples of $K$. One can easily verify from this that the limit $\lim_{N\to\infty}\Psi_N(m^N,\eps,f,\zeta)$ exists.
	
Since the set of atomic measures with rational weights is dense in $\MM_{0,1}$, by Lemma \ref{lem:Psidiff}, the same limit exists for general $\mu$ and it does not depend on the choice of $m$. Since \eqref{eq:NPsi} is decreasing in $\eps$, the limit in \eqref{eq:limNPsi2} is well defined.
\end{proof}

We will combine this result with Lemma \ref{zetacont2} to obtain the following.
\begin{lem}\label{zetacont2pM}
	If $\zeta\in \mathcal{M}_{0,1-q}^d$ and $\zeta'\in \mathcal{M}_{0,1-q'}^d$ are such that $\int_0^1|\zeta(s)-\zeta'(s)|ds\leq |q-q'|$ then, for any $f_1,f_2$ and $\mu,\mu'\in \MM_{0,1}$, 
	\begin{align}
	\label{eq:zetacontUspM}
&|\Psi(\mu,\eps,f_1,\zeta) - \Psi(\mu',\eps,f_2,\zeta')|  
\\
&\leq   
\Delta_{f_1}\bigl(d_1(\mu,\mu'),\eps\bigr)+6f_1''(1)|q-q'|+\|f_1'-f_2'\|_\infty.
\nonumber
	\end{align}
\end{lem}
\begin{proof}
Combining Lemma \ref{zetacont2} with the equation (\ref{lemPsidiffeq1}), we get
\begin{align*}
	&\bigl|\Psi_N(m,\eps,f_1,\zeta) - \Psi_N(m',\eps,f_2,\zeta')\bigr|
	\\
	&\leq
	\Delta_{f_1}\bigl(d_1(\mu_m,\mu_{m'}),\eps\bigr)+6f_1''(1)|q-q'|+\|f_1'-f_2'\|_\infty.
\end{align*}
Letting $\mu_m\to\mu$ and $\mu_{m'}\to\mu'$ and using previous lemma finishes the proof.
\end{proof}

\subsection{Proof of the upper bound of Proposition \ref{prop:Flim1}} \label{sec3.2}Fix some $\mu\in \MM_{0,1}$ and $\ef\in V$ throughout this subsection and let $m=m^N$ be a sequence with $\mu_{m^N}\to\mu$. Take $q=\int\! x^2d\mu(x).$ In view of Lemma \ref{lem:Psimueps}, it will be enough to show that for arbitrary $\zeta\in \mathcal{M}_{0,1-q}^d$ and $\eps_N\to0$,
\begin{equation}
\label{eq:UB}
\begin{aligned}
\limsup_{N\to\infty}\E F_{N,1}^h(m^N,\eps_N) \leq \limsup_{N\to\infty}\Psi_N(m^N,\eps_N,\zeta)  - \Upsilon(\zeta).
\end{aligned}
\end{equation}
Once we proved this, we can simply choose $\eps_N\to 0$ slowly enough so that
$$
\limsup_{N\to\infty}\E F_{N,1}^h(m^N,\eps_N) 
=
\lim_{\eps\downarrow 0}\limsup_{N\to\infty}\E F_{N,1}^h(m^N,\eps) 
$$
and, simultaneously,
$$
\limsup_{N\to\infty}\Psi_N(m^N,\eps_N,\zeta) =
\lim_{\eps\downarrow 0}\lim_{N\to\infty}\Psi_N(m^N,\eps,\zeta),
$$
which is equal to $\Psi(\mu,\zeta)$, by Lemma \ref{lem:Psimueps}. This will finish the proof of the upper bound of Proposition \ref{prop:Flim1}.

Our proof of  \eqref{eq:UB} uses Guerra's interpolation method and it is almost identical to the proof of that for mixed $p$-spin models on $\Sigma_N$ as in Sections 3.2--3.4 in \cite{SKmodel} except that we need to work with the shifted coordinates $\tilde{\bs}=\bs -m$ instead of $\bs$ and replace the cube $\Sigma_N$ by the band $B(m,\eps_N)$. 
Similar to \cite[Eq. (3.45)]{SKmodel}, we define, for $t\in[0,1]$, the interpolating Hamiltonian
\[
H_{N,t}(\tilde\bs,\alpha)= \sqrt t H_N^{m}(\tilde\bs)+\sqrt{1-t}\sum_{i=1}^{N}g_{\xi_q',i}(\alpha)\tilde\sigma_i + \sqrt t \sum_{i=1}^{N}g_{\theta_q,i}(\alpha)+\sum_{i=1}^{N}\ef(m_i)\tilde\sigma_i,
\]
indexed by  $(\tilde\bs,\alpha)\in (B(m,\eps_N)-m)\times\mathbb N^r$, where the Gaussian processes $g_{\xi_q',i}$ and $g_{\theta_q,i}$ are i.i.d. copies of the processes defined in  \eqref{eq:g_f}. Exactly as in  \cite[Eq. (3.18)]{SKmodel}, set
\begin{equation}
\label{eq:g}
g(\tilde\bs) = \sum_{p\geq1}2^{-p}x_pg_p(\tilde\bs),\qquad
g_p(\tilde\bs) = \frac{1}{N^{p/2}}\sum_{i_1,\ldots,i_p=1}^N g_{i_1,\ldots,i_p}'\tilde\sigma_{i_1}\cdots \tilde\sigma_{i_p},
\end{equation}
$g_{i_1,\ldots,i_p}'$ are i.i.d. standard Gaussian variables, and $x=(x_p)_{p\geq 1}$ is a sequence of i.i.d. uniform random variables on $[1,2]$. Let $s_N$ be a sequence such that $s_N\rightarrow\infty$ and $s_N^2/N\rightarrow 0$.
For $t\in [0,1]$, define the interpolating free energy by 
\begin{equation}\label{add:eq3}
\varphi_N(t) = \frac1N \E_x\E \log \sum_{\alpha\in \mathbb N^r}v_\alpha 
\sum_{\bs \in B(m,\eps_N)}\exp \bigl(H_{N,t}(\tilde \bs, \alpha)+s_Ng(\tilde\bs) \bigr).
\end{equation}
Here and hereinafter, $\E_x$ means the expectation with respect to the randomness $x$ only. Denote by $G_{N,t}(\tilde\bs,\alpha)$ the Gibbs measure and by $\langle\cdot\rangle_t$ the Gibbs average associated to this free energy. Observe that since $s_N^2/N\rightarrow 0,$ the term $s_Ng(\tilde\bs)$ plays the role as a vanishing perturbation such that 
\begin{align}
\begin{split}\label{prop:Flim1:proof:eq1}
\varphi_N(1)&=\e F_{N,1}^h(m)+\Upsilon(\zeta)+o(1),\\
\varphi_N(0)&=\Psi_N(m,\eps_N,\zeta)+o(1).
\end{split}
\end{align}
In order to compare these two sides, an application of the Gaussian integration by parts (see \cite[Theorem 3.5]{SKmodel}) implies that, as $N\to\infty$,
\begin{align*}
\varphi_N'(t)=-\frac{1}{2}\E_x\E\big\langle\bigl(\xi_q(R(\tilde{\bs}^1,\tilde{\bs}^2))-R(\tilde{\bs}^1,\tilde{\bs}^2)\xi_q'(q_{\alpha_1\wedge \alpha_2})+\theta_q(q_{\alpha_1\wedge \alpha_2})\bigr)\big\rangle_t+o(1).
\end{align*}
The term $o(1)$ comes from the bound on the expression involving self-overlaps,
$$
\xi_q(R(\tilde{\bs}^1,\tilde{\bs}^1))-R(\tilde{\bs}^1,\tilde{\bs}^1)\xi_q'(q_{\alpha_1\wedge \alpha_1})+\theta_q(q_{\alpha_1\wedge \alpha_1}).
$$
Indeed, the fact that $\eps_N\to 0$ and $\mu_{m^N}\to \mu$ ensures, by (\ref{eqSOonnbaq}), that $R(\tilde{\bs}^1,\tilde{\bs}^1)\approx 1-q,$  and by our choice of $q_r=1-q,$ we have $q_{\alpha_1\wedge \alpha_1}=1-q.$ Because the self-overlap is nearly constant, the proof of the extended Ghirlanda-Guerra identities in the average sense, as well as Talagrand's positivity principle, in \cite{SKmodel} requires no modifications and Theorem 3.4 in \cite{SKmodel} implies that
\begin{align*}
\lim_{N\rightarrow\infty}\E_x\E\langle \indic(R(\tilde{\bs}^1,\tilde{\bs}^2)\leq -\varepsilon)\rangle_t=0,\,\,\forall \varepsilon>0.
\end{align*}
This together with the fact that $\xi_q$ is a convex function on $\Reals_+$ implies that $$
\limsup_{N\rightarrow\infty}\varphi_N'(t)\leq 0.$$
Consequently, the asserted inequality follows from \eqref{prop:Flim1:proof:eq1}.
\qed

\subsection{Proof of the lower bound of Proposition \ref{prop:Flim1}}\label{sec3.3}
Let us consider an atomic measure $\mu$ with finitely many jumps and rational weights, and let
$$
q=\int\! x^2 d\mu(x),  \gamma=\int\!x(1-x)d\mu(x).
$$ 
Denote
\begin{equation}
\delta(\eps,\gamma):= (c_\xi+\|\ef\|_\infty)\Bigl(\indic(\gamma>0) \frac{\eps}{\gamma}\log \frac{\gamma}{\eps}+\indic(\gamma=0)\eps\log \frac{1}{\eps}\Bigr),
\label{deltaag2}
\end{equation}
for some constant $c_\xi$ that will be determined in the proof below. Let $\eps>0$ be such that $\eps<\gamma^2/2,$ when $\gamma>0$. We will show that if $\mu_m\to\mu$ then
\begin{equation}
\label{eq:FqnuLB}
\liminf_{N\to\infty} \E F_{N,1}^h(m,\eps)
\geq 
\inf_{\zeta\in \mathcal{M}_{0,1-q}^d}\bigl(\Psi(\mu,\eps,\zeta) - \Upsilon(\zeta)\bigr)
-\eps-\delta(\eps,\gamma).
\end{equation}
Before we prove this, let us show why this implies the lower bound in Proposition \ref{prop:Flim1}.

\begin{lem}
If (\ref{eq:FqnuLB}) holds for atomic measures $\mu$ with rational weights then the lower bound in Proposition \ref{prop:Flim1} holds for all $\mu\in \MM_{0,1}.$ 
\end{lem}
\begin{proof}
Take $\mu\in \MM_{0,1}$ with $\gamma=\int\!x(1-x)d\mu(x)$ and $q=\int\!x^2 d\mu(x).$ We can find an atomic $\mu'\in \MM_{0,1}$ with rational weights and $d_1(\mu,\mu')$ as small as we wish, but we can also make sure that if $\gamma=0$ then $\gamma'=\int\!x(1-x)d\mu'(x)=0$ and if $\gamma>0$ then $\gamma'>\gamma/2.$ Let $q'=\int\!x^2 d\mu'(x).$ When $d_1(\mu,\mu')<\eps^2/6,$ (\ref{lemPsidiffeq2}) implies that (assuming (\ref{eq:FqnuLB}))
\begin{align*}
&
\liminf_{N\to\infty} \E F_{N,1}^h(m,\eps)
\geq \liminf_{N\to\infty} \E F_{N,1}^h(m',\eps)-\Delta_\xi(d_1(\mu,\mu'),\eps)
\\
&\geq
\inf_{\zeta'\in \mathcal{M}_{0,1-q'}^d}\bigl(\Psi(\mu',\eps,\zeta') - \Upsilon(\zeta')\bigr)-\Delta_\xi(d_1(\mu,\mu'),\eps)
-\eps-\delta(\eps,\gamma).
\end{align*}
For any $\zeta'\in \mathcal{M}_{0,1-q'}^d$  we can find $\zeta\in \mathcal{M}_{0,1-q}^d$ such that $\int_0^1|\zeta(s)-\zeta'(s)|ds\leq |q-q'|$. Then by  Lemma \ref{zetacont},
		\begin{equation}
|\Upsilon(\zeta)  - \Upsilon(\zeta')|=
|\Upsilon(\xi_q,\zeta)  - \Upsilon(\xi_{q'},\zeta')|  \leq  c_\xi |q-q'|
	\end{equation}
and, using Lemma \ref{zetacont2pM} with $f_1=\xi_q$ and $f_2=\xi_{q'}$, we can bound the above lower limit from below by
\begin{align*}
&
\inf_{\zeta\in \mathcal{M}_{0,1-q}^d}\bigl(\Psi(\mu,\eps,\zeta) - \Upsilon(\zeta)\bigr)-2\Delta_{\xi_q}(d_1(\mu,\mu'),\eps) -\eps-\delta(\eps,\gamma)
-c_\xi|q-q'|.
\end{align*}
Here we also used that $\|\xi_q'-\xi_{q'}'\|_\infty\leq c_\xi|q-q'|$. Letting $\mu'\to\mu$ (so that $q'\to q$) and then letting $\eps\downarrow 0$ finishes the proof.
\end{proof}

We now proceed with the proof of (\ref{eq:FqnuLB}). Suppose that for some $K$, $m_1,\ldots,m_K$ is a sequence such that $\mu=K^{-1} \sum_{i\leq K}\delta_{m_i}$. For $i>K$, define $m_i$ periodically $m_i=m_{i-K}$ and define $m=m^N=(m_1,\ldots,m_N)$, so that $\mu_m\to\mu$. Since 
$$
\frac{\|m^N\|^2}{N}\to \int\!x^2d\mu(x)=q,
$$
we have
\begin{align*}
\E F_{N,1}^h(m,\eps)
&=\frac1N \E \log \sum_{\bs\in B(m^{N},\eps)}\exp\big( H^m_{N}(\tilde\bs)+\sum_{i\leq N+M}\ef(m_i)\tilde\sigma_i\big)
\\
&\approx
\frac1N \E \log \sum_{\bs\in B(m^{N},\eps)}\exp\big( H^q_{N}(\tilde\bs)+\sum_{i\leq N+M}\ef(m_i)\tilde\sigma_i\big)
\end{align*}
where $H_{N}^{q}(\tbs)$ is the Hamiltonian with the covariance $N\xi_q(R(\tbs^1,\tbs^2))$ with the function $\xi_q$ defined as in \eqref{eqXiq}. Recall that $H_{N}^{m}(\tbs)$ was defined exactly as $H_{N}^{q}(\tbs)$ only with $q$ given by $\|m^N\|^2/N$, and the standard interpolation argument (as in the proof of Lemma \ref{lem:Psidiff}) shows that we can replace $\|m^N\|^2/N$ by its limit $q$ to make sure we are working with the `same' Hamiltonian along the entire sequence.  

The proof of the lower bound will be based on a standard cavity computation. Fix $\eps>0$ and some integer $M\geq1$ and write
\begin{align}
\liminf_{N\to\infty}\E F_{N,1}^h(m,\eps)\geq &\liminf_{N\to\infty}\frac1M\Big(
\E \log \sum_{\bs\in B(m^{N+M},\eps)}\exp\big( H^q_{N+M}(\tilde\bs)+\sum_{i\leq N+M}\ef(m_i)\tilde\sigma_i\big)
\nonumber
\\
&-\E \log \sum_{\bs\in B(m^{N},\eps)}\exp\big( H^q_{N}(\tilde\bs)+\sum_{i\leq N}\ef(m_i)\tilde\sigma_i\big)
\Big).
\label{eq:lb1}
\end{align}
Notice that, since the lower limit on the left hand side does not change if we take it over $N$ proportional to $K$, we can take the lower limit on the right hand side also over such $N$. In particular, by periodicity, 
\begin{equation}
m_{N+i}=m_{i} \mbox{ and } \mu_{m^N}=\mu.
\label{eqPeriodMn}
\end{equation}

By an abuse of notation, let $m^M=(m^M_i)_{i\leq M}$ denote the vector with elements $m^M_i=m_{N+i}=m_i$. The latter are the so-called `cavity coordinates'. Note that
\[
B(m^N,\eps)\times  B(m^M,\eps)\subseteq B(m^{N+M},\eps).
\]
Assume henceforth that $M$ is large enough so that all the bands above are non-empty, for large $N$. Denoting by $H^q_{N+M}(\tilde\bs, \tilde \br)$ the value of the Hamiltonian $H^q_{N+M}$ at the vector obtained by concatenation of $\tilde\bs=\bs-m^N$ and $\tilde \br = \br-m^M$, \eqref{eq:lb1} is  bounded from below by
\begin{align}
\nonumber
\liminf_{N\to\infty}\frac1M\Big(&
\E \log \sum_{\bs\in B(m^{N},\eps)}\sum_{\br\in B(m^{M},\eps)}\exp\big( H^q_{N+M}(\tilde\bs,\tilde{\br})+\sum_{i\leq N}\ef(m_i)\tilde\sigma_i+\sum_{i\leq M}\ef(m_{i})\tilde\rho_i\big)\\
&-\E \log \sum_{\bs\in B(m^{N},\eps)}\exp\big( H^q_{N}(\tilde\bs)+\sum_{i\leq N}\ef(m_i)\tilde\sigma_i\big)
\Big).
\label{eq:lb2}
\end{align}
By replacing $B(m^N,\eps)$ by $B(m^N,\eps_N)$ in the first sum, with some sequence $\eps_N\to0$, we only reduce \eqref{eq:lb2}. To replace $B(m^N,\eps)$ by $B(m^N,\eps_N)$ in the second term, we can argue as in Lemma \ref{lem:Psidiff}, as follows. If $\bs\in B(m,\eps)\setminus  B(m,\eps_N)$ then there are two possibilities:
\begin{align}
N\eps_N&< \bs\cdot m -m\cdot m < N\eps,
\label{eqTpp1ag}
\\
N\eps_N&< m\cdot m - \bs\cdot m < N\eps.
\label{eqTpp2ag}
\end{align}
Since, by (\ref{eqPeriodMn}), $\mu_m=\mu$ and, therefore, 
$$
\frac{1}{N}(\|m\|_1-\|m\|_2^2)= \int\!x(1-x)d\mu_m(x)=\gamma.
$$ 
We will argue differently in the case when $\gamma>0$ or $\gamma=0.$

If $\gamma>0$, (arguing as below (\ref{eqTpp1}), (\ref{eqTpp2}) above) in the first case, 
\begin{equation}
\sum_{\sigma_i=1}m_i =(\|m\|_1+m\cdot \bs)/2 > (\|m\|_1+\|m\|_2^2+N\eps_N)/2\geq N\gamma/2
\label{eqTpp2ag1N}
\end{equation}
and, in the second case, 
\begin{equation}
\sum_{\sigma_i=-1}m_i =(\|m\|_1-m\cdot \bs)/2 > (\|m\|_1-\|m\|_2^2+N\eps_N)/2\geq N\gamma/2.
\label{eqTpp2ag2N}
\end{equation}
If $N\eps_N>1,$ $\eps<\gamma^2/2$ and $d:=\eps/\gamma$ then we can argue exactly as in Lemma \ref{lem:Psidiff} that
	 \begin{equation}
	\label{eq:B'toBag2}
	 \forall \bs\in B(m,\eps),\ \ \exists  \bs'\in B(m,\eps_N):\ \|\bs'-\bs\|_1\leq 2dN
	 \end{equation} 
and, therefore,
\begin{align*}
&
\E \log \sum_{\bs\in B(m^{N},\eps)}\exp\big( H^q_{N}(\tilde\bs)+\sum_{i\leq N}\ef(m_i)\tilde\sigma_i\big)
\\
&\leq 
\E \log \sum_{\bs\in B(m^{N},\eps_N)}\exp\big( H^q_{N}(\tilde\bs)+\sum_{i\leq N}\ef(m_i)\tilde\sigma_i\big)+(c_\xi+\|\ef\|_\infty)\frac{\eps}{\gamma}\log \frac{\gamma}{\eps}. 
\end{align*}
Using this, we can bound \eqref{eq:lb2} from below by  
\begin{align}
\nonumber
\liminf_{N\to\infty}\frac1M\Big(&
\E \log \sum_{\bs\in B(m^{N},\eps_N)}\sum_{\br\in B(m^{M},\eps)}\exp\big( H^q_{N+M}(\tilde\bs,\tilde{\br})+\sum_{i\leq N}\ef(m_i)\tilde\sigma_i+\sum_{i\leq M}\ef(m_{i})\tilde\rho_i\big)\\
&-\E \log \sum_{\bs\in B(m^{N},\eps_N)}\exp\big( H^q_{N}(\tilde\bs)+\sum_{i\leq N}\ef(m_i)\tilde\sigma_i\big)
\Big)-(c_\xi+\|\ef\|_\infty)\frac{\eps}{\gamma}\log \frac{\gamma}{\eps}.
\label{eq:lb2'}
\end{align}

Now, let us consider the case $\gamma=0,$ when all $m_i$ are equal to $0$ or $1.$ Since $\bs\cdot m -m\cdot m\leq N(\|m\|_1-\|m\|_2^2)=0,$ the first case (\ref{eqTpp1ag}) is not possible. The second case (\ref{eqTpp2ag}) can be rewritten as
\begin{equation}
N\eps_N< m\cdot m- \bs\cdot m=2\sum_{\sigma_i=-1}m_i
=2\sum_{\sigma_i=-1}\indic(m_i=1)< N\eps.
\label{eqTpp2ag2}
\end{equation}
First of all, if we flip all $\sigma_i=-1$ corresponding to $m_i=1$ to $+1$ then $m\cdot m- \bs\cdot m$ will become $0$. If we flip them consecutively then somewhere along the way we will find $\bs'\in B(m,\eps_N).$ On the other hand, the second inequality in (\ref{eqTpp2ag2}) implies that the number of such $m_i= 1$ is bounded by $N\eps/2.$ Therefore, we need to flip at most $N\eps/2$ coordinates, which proves that
	 \begin{equation}
	\label{eq:B'toBag3}
	 \forall \bs\in B(m,\eps),\ \ \exists  \bs'\in B(m,\eps_N):\ \|\bs'-\bs\|_1\leq \eps N/2.
	 \end{equation} 
Therefore, 
\begin{align*}
&
\E \log \sum_{\bs\in B(m^{N},\eps)}\exp\big( H^q_{N}(\tilde\bs)+\sum_{i\leq N}\ef(m_i)\tilde\sigma_i\big)
\\
&\leq 
\E \log \sum_{\bs\in B(m^{N},\eps_N)}\exp\big( H^q_{N}(\tilde\bs)+\sum_{i\leq N}\ef(m_i)\tilde\sigma_i\big)+(c_\xi+\|\ef\|_\infty)\eps\log \frac{1}{\eps}
\end{align*}
and \eqref{eq:lb2} is bounded from below by  
\begin{align}
\nonumber
\liminf_{N\to\infty}\frac1M\Big(&
\E \log \sum_{\bs\in B(m^{N},\eps_N)}\sum_{\br\in B(m^{M},\eps)}\exp\big( H^q_{N+M}(\tilde\bs,\tilde{\br})+\sum_{i\leq N}\ef(m_i)\tilde\sigma_i+\sum_{i\leq M}\ef(m_{i})\tilde\rho_i\big)\\
&-\E \log \sum_{\bs\in B(m^{N},\eps_N)}\exp\big( H^q_{N}(\tilde\bs)+\sum_{i\leq N}\ef(m_i)\tilde\sigma_i\big)
\Big)-(c_\xi+\|\ef\|_\infty)\eps\log \frac{1}{\eps}.
\label{eq:lb2'agA}
\end{align}
Recalling the notation (\ref{deltaag2}), both cases can be combined as
\begin{align}
\nonumber
\liminf_{N\to\infty}\frac1M\Big(&
\E \log \sum_{\bs\in B(m^{N},\eps_N)}\sum_{\br\in B(m^{M},\eps)}\exp\big( H^q_{N+M}(\tilde\bs,\tilde{\br})+\sum_{i\leq N}\ef(m_i)\tilde\sigma_i+\sum_{i\leq M}\ef(m_{i})\tilde\rho_i\big)\\
&-\E \log \sum_{\bs\in B(m^{N},\eps_N)}\exp\big( H^q_{N}(\tilde\bs)+\sum_{i\leq N}\ef(m_i)\tilde\sigma_i\big)
\Big)-\delta(\eps,\gamma).
\label{eq:lb2'agAag}
\end{align}
The advantage of working with \eqref{eq:lb2'agAag} instead of \eqref{eq:lb2} is that the self-overlap $R(\tbs,\tbs)$ converges uniformly over $\bs\in B(m,\eps_N)$ to $1-q$, a fact which will later be important when we invoke the Ghirlanda-Guerra identities.

Next we use the standard cavity computation, known as the Aizenman-Sims-Starr scheme \cite{AS2} (see e.g. \cite{SKmodel}, or \cite{chen2013}). Consider the Hamiltonian
\begin{equation}
 H_{N,M}^q(\tilde\bs) := \sum_{p\geq2}\frac{\beta_p(q)}{(N+M)^{(p-1)/2}}\sum_{1\leq i_1,\ldots,i_p\leq N}g_{i_1,\ldots,i_p}\tilde\sigma_{i_1}\cdots \tilde\sigma_{i_p},
\end{equation}
with the covariance
$$
\E  H_{N,M}^q(\tilde\bs^1) H_{N,M}^q(\tilde\bs^2) = (N+M) \xi_q \Big(\frac{N}{N+M}R(\tilde\bs^1,\tilde\bs^2 )\Big)
$$
and independent Hamiltonians $z(\tilde\bs)$ and $y(\tilde\bs)$ with covariances
\begin{align*}
\E z(\tilde\bs^1)z(\tilde\bs^2) &=  \xi_q' \bigl(R(\tilde\bs^1,\tilde\bs^2 )\bigr),
\\
\E y(\tilde\bs^1)y(\tilde\bs^2) &=  \theta_q' \bigl(R(\tilde\bs^1,\tilde\bs^2 )\bigr),
\end{align*}
where $\theta_q(x)=x\xi_q'(x)-\xi_q(x)$. Let  $z_i(\tilde\bs)$ be independent copies of  $z(\tilde\bs)$ for $i\geq 1.$ We denote by $G_{M,N}(\tilde\bs)$ the Gibbs measure proportional to
\begin{equation}
\label{eq:GNM}
\indic_{B(m^N,\eps_N)}(\bs)\exp\big( H_{N,M}^q(\tilde\bs)+\sum_{i\leq N}\ef(m_i)\tilde\sigma_i\big)
\end{equation}
and by $\langle\,\cdot\,\rangle_{N,M}$ its average. If we denote
\begin{equation}
\label{eq:TNM}
\begin{aligned}
T_{N,M}:= &\,\, \frac1M
\E \log \Big\langle\sum_{\br\in B(m^{M},\eps)}\exp\sum_{i\leq M}\big( \tilde\rho_i z_i(\tilde\bs) 
+\ef(m_{i})\tilde\rho_i
\big)\Big\rangle_{N,M}
\\
&\,\,\,- \frac1M\E \log \Big\langle \exp\sqrt M y(\tilde\bs)\Big\rangle_{N,M}
\end{aligned}
\end{equation}
then a straightforward interpolation argument can be used to rewrite \eqref{eq:lb2'agAag} as
\begin{equation}
\label{eq:lb3}
\liminf_{N\to\infty} T_{N,M} -\delta(\eps,\gamma).
\end{equation}
Recall the perturbation $s_Ng(\tilde\bs)$ from \eqref{add:eq3}. Everywhere above we could have replaced $H_N^{q}$ by the perturbed Hamiltonian
\begin{equation}
\label{eq:Hpert2}
H_{N,q}^{{\rm pert}}(\tilde\bs)=H_N^{q}(\tilde\bs)+s_N g(\tilde\bs),
\end{equation}
with $s_N=N^{1/3}$ (here we can take any power strictly between $1/4$ and $1/2$). Then one can still show (see Section 3.5 in \cite{SKmodel}) that (\ref{eq:lb3}) still holds uniformly over the choice of $(x_p)$ in the perturbation term $g(\tilde\bs)$, with the Gibbs measure \eqref{eq:GNM} modified by
\begin{align*}
G_{N,M}^{\rm{pert}}(\tilde\bs)\thicksim 
\indic_{B(m^N,\eps_N)}(\bs)
\exp\bigl(H_{N,M}^q(\tilde\bs)+\sum_{i\leq N}\ef(m_i)\tilde{\sigma}_i+s_Ng(\tilde\bs)\bigr).
\end{align*}
Moreover, we can choose the parameters in the perturbation term $x_p^{N,M}$ in such a way that the above Gibbs measure satisfies the Ghirlanda-Guerra identities, asymptotically. This chain of arguments is standard, and we refer the reader to Section 3.6 of \cite{SKmodel} for more details.

Next, consider a subsequence along which the lower limit in \eqref{eq:lb3} is obtained, and take a further subsequence along which the array of overlaps $(R(\tilde\bs^{\ell}\cdot \tilde\bs^{\ell'}) )_{\ell,\ell'\geq1}$ of configurations sampled from $G_{N,M}^{\rm{pert}}$ converges in the sense of finite dimensional distributions. By the main result of \cite{ultrametricity}, the Ghirlanda-Guerra identities  imply that the limiting array of overlaps is ultrametric and it can be approximated (in the sense of finite dimensional distributions) by the overlaps generated by a sequence of Ruelle probability cascades, say corresponding to a sequence of distribution functions $\zeta_n$  with finitely many atoms, as in \eqref{eq:zeta_fop}. 

Recall $\Psi_M(m,\eps,\zeta)$ and  $\Upsilon(\zeta)$ from \eqref{eq:Psi} and \eqref{eq:Upsilon}. Denote
\[
T_{\zeta,M}:= \Psi_M(m^M,\eps,\zeta) - \Upsilon(\zeta).
\]
Recall the notation $T_{N,M}$ from \eqref{eq:TNM}. Note that the covariance function of the Gaussian processes  $g_{\xi'_{q},i}$, $g_{\theta_{q}}$ and $z_i$, $y$ used in  the definition of $T_{\zeta,M}$ and $T_{N,M}$ has the same dependence on the overlap the variables $\alpha$ and $\tilde \bs$, respectively.    Combining this with a straightforward  generalization of \cite[Theorem 1.3]{SKmodel}, we have that, there exists $\zeta^M\in \mathcal{M}_{0,1-q}^d$ such that 
\[
\big|\, \liminf_{N\to\infty} T_{N,M} -  T_{\zeta^M,M} \big| < \eps/2.
\] 
Finally, take a subsequence of $(\zeta^M)$ that converges to some limit distribution $\zeta'$. By Lemma \ref{zetacont}, if we choose $\zeta\in \mathcal{M}_{0,1-q}^d$ such that $d_1(\zeta',\zeta)$ is small enough, we can ensure that $|T_{\zeta^M,M}-T_{\zeta,M}|<\eps/2$ for large enough $M$. This together with (\ref{eq:lb3}) gives that
$$
\liminf_{N\to\infty}\E F_{N,1}^h(m,\eps)
\geq \Psi_M(m^M,\eps,\zeta) - \Upsilon(\zeta) -\eps
-\delta(\eps,\gamma),
$$
for $M$ large enough. Taking the limit $M\to\infty$ and using Lemma \ref{lem:Psimueps} proves \eqref{eq:FqnuLB}. 
\qed

\subsection{Proof of Theorem \ref{thm1}.} \label{sec3.4} 
Let us recall the definition of $\Psi_N(m,\eps,\zeta)$ in (\ref{eq:Psi}),
\begin{align}
\label{eq:PsiNoz}
\Psi_N(m,\eps,\zeta) &= \frac1N \E \log \sum_{\alpha\in \mathbb N^r}v_\alpha 
\sum_{\bs \in B(m,\eps)}\exp \sum_{i\leq N} (g_{\xi_q',i}(\alpha)+\ef(m_i))\tilde\sigma_i,
\end{align}
for $\zeta\in \mathcal{M}_{0,1-q}^d.$ For $a\in [0,1]$ and $(t,x)\in [0,1-q]\times \mathbb{R}$, let $\Phi_{a,\zeta}(t,x)$ be the solution of
\begin{align*}
\partial_t\Phi_{a,\zeta}&=-\frac{\xi_q''(t)}{2}\bigl(\partial_{xx}\Phi_{a,\zeta}+\zeta(t)(\partial_x\Phi_{a,\zeta})^2\bigr)
\end{align*}
with the boundary condition
\begin{align}
\Phi_{a,\zeta}(1-q,x)&=\log\sum_{\sigma=\pm 1}e^{(\sigma-a)x}
=\log 2-ax+\log \ch x.
\label{eqS24bc}
\end{align}
By Proposition \ref{prop:Flim1}, in order to finish the proof of Theorem \ref{thm1}, we need to prove the following.
\begin{prop}\label{ThSec24prop}
 For any $\zeta\in \mathcal{M}_{0,1-q}^d,$ and $\mu\in \MM_{0,1}$ with $\int\! x^2d\mu(x)=q$, 
\begin{equation}
\lim_{\eps\downarrow 0}\lim_{N\to\infty} \Psi_N(m,\eps,\zeta) = 
\inf_{\lambda\in\Reals}
\int\! \Phi_{a,\zeta}(0,\lambda a+\ef(a))\,d\mu(a)
\label{eqSec24Prop}
\end{equation}
for any sequence $m=m^N$ such that $\mu_m\to\mu.$ 
\end{prop}
The infimum in the Proposition \ref{prop:Flim1} is taken over atomic $\zeta\in \mathcal{M}_{0,1-q}^d$, but, since the boundary condition satisfies
\begin{equation}
|\partial_x \Phi_{a,\zeta}(1-q,x)|\leq 2, |\partial_{xx}\Phi_{a,\zeta}(1-q,x)|\leq 1,
\label{sec24boundDer}
\end{equation}
one can show (using the standard argument of Guerra mentioned in Lemma \ref{zetacont} above) that the functional $\zeta\to \Phi_{a,\zeta}(0,x)$ is uniformly $d_1$-Lipschitz over all parameters, and the infimum can be taken over all $\zeta.$

Since for the rest of the section $\zeta\in \mathcal{M}_{0,1-q}^d$ is fixed, we will omit it and write (\ref{eq:PsiNoz}) as
\begin{align}
\label{eq:PsiNozag}
\Psi_N(m,\eps) &= \frac1N \E \log \sum_{\alpha\in \mathbb N^r}v_\alpha 
\sum_{\bs \in B(m,\eps)}\exp \sum_{i\leq N} (g_{\xi_q',i}(\alpha)+\ef(m_i))\tilde\sigma_i.
\end{align}
For $\lambda\in \mathbb{R}$ and $A\subseteq \Sigma_N,$ let us define 
\begin{align}
\nonumber
\Psi_N(m;\lambda, A) 
&= \frac1N \E \log \sum_{\alpha\in \mathbb N^r}v_\alpha 
\sum_{\bs \in A} \exp\Bigl[\,\sum_{i\leq N} \bigl(g_{\xi_q',i}(\alpha)+\ef(m_i)\bigr)\tilde \sigma_i+\lambda N R(\tilde\bs,m)\,\Bigr]
\\
&= \frac1N \E \log \sum_{\alpha\in \mathbb N^r}v_\alpha 
\sum_{\bs \in A} \exp\sum_{i\leq N} \bigl(g_{\xi_q',i}(\alpha)+\lambda m_i+\ef(m_i)\bigr)\tilde \sigma_i.
\label{eq:Phi}
\end{align}
For $\eps>0$, recall that $B(m,\eps)=\{\bs: |R(\tilde\bs,m)|<\eps\}$ and let
$$
B^+(m,\eps)=\{\bs: R(\tilde\bs,m)\geq \eps\},\,\,
B^-(m,\eps)=\{\bs: R(\tilde\bs,m)\leq -\eps\}.
$$
Note that for $\gamma\geq 0$,
\begin{equation}
\Psi_N(m;\lambda, B^\pm(m,\eps))\leq \Psi_N(m;\lambda\pm \gamma, \Sigma_N)-\gamma\eps.
\label{eqSec24PsiEcomp}
\end{equation}
Our strategy to prove Proposition \ref{ThSec24prop} will be to show that, with the choice of $\lambda=\lambda_0$ corresponding to the minimizer in (\ref{eqSec24Prop}), for any $\eps>0$, the  quantities $\Psi_N(m;\lambda, B^\pm(m,\eps))$ will be strictly smaller than $\Psi_N(m;\lambda, \Sigma_N),$ which will imply (via concentration) that $\Psi_N(m;\lambda_0, B(m,\eps))\approx \Psi_N(m;\lambda_0, \Sigma_N).$ To achieve this goal, we need two auxiliary lemmas.
\begin{lem}\label{add:lem1}
For any $\lambda>0,$ we have
\begin{equation}
\label{add:lem1Eq1}
\Psi_N(m;\lambda, \Sigma_N)=\int\! \Phi_{a,\zeta}(0,\lambda a+\ef(a))d\mu_m(a)
\end{equation}
and
\begin{equation}
\bigl| \Psi_N(m,\eps) -\Psi_N(m;\lambda, B(m,\eps)) \bigr| \leq |\lambda| \eps.
\end{equation}
\end{lem}
\begin{proof}
Recall the notation in (\ref{eqS24bc}). Then, by the standard properties of the Ruelle probability cascades,
\begin{align*}
\Psi_N(m;\lambda, \Sigma_N)
&= \frac1N \E \log \sum_{\alpha\in \mathbb N^r}v_\alpha 
\sum_{\bs \in \Sigma_N} \exp\sum_{i\leq N} \bigl(g_{\xi_q',i}(\alpha)+\lambda m_i+\ef(m_i)\bigr)\tilde \sigma_i
\\
&= \frac1N \E \log \sum_{\alpha\in \mathbb N^r}v_\alpha 
\prod_{i\leq N}\sum_{\sigma=\pm1} \exp (\sigma-m_i)\bigl(g_{\xi_q',i}(\alpha)+\lambda m_i+\ef(m_i)\bigr)
\\
&= \frac1N\sum_{i\leq N} \E \log \sum_{\alpha\in \mathbb N^r}v_\alpha 
\exp \Phi_{m_i,\zeta}\bigl(1-q,g_{\xi_q',i}(\alpha)+\lambda m_i+\ef(m_i)\bigr)
\\
&= \frac1N\sum_{i\leq N} \Phi_{m_i,\zeta}\bigl(0,\lambda m_i+\ef(m_i)\bigr)
=\int\! \Phi_{a,\zeta}(0,\lambda a+\ef(a))d\mu_m(a),
\end{align*}
which finishes the proof of the first claim. The second claim is obvious because, for $\bs\in B(m,\eps),$ we have $|R(\tilde\bs,m)|<\eps$. 
\end{proof}

Next, we will show that the minimizer in (\ref{eqSec24Prop}) is finite under some assumption on $\mu$. The case where this assumption is violated will be handled differently.
\begin{lem}\label{add:lem2}
If $s=\int\! a(1-a)\,d\mu(a)>0$ then 
\begin{equation}
\lim_{|\lambda|\rightarrow\infty}\int\!\Phi_{a,\zeta}(0,\lambda a+\ef(a))\,d\mu(a)=+\infty.
\label{lem2eq1qp}
\end{equation}
If $s=\int\! a(1-a)\,d\mu(a)=0$ then 
\begin{align}
\inf_{\lambda\in\Reals}\int\!\Phi_{a,\zeta}(0,\lambda a+\ef(a))\,d\mu(a)
&=
\Phi_{0,\zeta}(0,\ef(0))\mu(\{0\}).
\end{align}
\end{lem}
\begin{proof}
Note that the boundary condition in (\ref{eqS24bc}) satisfies
$$
\Phi_{a,\zeta}(1-q,x)=\log 2-ax+\log \ch(x) \geq -ax+|x|.
$$
Moreover, $\Phi_{a,\zeta}(0,x)$ is non-decreasing in $\zeta$ and, therefore, setting $\zeta$ to be identically $0$ on $[0,1-q)$ for the lower bound and letting $g\sim N(0,\xi_q'(1-q))$, using the Feynman-Kac formula we get
\begin{align}
\Phi_{a,\zeta}(0,\lambda a+\ef(a))
&\geq \e \bigl(-a(\lambda a+\ef(a)+g)+|\lambda a+\ef(a)+g|\bigr)
\nonumber
\\
&\geq
-|\lambda| a^2-a\ef(a)+|\lambda| a - |\ef(a)| -\e |g|
\geq |\lambda|a(1-a) -c,
\label{lem2eq1qp2}
\end{align}
for some constant $c$ that depends on $\ef$ and $\xi.$ Integrating over $\mu$  finishes the proof in the case when $s>0$.

Let us now consider the case when $\mu$ is concentrated on $\{0,1\}.$ Then,
$$
\int\!\Phi_{a,\zeta}(0,\lambda a+\ef(a))\,d\mu(a)
=
\Phi_{0,\zeta}(0,\ef(0))\mu(\{0\})+\Phi_{1,\zeta}(0,\lambda+\ef(0))\mu(\{1\}).
$$
The first term does not depend on $\lambda$ and, when $a=1,$ the boundary condition
$$
\Phi_{1,\zeta}(1-q,x)=\log 2-x+\log \ch(x) = \log(1+e^{-2x})
$$
is decreasing. This means that the infimum will be achieved by letting $\lambda\to+\infty$ and, since 
$$
\lim_{\lambda\to+\infty}\Phi_{1,\zeta}(1-q,\lambda+v(1)+x)= 0
$$
for all $x$, by the monotone convergence theorem, $\lim_{\lambda\to+\infty}\Phi_{1,\zeta}(0,\lambda+v(1))= 0.$ This proves the second claim.
\end{proof}

\begin{proof}[Proof of Proposition \ref{ThSec24prop}] 
	First of all, note that $x\mapsto\Phi_{a,\zeta}(0,x)$	is a twice differentiable convex function with uniformly bounded first and second derivatives, see Subsection \ref{subphi}. As a result, for any $\mu\in \MM_{0,1}$,
	$$
	\lambda\mapsto \int\! \Phi_{a,\zeta}(0,\lambda a+\ef(a))d\mu(a)
	$$
	is twice differentiable with
	\begin{align*}
	\frac{d}{d\lambda}\int\! \Phi_{a,\zeta}(0,\lambda a+\ef(a))d\mu(a)&=\int\! a\partial_x\Phi_{a,\zeta}(0,\lambda a+\ef(a))d\mu(a),\\
	\frac{d^2}{d\lambda^2}\int\! \Phi_{a,\zeta}(0,\lambda a+\ef(a))d\mu(a)&=\int\! a^2\partial_{xx}\Phi_{a,\zeta}(0,\lambda a+\ef(a))d\mu(a).	
	\end{align*}
	In addition, the second derivative is uniformly bounded over all choices of $\lambda,$ $\ef$, and $\mu.$

We will first consider the case $s=\int\! a(1-a)\,d\mu(a)>0$.  By the previous lemma, the infimum in (\ref{eqSec24Prop}) is achieved at some finite $\lambda=\lambda_0,$ which satisfies
\begin{equation}
\int\!a\partial_x\Phi_{a,\zeta}(0,\lambda_0 a+\ef(a))\,d\mu(a) = 0.
\label{sec24propProofD}
\end{equation}
By (\ref{eqSec24PsiEcomp}) and Lemma \ref{add:lem1}, for $\gamma\geq 0,$
\begin{equation}\label{sec24propProofE}
\Psi_N(m;\lambda_0, B^{\pm}(m,\varepsilon))\leq
\int\! \Phi_{a,\zeta}(0,(\lambda_0\pm\gamma) a+\ef(a))d\mu_m(a)-\gamma\eps.
\end{equation}
When $\gamma=0$, the right hand side equals $\Psi_N(m;\lambda_0).$ Since $\Phi_{a,\zeta}(0,(\lambda+\gamma) a+\ef(a))$ is bounded and continuous in $a$ from \eqref{ex:lip} and $\mu_m\to\mu$, the right hand side converges to
$$
\int\! \Phi_{a,\zeta}(0,(\lambda_0\pm\gamma) a+\ef(a))d\mu(a)-\gamma\eps.
$$
Since the derivatives of convex functions converge to the derivative of the limit, 
the derivative of the right-hand side of \eqref{sec24propProofE} in $\gamma$ at $\gamma=0$ (denote it $D_N$) converges to 
$$
D_N\to \pm\int\!a\partial_x\Phi_{a,\zeta}(0,\lambda_0 a+\ef(a))\,d\mu(a)-\eps = -\eps,
$$
by (\ref{sec24propProofD}). Finally, from the discussion at beginning of our proof, the second derivative of the right-hand side of \eqref{sec24propProofE} in $\gamma$ is bounded uniformly over all parameters by some constant $L$. Therefore, using Taylor's formula and taking $\gamma=-D_N/L>0,$
\begin{align}
\nonumber
\Psi_N(m;\lambda_0, B^{\pm}(m,\varepsilon))
&\leq
\Psi_N(m;\lambda_0, \Sigma_N) +D_N\gamma+\frac{L\gamma^2}{2}
\\
&=
\Psi_N(m;\lambda_0, \Sigma_N) - \frac{D_N^2}{2L}
\leq
\Psi_N(m;\lambda_0, \Sigma_N) - \frac{\eps^2}{4L},
\label{eqS24st1}
\end{align}
for large $N$. 

Let us define a random analogue of (\ref{eq:Phi}),
\begin{align}
\psi_N(m;\lambda, A) 
&= \frac1N \log \sum_{\alpha\in \mathbb N^r}v_\alpha 
\sum_{\bs \in A} \exp\sum_{i\leq N} \bigl(g_{\xi_q',i}(\alpha)+\lambda m_i+\ef(m_i)\bigr)\tilde \sigma_i.
\label{eq:PhiEn}
\end{align}
Let us recall (see e.g. Chapter 2 of \cite{SKmodel}) that the weights $v_\alpha$ of the Ruelle probability cascades are defined as $w_\alpha/\sum_\alpha w_\alpha,$ where $w_\alpha$ are defined as a certain product along the path of the tree of values of Poisson processes. In particular, we can rewrite $\psi_N(m;\lambda, A)$ as
$$
\frac1N \log \sum_{\alpha\in \mathbb N^r}w_\alpha 
\sum_{\bs \in A} \exp\sum_{i\leq N} \bigl(g_{\xi_q',i}(\alpha)+\lambda m_i+\ef(m_i)\bigr)\tilde \sigma_i
-\frac1N \log \sum_{\alpha\in \mathbb N^r}w_\alpha=:R_1-R_2.
$$
By the Bolthausen-Sznitman invariance property \cite[Theorem 2.6]{SKmodel} for the Poisson processes $\Pi_t$ on $(0,\infty)$ with the mean measure $t x^{-1-t}dx$ for $t\in (0,1)$, both terms above are equal in distribution (see the proof of \cite[Theorem 2.9]{SKmodel}) to
$$
R_j=c_j+\frac1N \log \sum_{x\in \Pi_{\zeta_0}}x,
$$
for some constants $c_1,c_2$ (note, however, that the two terms are not independent of each other). This implies that
$$
\p\Bigl(\bigl|\psi_N(m;\lambda, A)-\Psi_N(m;\lambda, A)\bigr|\geq \frac{2t}{N}\Bigr)
\leq
2\p\Bigl(\bigl|\log \sum_{x\in \Pi_{\zeta_0}}x-\e \log \sum_{x\in \Pi_{\zeta_0}}x\bigr|\geq t\Bigr).
$$
In other words, the fluctuations of $\psi_N(m;\lambda, A)$ are of order $1/N$. The bound in (\ref{eqS24st1}) implies that, with high probability,
\begin{align*}
\psi_N(m;\lambda_0, B^\pm(m,\eps))
\leq
\psi_N(m;\lambda_0, \Sigma_N) - \frac{\eps^2}{8L}.
\end{align*}
Since $\Sigma_N$ is a union of $B(m,\eps)$, $B^+(m,\eps)$ and $B^-(m,\eps),$ this forces that
\begin{align*}
\psi_N(m;\lambda_0, B^+(m,\eps)\cup B^-(m,\eps))
\leq
\frac{\log 2}{N}+\psi_N(m;\lambda_0, \Sigma_N)- \frac{\eps^2}{8L},
\end{align*}
where the right-hand side used the bound $\log (a+b)\leq \log 2+\max(\log a,\log b)$ for $a,b>0.$ Consequently,
\begin{align*}
\psi_N(m;\lambda_0, \Sigma_N)+\frac{1}{N}\log \bigl(1-2e^{-\varepsilon^2N/(8L)}\bigr)&\leq\psi_N(m;\lambda_0,B(m,\eps))\leq\psi_N(m;\lambda_0, \Sigma_N),
\end{align*}
where the left-hand side used the fact that if $\frac{1}{N}\log(a+b)\leq \frac{1}{N}\log (a+b+c)-\delta$ for some $a,b,c,\delta>0$, then $(a+b)/(a+b+c)\leq e^{-N\delta}$ so that $$
 \frac{1}{N}\log (a+b+c)+\frac{1}{N}\log(1-e^{-\delta N})\leq\frac{1}{N}\log c.$$
Since $\lambda_0$ was the minimizer, by Lemma \ref{add:lem1} we get
\begin{equation}
\lim_{N\to\infty} \psi_N(m;\lambda_0, B(m,\eps))= \inf_{\lambda\in\Reals}
\int\! \Phi_{a,\zeta}(0,\lambda a+\ef(a))\,d\mu(a).
\end{equation}
Finally, using the second claim in Lemma \ref{add:lem1} and letting $\eps\to 0$ finishes the proof.

It remains to consider the case when $\mu$ is concentrated on $\{0,1\}$. By Lemma \ref{lem:Psimueps}, to compute the limit of $\Psi_N(m,\eps)$, we can choose any sequence $m=m^N$ such that $\mu_m\to\mu.$ In particular, we can choose $\mu_m$ also concentrated on $\{0,1\}.$ Let us suppose that the first $N_1$ coordinates are $m_i=0$ and the last $N-N_1$ coordinates are $m_i=1$, and $N_1/N\to \mu(\{0\}).$ 

In this case, the condition $\bs\in B(m,\eps),$ or $|R(\tilde \bs,m)|<\eps,$ means that 
$$
2\sum_{i=1}^{N} \indic(m_i=1,\sigma_i=-1)
=(m\cdot m-m\cdot\sigma) < N\eps.
$$
This means that, when $m_i=0,$ there are no constraints on $\sigma_i$ and, when $m_i=1,$ we must have $\sigma_i=1$ with at most $N\eps/2$ exceptions. This means that in the definition of $\Psi_N(m,\eps),$
we can replace the sum in the exponent by $\sum_{i\leq N_1} (g_{\xi_q',i}(\alpha)+\ef(0))\sigma_i$ (for example, using interpolation) and we can replace the constraint $\bs\in B(m,\eps)$ by  $(\sigma_1,\ldots,\sigma_{N_1})\in\Sigma_{N_1}.$ This will change $\Psi_N(m,\eps)$ by at most $O(\eps\log\eps^{-1}).$ This implies that
\begin{align*}
\Psi_N(m,\eps) &\approx \frac1N \E \log \sum_{\alpha\in \mathbb N^r}v_\alpha 
\sum_{\Sigma_{N_1}}\exp \sum_{i\leq N_1} (g_{\xi_q',i}(\alpha)+\ef(0))\tilde\sigma_i
\\
&= \frac{N_1}{N} \E \log \sum_{\alpha\in \mathbb N^r}v_\alpha 
2\ch\bigl(g_{\xi_q'}(\alpha)+\ef(0)\bigr)=\frac{N_1}{N} \Phi_{0,\zeta}(0,\ef(0)).
\end{align*}
Comparing this with the second claim in Lemma \ref{add:lem2} finishes the proof in the case when $s=0$.
\end{proof}


\section{Properties of TAP representations}\label{SecOptimal}

We present the proof of Theorem \ref{lem:TAPnew}. To prepare for our proof, Subsection \ref{relation} first establishes connections among the Parisi PDEs introduced in the introduction, while Subsection \ref{UniformBound} derives some uniform upper bounds for $\Phi_{a,\zeta}(0,\ef(a))$. The details of the proof of Theorem \ref{lem:TAPnew} are given in Subsection~\ref{sec5.3}.

\subsection{Relations among Parisi PDEs}\label{relation}

Recall that the Parisi PDE $\Phi_{\zeta}$ defined in (\ref{ParisiPDEOrig}) has the boundary condition $\Phi_\zeta(1,x)=\log 2+\log \ch x$. For $a\in \Reals$ and $\zeta\in \mathcal{M}_{0,1},$ denote by $\Theta_{a,\zeta}(t,x)$ the solution of the Parisi PDE on $[0,1]\times \Reals,$
\begin{equation}
\partial_t \Theta_{a,\zeta} = -\frac{\xi''(t)}{2}\Bigl(
\partial_{xx} \Theta_{a,\zeta} + \zeta(t)\bigl(\partial_x \Theta_{a,\zeta}\bigr)^2
\Bigr)
\label{ParisiPDEATO}
\end{equation}
with the boundary condition 
\begin{equation}
\Theta_{a,\zeta}(1,x)=\log 2-ax +\log\ch x =\log\sum_{\sigma=\pm 1}e^{(\sigma-a)x}.
\label{ParisiPDEBoundaryATO}
\end{equation}
In other words, the two solutions $\Phi_{\zeta}$ and $\Theta_{a,\zeta}$ satisfy the same PDE, but with different boundary conditions. 
The following proposition shows that these two solutions are indeed connected through an elementary formula.

\begin{prop}\label{PropATO}
	For any $x,a\in\Reals$, $t\in[0,1]$, and $\zeta\in \mathcal{M}_{0,1},$
	\begin{equation}
	\Theta_{a,\zeta}(t,x) = \Phi_{\zeta}\Bigl(t,x-a\int_t^1\!\xi''(s)\zeta(s)\,ds \Bigr)-ax +\frac{a^2}{2}\int_t^1\!\xi''(s)\zeta(s)\,ds.
	\label{eqPropATO}
	\end{equation}
\end{prop}

\begin{remark}\rm
	Recall $\Phi_{a,\zeta}$ from \eqref{ParisiPDE}. As an immediate consequence of this proposition, by shifting the measure $\zeta$ from $[0,1-q]$ to $[q,1]$, we can now unify all Parisi PDEs together.
	To see this, let $a\in [-1,1]$ and $q\in [0,1]$. For any $\zeta\in \mathcal{M}_{0,1},$ if we recall the shift operator $\theta_q$ in (\ref{eqShiftOper}) and let
$\zeta_q=\theta_q \zeta\in \mathcal{M}_{0,1-q}$, defined by $$\zeta_q(t)=\zeta(q+t),\,\,\forall t\in [0,1-q],$$ then, for any $x\in \mathbb{R},$
	\begin{align}\label{add:eq2}
	\Phi_{a,\zeta_q}(0,x)&=\Theta_{a,\zeta}(q,x)=\Phi_{\zeta}\Bigl(q,x-a\int_{q}^1\!\xi''(s)\zeta(s)\,ds \Bigr)-ax +\frac{a^2}{2}\int_{q}^1\!\xi''(s)\zeta(s)\,ds.
	\end{align}
	Note that since $\Phi_{a,\zeta_q}$ attains the global minimum at $\Psi(a,\zeta_q)$, taking infimum over $x$ leads to
	\begin{align}\label{add:eq1}
	\Phi_{a,\zeta_q}(0,\Psi(a,\zeta_q))	&=\Lambda_\zeta(q,a)-\frac{a^2}{2}\int_{q}^1\!\xi''(s)\zeta(s)\,ds,\,\,\forall a\in (-1,1).
	\end{align}
\end{remark}

\begin{proof}[Proof of Proposition \ref{PropATO}]
	It suffices to prove \eqref{eqPropATO} only for continuous $\zeta.$ Denote the right hand side of \eqref{eqPropATO} by $f(t,x)$ and set $$b(t,x)=x-a\int_t^1\!\xi''(s)\zeta(s)\,ds.$$ Then
	\begin{align*}
	\partial_tf(t,x)&=\partial_t\Phi_{\zeta}(t,b(t,x))+a\xi''(t)\zeta(t)\partial_x\Phi_{\zeta}(t,b(t,x))-\frac{a^2}{2}\xi''(t)\zeta(t)
	\end{align*}
	and
	\begin{align*}
	\partial_{x}f(t,x)&=\partial_{x}\Phi_{\zeta}(t,b(t,x))-a,\\
	\partial_{xx}f(t,x)&=\partial_{xx}\Phi_{\zeta}(t,b(t,x)).
	\end{align*}
	From this, a direct verification gives
	\begin{align*}
	\partial_t f(t,x)&=-\frac{\xi''(t)}{2}\bigl(\partial_{xx}f(t,x)+\zeta(t)\bigl(\partial_{x}f(t,x)\bigr)^2\bigr).
	\end{align*}
	Note that $f(1,x)=\log 2-ax+\log \ch x=\Theta_{a,\zeta}(1,x)$. Finally, we recall that it was proved in Jagannath-Tobasco \cite[Lemma 13]{JT16} that the classical Parisi PDE has a unique solution. The same proof therein applies to the current setting  with no essential changes and yields the uniqueness of the Parisi PDE solution with boundary condition $\log 2-ax+\log\ch x$. 
\end{proof}


\subsection{Uniform upper bounds for $\Phi_{a,\zeta}(0,\ef(a))$}\label{UniformBound} Note that
\begin{align*}
\sup_{a\in[0,1)}(1-a)\th^{-1}(a)\leq 1,\,\mbox{ and }
\lim_{a\uparrow 1}(1-a)\th^{-1}(a)=0.
\end{align*}
As a result, for any $\ef \in V\cup \oV$, $(1-a)\ef(a)$ is bounded from above by some absolute constant. Since $\Phi_{a,\zeta}$ is non-decreasing in $\zeta\in\mathcal{M}_{0,1-q}$, comparing with $\zeta\equiv 1,$ one can see that, with $c^2=\xi_q'(1-q)\leq \xi'(1),$
\begin{align*}
\Phi_{a,\zeta}(0,x)
&\leq \log\sum_{\sigma=\pm 1}\exp\bigl((\sigma-1)x+\frac{1}{2}(\sigma-a)^2c^2\bigr)
\\
&=(1-a)x+\frac{1}{2}(1-a)^2c^2+\log\bigl(1+e^{2a(c^2-x)}\bigr).
\end{align*}
In particular, for some constant $c_\xi,$ for all $\ef \in V\cup \oV$,
\begin{align}
\Phi_{a,\zeta}(0,\ef (a)) &\leq 
(1-a)\ef(a)+\frac{1}{2}(1-a)^2c^2+\log\bigl(1+e^{2a(c^2-\ef(a))}\bigr)
\leq c_\xi.
\label{lem:Tapcont:proof:eq5}
\end{align}
Moreover, if we take $\ef(a)=\ef_{\zeta'}(a)=\Psi(a,\zeta')$ for some $\zeta'\in \mathcal{M}_{q'}$ and use Lemma \ref{ex:lem2} below, which states that
$$
c_1'+c_2'\tanh^{-1} (a)\leq \ef_{\zeta'}(a) \leq c_1+c_2\tanh^{-1}(a)
$$
for some absolute constants $c_1,c_1'\in\Reals,$ $c_2,c_2'>0$, we also have
\begin{align}
0\leq\Phi_{a,\zeta}(0,\ef_{\zeta'}(a))
&\leq
(1-a)(c_1+c_2\tanh^{-1}(a))+\frac{1}{2}(1-a)^2c^2
\nonumber
\\
&
+\log \bigl(1+\exp \bigl(2ac^2-2c_1'a-2c_2'a\tanh^{-1} (a)\bigr)\bigr)
=:M(a),
\label{lem:Tapcont:proof:eq6}
\end{align}
We can see that $\lim_{a\uparrow 1}M(a)=0,$ which shows that $\Phi_{a,\zeta}(0,\ef_{\zeta'}(a))$ is small in the neighbourhood of $1$ uniformly over the choice of $\zeta$ and $\zeta'.$ Using (\ref{add:eq1}), this implies that
\begin{align}\label{add:eq1Ago}
0\leq \Lambda_\zeta(q,a)-\frac{a^2}{2}\int_{q}^1\!\xi''(s)\zeta(s)\,ds \leq M(|a|),
\end{align}
so the expression in the middle goes to zero as $a\to \pm 1$, uniformly over $q$ and $\zeta$.

\subsection{Proof of Theorem \ref{lem:TAPnew}}\label{sec5.3} Our approach relies on the convexity of the Parisi functional $\oP_\mu^\ef(\cdot,\cdot)$ on the band (defined in \eqref{add:eq10}) as well as some computations on the directional derivative of this functional. Let $\mu$ be a probability measure on $[0,1].$

\begin{proof}[Proof of $(i)$] First we establish the first equality in $(i).$ To see this, observe that \eqref{add:eq1} and \eqref{lem:Tapcont:proof:eq6} together imply that whenever $\zeta\in \mathcal{M}_{0,1}$ and $\zeta_q=\theta_q\zeta\in \mathcal{M}_{0,1-q}$ satisfy $\zeta_q(t)=\zeta(q+t)$ for all $t\in [0,1-q]$, we have 
	\begin{align*}
	\Lambda_{\zeta}(q,1)=\lim_{a\to 1}\Lambda_{\zeta}(q,a)=\frac{1}{2}\int_q^{1}\xi''(s)\zeta(s)ds.
	\end{align*}
	Using \eqref{add:eq1} and this equation yield that
	\begin{align*}
	\oP_\mu^{\ef_{\zeta_q}}(0,\zeta_q)&=\int_{[0,1)}\Bigl(\Lambda_\zeta(q,a)-\frac{a^2}{2}\int_{q}^1\!\xi''(s)\zeta(s)\,ds\Bigr)d\mu(a)-\frac{1}{2}\int_q^1(s-q)\xi''(s)\zeta(s)ds\\
	&=\int_{[0,1]}\Bigl(\Lambda_\zeta(q,a)-\frac{a^2}{2}\int_{q}^1\!\xi''(s)\zeta(s)\,ds\Bigr)d\mu(a)-\frac{1}{2}\int_q^1(s-q)\xi''(s)\zeta(s)ds=\nTAP(\mu,\zeta),
	\end{align*}
	which implies that
	\begin{align*}
	\inf_{\zeta\in\mathcal{M}_{0,1-q}}\oP_\mu^{\ef_{\zeta}}(0,\zeta)=\inf_{\zeta\in \mathcal{M}_{0,1}}\oP_\mu^{\ef_{\zeta_q}}(0,\zeta_q)=\inf_{\zeta\in \mathcal{M}_{0,1}}\nTAP(\mu,\zeta)
	=\inf_{\zeta\in \mathcal{M}_{q,1}}\nTAP(\mu,\zeta)=\nTAP(\mu).
	\end{align*}
	This establishes one of the equalities in $(i)$. For the rest of the equalities, they follow immediately if the following claim is valid: for any $\zeta\in \mathcal{M}_{0,1-q},$
	\begin{align}\label{lem:TAPcont:proof:eq2}
	\begin{split}
	\inf_{\ef \in V}\int_{[0,1]} \Phi_{a,\zeta}(0,\ef(a))\,d\mu(a)
	&=\int_{[0,1)} \Phi_{a,\zeta}(0,\Psi(a,\zeta))\,d\mu(a)\\
	&=\inf_{\ef \in \text{\oV}}\int_{[0,1)} \Phi_{a,\zeta}(0,\ef(a))\,d\mu(a).
	\end{split}
	\end{align}
	To show this claim, observe that for any $a\in [0,1),$ $\Phi_{a,\zeta}(0,\cdot)$ is a strictly convex function (see \eqref{convexity0}) and has a unique global minimum at $\Psi(a,\zeta)$ since $$\partial_x\Phi_{a,\zeta}(a,\Psi(a,\zeta))=0.$$ These imply that
	\begin{align}\label{add:eq16}
	\Phi_{a,\zeta}(a,\Psi(a,\zeta))\leq \Phi_{a,\zeta}(a,x),\,\,\forall x\in \mathbb{R}.
	\end{align}
	From this, Lemma \ref{ex:lem2}, and noting that $\Phi_{a,\zeta}$ is always nonnegative, we have 
	\begin{align}
	\begin{split}\label{add:eq13}
	\int_{[0,1)} \Phi_{a,\zeta}(0,\Psi(a,\zeta))\,d\mu(a)&\leq \inf_{\ef \in V}\int_{[0,1]} \Phi_{a,\zeta}(0,\ef(a))\,d\mu(a),\\
	\int_{[0,1)} \Phi_{a,\zeta}(0,\Psi(a,\zeta))\,d\mu(a)&=\inf_{\ef \in \text{\oV}}\int_{[0,1)} \Phi_{a,\zeta}(0,\ef(a))\,d\mu(a).
	\end{split}
	\end{align}
	The second line here gives the second equality of \eqref{lem:TAPcont:proof:eq2}. To show the first equality of \eqref{lem:TAPcont:proof:eq2}, it remains to establish the reverse inequality for the first equation of \eqref{add:eq13}. To this end, for $L>0,$ set  $\ef _L\in V$  by  $\ef _L(a)=\Psi(a,\zeta)\wedge L$ for $a\in [0,1)$ and $\ef_L(1)=L.$ Write
	\begin{align*}
	\int_{[0,1]} \Phi_{a,\zeta}(0,\ef_L(a))\,d\mu(a)&=\int_{[0,1)} \Phi_{a,\zeta}(0,\ef_L(a))\,d\mu(a)+\Phi_{1,\zeta}(0,L)\mu(\{1\}).
	\end{align*}
	Passing to the limit via the bound \eqref{lem:Tapcont:proof:eq5} and the dominated convergence theorem gives that
	\begin{align*}
	&\limsup_{L\rightarrow\infty}\int_{[0,1]} \Phi_{a,\zeta}(0,\ef_L(a))\,d\mu(a)\\
	&\leq \int_{[0,1)} \Phi_{a,\zeta}(0,\Psi(a,\zeta))\,d\mu(a)+\limsup_{L\rightarrow\infty}\Phi_{1,\zeta}(0,L)\mu(\{1\}).
	\end{align*}
	Note that $\Phi_{1,\zeta}(1,x)=\log (1+e^{-2x})$ is a strictly decreasing function with $\Phi_{1,\zeta}(1,\infty)=0$. These properties are also valid for $\Phi_{1,\zeta}(0,\cdot)$, which can be seen from the representation \eqref{control}. Hence, $\limsup_{L\rightarrow\infty}\Phi_{1,\zeta}(0,L)=0$ and consequently,
	\begin{align*}
	\inf_{\ef \in V}\int_{[0,1]} \Phi_{a,\zeta}(0,\ef(a))\,d\mu(a)&\leq\limsup_{L\rightarrow\infty}\int_{[0,1]} \Phi_{a,\zeta}(0,\ef_L(a))\,d\mu(a)\\
	&\leq \int_{[0,1)} \Phi_{a,\zeta}(0,\Psi(a,\zeta))\,d\mu(a).
	\end{align*}
	This completes the proof of our claim.
\end{proof}

\begin{proof}[Proof of $(ii)$]
By part $(i)$, we can work with any of the four representations. We will use the first one, $\nTAP(\mu)$. If we denote
$$
D_\zeta(q,a):= \Lambda_\zeta(q,a)-\frac{a^2}{2}\int_{q}^1\!\xi''(s)\zeta(s)\,ds
$$
then, for $\mu$ with $\int a^2d\mu(a)=q,$ we can rewrite
$$
\nTAP(\mu,\zeta):=
\int\! D_{\zeta}(q,a)\,d\mu(a)-\frac{1}{2}\int_q^{1}\!(s-q)\xi''(s)\zeta(s)\,ds.
$$
By (\ref{add:eq1Ago}), for any $\eps>0$, we can find $\eta\in (0,1)$ such that, $0\leq D_\zeta(q,a)\leq \eps$ for $|a|\geq \eta.$ Since, $\Lambda_\zeta(q,a)$ and $D_\zeta(q,a)$ are even (and concave), if we let
$$
\nTAP^\eta(\mu,\zeta):=
\int\! D_\zeta(q,|a|\wedge \eta)\,d\mu(a)-\frac{1}{2}\int_q^{1}\!(s-q)\xi''(s)\zeta(s)\,ds
$$
then 
\begin{equation}
\bigl| \nTAP^\eta(\mu,\zeta)-\nTAP(\mu,\zeta) \bigr| \leq \eps. 
\label{eqTAPapprxETA}
\end{equation}
Using the fact that $\partial_x\Phi_{\zeta}(q,\cdot)$ is strictly increasing with $\partial_x\Phi_{\zeta}(q,\pm\infty)=\pm 1$, it is easy to check that, for $|a|<\eta$, the infimum in $\Lambda_\zeta(q,a)=\inf_{x\in\Reals}(\Phi_\zeta(q,x)-ax)$ is achieved on $x\in [-L,L]$, where $L$ depends on $\eta$ only. This implies that $\Lambda_\zeta(q,a)$ and $D_\zeta(q,a)$ are Lipschitz on $[-\eta,\eta]$ and $D_\zeta(q,|a|\wedge \eta)$ is  Lipschitz on $[-1,1]$, uniformly over $q$ and $\zeta$. Therefore,
$$
\sup_{q,\zeta}\Bigl|
\int\! D_\zeta(q,|a|\wedge \eta)\,d\mu(a)-\int\! D_\zeta(q,|a|\wedge \eta)\,d\mu'(a)
\Bigr| \leq Cd_{\mathrm{BL}}(\mu,\mu'),
$$
where $d_{\mathrm{BL}}(\mu,\mu')$ is the bounded Lipschitz metric on $\MM_{0,1}.$ Moreover, since $\Phi_\zeta(q,x)$ is Lipschitz in $q$ and $d_1$-Lipschitz in $\zeta$ (with Lipschitz constants that depend only on $\xi$), these properties are inherited by $\Lambda_\zeta(q,a)$, $D_\zeta(q,a)$ and $D_\zeta(q,|a|\wedge \eta)$, and, therefore, $\nTAP^\eta(\mu,\zeta).$ If $\mu\to\mu_0$ then $q=\int a^2d\mu(a)\to q_0=\int a^2d\mu_0(a)$, and all the properties above imply that $\inf_\zeta \nTAP^\eta(\mu,\zeta)$ is continuous in $\mu$. By (\ref{eqTAPapprxETA}), this proves that $\nTAP(\mu)$ is continuous in $\mu$.
\end{proof}

From now on we can assume that $\mu\neq\delta_1,$ because all the remaining claims are trivial in this case.

\begin{proof}[Proof of $(iii)$]
	Let $\zeta_n\in \mathcal{M}_{0,1-q}$ be a sequence that weakly converges to some $\zeta_0$ and satisfies 
	\begin{align*}
	\inf_{\zeta\in \mathcal{M}_{0,1-q}}\oP_\mu^{\ef_\zeta}(0,\zeta)=\lim_{n\rightarrow\infty}\oP_\mu^{\ef_{\zeta_n}}(0,\zeta_n).
	\end{align*}
	For any $\delta\in (0,1)$, write
	\begin{align*}
	\int_{[0,1)}\Phi_{a,\zeta_n}(0,\Psi(a,\zeta_n))d\mu(a)&=\int_{[0,\delta]}\Phi_{a,\zeta_n}(0,\Psi(a,\zeta_n))d\mu(a)\\
	&+\int_{(\delta,1)}\Phi_{a,\zeta_n}(0,\Psi(a,\zeta_n))d\mu(a).
	\end{align*}
	Here, the first term converges to $$\int_{[0,\delta]}\Phi_{a,\zeta_0}(0,\Psi(a,\zeta_0))d\mu(a)$$ as a consequence of Lemma \ref{ex:lem1} and the fact that $\Phi_{a,\zeta_n}(t,x)$ converges to $\Phi_{a,\zeta_0}(t,x)$ uniformly over all $a\in [0,1],t\in [0,1-q],x\in \mathbb{R}.$ As for the second term, note that the inequality \eqref{lem:Tapcont:proof:eq6} implies 
	\begin{align*}
	0&\leq\int_{(\delta,1)}\Phi_{a,\zeta_n}(0,\Psi(a,\zeta_n))d\mu(a)\leq \sup_{\delta<a<1}M(a)\rightarrow 0
	\end{align*}
	as $\delta\uparrow 1.$ Hence, we arrive at
	\begin{align*}
	\lim_{n\rightarrow\infty}\int_{[0,1)}\Phi_{a,\zeta_n}(0,\Psi(a,\zeta_n))d\mu(a)=\int_{[0,1)}\Phi_{a,\zeta_0}(0,\Psi(a,\zeta_0))d\mu(a)
	\end{align*}
	and 
	$$
	\lim_{n\rightarrow\infty}\oP_\mu^{\ef_{\zeta_n}}(0,\zeta_n)=\oP_\mu^{\ef_{\zeta_0}}(0,\zeta_0),
	$$
	which shows that $\zeta_0$ is a minimizer of $\inf_{\zeta\in \mathcal{M}_{0,1-q}}\oP_\mu^{\ef_\zeta}(0,\zeta)$.
	
	Next, we establish the uniqueness of $\zeta_0$. Assume that $\zeta_1$ is another minimizer and $\zeta_1\not\equiv\zeta_0.$ For $b\in (0,1)$, let 
	$$
	\zeta_b = (1-b)\zeta_0+b \zeta_1,\,\,
	\ef_b=(1-b)\ef_{\zeta_0}+b\ef_{\zeta_1}.
	$$ 
	Since $\mu\neq \delta_1$, we can use the strict convexity in $\eqref{convexity}$ below to get	
	\begin{align*}
	\int_{[0,1)} \Phi_{a,\zeta_b}(0,\ef_b(a))d\mu(a)
	&< (1-b)\int_{[0,1)} \Phi_{a,\zeta_0}(0,\ef_{\zeta_0}(a))d\mu(a)\\
	&\quad+b\int_{[0,1)} \Phi_{a,\zeta_1}(0,\ef_{\zeta_1}(a))d\mu(a).
	\end{align*}
	On the other hand, note that, by \eqref{add:eq16}, 
	\begin{align*}
	\int_{[0,1)}\Phi_{a,\zeta_b}(0,\Psi(a,\zeta_b))d\mu(a)&\leq \int_{[0,1)} \Phi_{a,\zeta_b}(0,\ef_b(a))d\mu(a).
	\end{align*} 
	This and the above inequality together lead to a contradiction,
	\begin{align}
	\label{add:eq5}
	\inf_{\zeta\in \mathcal{M}_{0,1-q}}\oP_\mu^{\ef_\zeta}(0,\zeta)&<(1-b)\inf_{\zeta\in \mathcal{M}_{0,1-q}}\oP_\mu^{\ef_\zeta}(0,\zeta)+b\inf_{\zeta\in \mathcal{M}_{0,1-q}}\oP_\mu^{\ef_\zeta}(0,\zeta)=\inf_{\zeta\in \mathcal{M}_{0,1-q}}\oP_\mu^{\ef_\zeta}(0,\zeta).
	\end{align}
	Hence, the minimizer must be unique when $\mu\neq \delta_1$.	 
\end{proof}

\begin{proof}[Proof of $(iv)$] Note that the minimality of $\zeta_0$ implies 
	\begin{align*}
	\inf_{\zeta\in \mathcal{M}_{0,1-q}}\oP_\mu^{\ef_\zeta}(0,\zeta)=\oP_\mu^{\ef_{\zeta_0}}(0,\zeta_0)\geq \inf_{\zeta\in \mathcal{M}_{0,1-q}}\oP_\mu^{\ef_{\zeta_0}}(0,\zeta)\geq\inf_{\lambda\in \mathbb{R},\zeta\in \mathcal{M}_{0,1-q}}\oP_\mu^{\ef_{\zeta_0}}(\lambda,\zeta)
	=\oP_\mu^{\ef_{\zeta_0}}.
	\end{align*}
	Also note that from \eqref{add:eq16},  for any $\zeta\in \mathcal{M}_{0,1-q}$ and $\lambda\in \mathbb{R}$,
	\begin{align*}
	\oP_\mu^{\ef_{\zeta}}(0,\zeta)&\leq 	\oP_\mu^{\ef_{\zeta_0}}(\lambda,\zeta),
	\end{align*}
	which leads to
	\begin{align*}
	\oP_\mu^{\ef_{\zeta_0}}(0,\zeta_0)=
	\inf_{\zeta\in \mathcal{M}_{0,1-q}}\oP_\mu^{\ef_\zeta}(0,\zeta)\leq
	\inf_{\lambda\in \mathbb{R},\zeta\in \mathcal{M}_{0,1-q}}\oP_\mu^{\ef_{\zeta_0}}(\lambda,\zeta)=\oP_\mu^{\ef_{\zeta_0}}.
	\end{align*}
	This completes our proof.
\end{proof}

\begin{proof}[Proof of $(v)$]
	
	Let $\delta_1\neq \mu\in \MM_{0,1}.$ Let $\zeta_0$ be the minimizer from part $(iii)$. By part $(iv)$, the pair $(0,\zeta_0)$ is a minimizer of $\oP_\mu^{\ef_{\zeta_0}}(\lambda,\zeta)$. Also, by the definition of $\ef_{\zeta_0}$ in (\ref{add:eq12}) and (\ref{eqDefPsief}), we have $\partial_{x}\Phi_{a,\zeta_0}(0,\ef_{\zeta_0}(a))=0$ for all $a\in [0,1)$. 
	
	We follow a similar argument as \cite[Theorem 1]{Chen13} (see also \cite[Lemma 4.14]{SKbonus}). Let $c$ be the smallest point in the support of $\zeta_0.$ Assume on the contrary that $c>0.$ Note that from the optimality of $\zeta_0$ in $\oP_\mu^{\ef_{\zeta_0}}(0,\cdot)$, Remark \ref{rmk2} below states that
	\begin{align}\label{add:prop0:proof:eq2Ag}
	\begin{split}
	\int_{[0,1)} \e\bigl(\partial_x\Phi_{a,\zeta_0}(c,\ef_{\zeta_0}(a)+z(c))\bigr)^2d\mu(a)&=c,
	\end{split}\\
	\begin{split}\label{add:prop0:proof:eq2}
	\xi_q''(c)\int_{[0,1)} \e\bigl(\partial_{xx}\Phi_{a,\zeta_0}(c,\ef_{\zeta_0}(a)+z(c))\bigr)^2d\mu(a)&\leq 1,
	\end{split}
	\end{align}
	where $z(c)$ is a centered normal random variable with variance $\xi_{q}'(c)$. Define an auxiliary function $A:[0,c]\to [0,c]$ by
	\begin{align*}
	A(t)&=\int_{[0,1)}\e\bigl(\partial_x\Phi_{a,\zeta_0}(c,\ef_{\zeta_0}(a)+z_1(t)\bigr)\bigl(\partial_x\Phi_{a,\zeta_0}(c,\ef_{\zeta_0}(a)+z_2(t)\bigr)d\mu(a),
	\end{align*}
	where $z_1(t)$ and $z_2(t)$ are jointly Gaussian random variables with mean zero and variance $\e (z_1(t))^2=\e (z_2(t))^2=\xi_q'(c)$ and $\e z_1(t)z_2(t)=\xi_q'(t).$ From this construction and (\ref{add:prop0:proof:eq2Ag}), evidently $A(c)=c$. In addition, since $\partial_x\Phi_{a,\zeta_0}(0,\ef_{\zeta_0}(a))\equiv 0$, we also have that
	\begin{align*}
	A(0)&=\int_{[0,1)} \bigl(\e \partial_x\Phi_{a,\zeta_0}(c,\ef_{\zeta_0}(a)+z(c))\bigr)^2d\mu(a)\\
	&=\int_{[0,1)} \bigl(\e \partial_x\Phi_{a,\zeta_0}(0,\ef_{\zeta_0}(a))\bigr)^2d\mu(a)=0,
	\end{align*}
	where the second equality holds because $\zeta_{0}(s)= 0$ for $s\in[0,c)$ and
	$$
	\Phi_{a,\zeta_0}(0,x)=\e\Phi_{a,\zeta_{0}}(c,x+z(c)).
	$$
	Next, a direct differentiation using Gaussian integration by parts, the bounds in \eqref{subphi:eq1}, and the dominated convergence theorem gives
	\begin{align*}
	A'(t)&=\xi_q''(t)\int_{[0,1)} \e\bigl(\partial_{xx}\Phi_{a,\zeta_0}(c,\ef_{\zeta_0}(a)+z_1(t)\bigr)\bigl(\partial_{xx}\Phi_{a,\zeta_0}(c,\ef_{\zeta_0}(a)+z_2(t)\bigr)d\mu(a),
	\end{align*}
	from which we see that  $$0\leq A'(t)<A'(c)\leq 1$$ for all $t\in [0,c)$, where the first inequality is obtained by using conditional expectation and integrating the independent components of $z_1(t)$ and $z_2(t)$ first,  second inequality is by the Cauchy-Schwarz inequality, and third inequality follows by \eqref{add:prop0:proof:eq2}. This contradicts that both $A(0)=0$ and $A(c)=c$. Hence, the smallest point in the support of $\zeta_0$ must be zero and this completes our proof.
\end{proof}

\begin{proof}[Proof of $(vi)$]	
	We show that if  $\oP_\mu^{\ef_{\zeta_1}}(0,\zeta_1) = \inf_{\lambda,\zeta}\oP_\mu^{\ef_{\zeta_1}}(\lambda,\zeta)$ then $\zeta_1=\zeta_0.$ It suffices to show that 
	\begin{align}
	\label{add:eq4}
	\oP_\mu^{\ef_{\zeta_1}}(0,\zeta_1)=\oTAP(\mu)=\inf_{\ef\in \oV,\,\zeta\in \mathcal{M}_{0,1-q}}\oP_\mu^\ef(0,\zeta).
	\end{align} 
	Indeed, if this holds then, from part $(i)$,
	$$
	\oP_\mu^{\ef_{\zeta_1}}(0,\zeta_1)=\inf_{\zeta\in \mathcal{M}_{0,1-q}}\oP_\mu^{\ef_{\zeta}}(0,\zeta).
	$$
	Then, by part $(iii)$ and the assumption that $\mu\neq \delta_1$, we get $\zeta_1=\zeta_0.$ To prove \eqref{add:eq4}, we can argue as follows. Notice that from \eqref{convexity},
	$$
	(\ef,\zeta)\in \oV\times \mathcal{M}_{0,1-q}\mapsto \oP_\mu^{\ef}(0,\zeta)
	$$ 
	is a convex function. From the minimality of $(0,\zeta_1)$ in $\oP^{\ef_{\zeta_1}}_\mu(\lambda,\zeta)$,
	\begin{align*}
	\oP_\mu^{\ef_{\zeta_1}}(0,\zeta_1)\leq \oP_\mu^{\ef_{\zeta_1}}(0,\zeta),\,\,\forall \zeta\in \mathcal{M}_{0,1-q}
	\end{align*}
	and, by \eqref{add:eq16},
	\begin{align*}
	\oP_\mu^{\ef_{\zeta_1}}(0,\zeta_1)\leq \oP_\mu^{\ef}(0,\zeta_1),\,\,\forall \ef \in \oV.
	\end{align*}
	In other words, $(\ef_{\zeta_1},\zeta_1)$ is a local minimum of $\oP_\mu^{\ef}(0,\zeta)$ in the two coordinates $\ef$ and $\zeta$ separately.	
	From Lemma \ref{lem:dd} below, for any $\zeta\in \mathcal{M}_{0,1-q}$ and $\ef\in \oV$ satisfying $\ef=\ef_{\zeta_1}$ on $[\delta,1)$ for some $\delta\in (0,1),$ the directional derivative of $(\ef,\zeta)\to\oP_\mu^{\ef}(0,\zeta)$ exists along the direction from $(\ef_{\zeta_1},\zeta_1)$ to $(\ef,\zeta)$ and is equal to \eqref{lem:dd:eq1}. From this, it can be checked that 
	\begin{align*}
	&\frac{d}{db}\oP_\mu^{(1-b)\ef_{\zeta_1}+b\ef}(0,(1-b)\zeta_1+b\zeta)\Big|_{b=0^+}\\
	&=\frac{d}{db}\oP_\mu^{\ef_{\zeta_1}}(0,(1-b)\zeta_1+b\zeta)\Big|_{b=0^+}+\frac{d}{db}\oP_\mu^{(1-b)\ef_{\zeta_1}+b\ef}(0,\zeta_1)\Big|_{b=0^+}\geq 0,
	\end{align*}
	where the last inequality is a consequence of the previous two displays. With this, for any $\varepsilon>0,$ there exists some small $b>0$ such that
	\begin{align*}
	\oP_\mu^{\ef_{\zeta_1}}(0,\zeta_1)-b\varepsilon\leq\oP_\mu^{(1-b)\ef_{\zeta_1}+b\ef}(0,(1-b)\zeta_1+b\zeta).
	\end{align*}
	On the other hand, using the convexity on the right-hand side yields
	\begin{align*}
	\oP_\mu^{(1-b)\ef_{\zeta_1}+b\ef}(0,(1-b)\zeta_1+b\zeta)&\leq (1-b)\oP_\mu^{\ef_{\zeta_1}}(0,\zeta_1)+b\oP_\mu^{\ef}(0,\zeta).
	\end{align*}
	Putting these two inequalities together gives
	$
	\oP_\mu^{\ef_{\zeta_1}}(0,\zeta_1)\leq \varepsilon+\oP_\mu^{\ef}(0,\zeta),
	$
	and letting $\varepsilon\downarrow 0$, we get 
	\begin{align}\label{ex:eq3}
	\oP_\mu^{\ef_{\zeta_1}}(0,\zeta_1)\leq  \mathcal{P}_\mu^{\ef}(0,\zeta). 
	\end{align}
	Note that we proved this for $\ef\in \oV$ satisfying $\ef =\ef_{\zeta_1}$ on $[\delta,1)$ for any $\delta\in (0,1).$ In what follows, we show that this implies the same inequality for all $\ef\in V$ and $\zeta\in \mathcal{M}_{0,1-q}.$ 
	
	For any $\delta\in (1/3,1),$ let $a(\delta):=(3\delta-1)/2.$ Then $a(\delta)\in (0,\delta).$ For any $\ef\in V$, since $\ef$ is bounded on $[0,1]$, we can construct $\ef_\delta\in \oV$ so that $\ef_\delta=\ef$ on $[0,(3\delta-1)/2]$ and $\ef_\delta=\ef_{\zeta_1}$ on $[\delta,1)$ as long as $\delta$ is sufficiently close to $1.$ From this, write
	\begin{align*}
	&\int_{[0,1)}\Phi_{a,\zeta}(0,\ef_\delta(a))\mu(da)-\int_{[0,a(\delta))}\Phi_{a,\zeta}(0,\ef(a))\mu(da)\\
	&=\int_{[a(\delta),1)}\Phi_{a,\zeta}(0,\ef_\delta(a))\mu(da)\\
	&=\int_{[a(\delta),\delta)}\Phi_{a,\zeta}(0,\ef_\delta(a))\mu(da)+\int_{[\delta,1)}\Phi_{a,\zeta}(0,\ef_{\zeta_1}(a))\mu(da)
	\end{align*}
	and use the bounds \eqref{lem:Tapcont:proof:eq5} and \eqref{lem:Tapcont:proof:eq6} to get
	\begin{align*}
	\Bigl|\int_{[0,1)}\Phi_{a,\zeta}(0,\ef_\delta(a))\mu(da)-\int_{[0,a(\delta))}\Phi_{a,\zeta}(0,\ef(a))\mu(da)\Bigl|&\leq c_\xi \mu([a(\delta),\delta))+M(\delta),
	\end{align*}
	Here the second term vanishes as $\delta \to 1^-$. The first term can be handled as follows.
	Note that $\mu([s,1))$ is a nonincreasing function, so $\lim_{s\to 1^-}\mu([s,1))$ exists. This implies that 
	\begin{align*}
	\lim_{\delta\to 1^-}\mu([a(\delta),\delta))=\lim_{\delta\to 1^-}\mu([a(\delta),1))-\lim_{\delta\to 1^-}\mu([\delta,1))=0.
	\end{align*}
	Hence,
	\begin{align*}
	\lim_{\delta\to 1^-}\Bigl|\int_{[0,1)}\Phi_{a,\zeta}(0,\ef_\delta(a))\mu(da)-\int_{[0,a(\delta))}\Phi_{a,\zeta}(0,\ef(a))\mu(da)\Bigl|=0
	\end{align*}
	and, starting with (\ref{ex:eq3}) for $\ef_\delta$,
	\begin{align*}
	\oP_\mu^{\ef_{\zeta_1}}(0,\zeta_1)\leq  \lim_{\delta \to 1^-}\oP_\mu^{\ef_\delta}(0,\zeta)=\oP_\mu^{\ef'}(0,\zeta)\leq \mathcal{P}_\mu^{\ef}(0,\zeta),
	\end{align*}
	where $\ef'$ is the restriction of $\ef$ on $[0,1).$ 
	This establishes \eqref{ex:eq3} for all $\ef\in V.$ Now from $(i)$, we see that $\zeta_1$ minimizes $$\inf_{\zeta\in \mathcal{M}_{0,1-q}}\oP_\mu^{\ef_\zeta}(0,\zeta)$$ and from $(iii)$, $\zeta_0=\zeta_1.$ This finishes the proof. 
\end{proof}

\section{Optimizing over the external field}\label{SecOptimizingEF}

In this section, we will prove the upper and lower bounds of Lemmas \ref{lem:TAP+} and \ref{lem:TAP-} on the limiting replicated free energy on the band. 
We will use the representation $\TAP(\mu)$ of \eqref{eq:TAPmu}, which by Theorem \ref{lem:TAPnew} $(i)$ is equivalent to $\nTAP(\mu)$.
The upper bound of Lemma \ref{lem:TAP+} will be straightforward to prove, by introducing an arbitrary continuous external field and then applying the Guerra upper bound. The lower bound contains the key step, where will need to use the optimal external field found in the last section (for which the Parisi measure has zero in the support).

\subsection{Proof of Lemma \ref{lem:TAP+}}

Notice that, since
\begin{equation*}
n\e \nTAP_{N,n}(m,\eps,\delta)\,\mbox{ is increasing in $\eps$ and $\delta$ and sub-additive in $n$,}
\end{equation*}
the upper limit in \eqref{eq:lemTAP+} is increasing in $\eps$ and $\delta$ and decreasing in $n$, so the infimum over $\eps,\delta,n$ can be replaced by the limit $\eps,\delta\downarrow 0$ and $n\uparrow\infty.$ Also, we can always choose $\eps=\eps_N$ and $\delta=\delta_N$ going to zero and $n=n_N$ going to infinity slowly enough so that
\begin{equation}
\inf_{\eps,\delta,n}\limsup_{N\to\infty}\e \nTAP_{N,n}(m,\eps,\delta)
=
\limsup_{N\to\infty}\e \nTAP_{N,n_N}(m,\eps_N,\delta_N).
\end{equation}
Using this representation and the equivalence $\nTAP(\mu)
=\TAP(\mu)$ in Theorem \ref{lem:TAPnew},  to prove the lemma we need to show that
\begin{equation}
\label{eq:rev3}
\limsup_{N\to\infty}\e \nTAP_{N,n_N}(m,\eps_N,\delta_N)\leq\TAP(\mu).
\end{equation}

First, consider the case $q=\int\! x^2\,\mu(dx)>0.$ Using the approximation in Lemma \ref{LemHtoHm} and \eqref{eq:TAPmAg2}, it is enough to prove that
$$
\limsup_{N\to\infty}\e F_{N,n_N}(m,\eps_N,\delta_N)\leq \TAP(\mu).
$$
For any $\ef\in C([0,1])$, let $h_i = \ef(m_i)$ be the corresponding external field. Since $\|h\|/\sqrt{N}\leq \|\ef\|_\infty<\infty$, by (\ref{eqFtoFh}), deterministically,
\begin{equation}
\lim_{N\to\infty}\bigl|F_{N,n_N}(m,\eps_N,\delta_N)-F_{N,n_N}^h(m,\eps_N,\delta_N)\bigr| = 0.
\label{eqFtoFhAg}
\end{equation}
Using that $F_{N,n_N}^h(m,\eps_N,\delta_N)\leq F_{N,1}^h(m,\eps_N)$, by Theorem \ref{thm1},
\begin{align*}
	&
	\limsup_{N\to\infty} \e F_{N,n_N}(m,\eps_N,\delta_N)
	\leq 
	\limsup_{N\to\infty} \e  F_{N,1}^h(m,\eps_N)
	\leq  \PP^\ef_\mu .
\end{align*}
Taking infimum over $\ef\in V$ yields the assertion.

In the case when $q=\int\! x^2\,\mu(dx)=0,$ we can use the approximation in (\ref{eq:TAPmAg22}) and, in this case, Theorem \ref{thm1ag2} implies the claim without the need to introduce any external field, because the Parisi formula in Theorem \ref{thm1ag2} equals $\TAP(\delta_0).$\qed

\subsection{Proof of Lemma \ref{lem:TAP-}}
Similarly to \eqref{eq:rev3}, to prove the lemma we need to show that
\begin{equation*}
	\liminf_{N\to\infty}\e \nTAP_{N,n_N}(m,\eps_N,\delta_N)\geq\TAP(\mu),
\end{equation*}
where $\eps_N,\delta_N$ go to zero and $n_N$ goes to infinity slowly enough.

Again, first, consider the case $q=\int\! x^2\,\mu(dx)>0.$ Using the approximation in Lemma \ref{LemHtoHm} and \eqref{eq:TAPmAg2}, it is enough to prove that
$$
\liminf_{N\to\infty}\e F_{N,n_N}(m,\eps_N,\delta_N)\geq \TAP(\mu).
$$
Consider the external field $h$ defined through
\begin{equation}
h_i=\ef_{\zeta_0}(m_i)=\Psi\bigl(m_i,\zeta_{0}\bigr)
\label{eqFPhag}
\end{equation}
where $\zeta_0$ is the minimizer found in Theorem \ref{lem:TAPnew}. Since $\ef$ is bounded on $[0,1-\eta]$ and all $m_i\in [0,1-\eta],$ by (\ref{eqFtoFh}),
$$
\liminf_{N\rightarrow\infty}\e F_{N,n_N}(m,\eps_N,\delta_N)=
\liminf_{N\rightarrow\infty}\e F_{N,n_N}^h(m,\eps_N,\delta_N).
$$
The Parisi formula in Theorem \ref{thm1} implies that
\begin{align*}
\lim_{N\rightarrow\infty}\e F_{N,1}^{h}(m,\varepsilon_N,\delta_N)= \inf_{\lambda,\zeta}\PP^{\ef_{\zeta_0}}_\mu(\lambda,\zeta)=\inf_{\lambda,\zeta}\oP^{\ef_{\zeta_0}}_\mu(\lambda,\zeta),
\end{align*}
where the second equality holds because $\supp(\mu)\subseteq [0,1-\eta]$ (so the functionals $\PP$ and $\oP$ coincide). By our choice of $\zeta_0$, Theorem \ref{lem:TAPnew} $(iv)$ implies that the right hand side equals $\oP_\mu^{\ef_{\zeta_0}}$, which equals to $\TAP(\mu)$ by Theorem \ref{lem:TAPnew} $(i)$. Furthermore, since zero is in the support of $\zeta_0$ (Theorem \ref{lem:TAPnew} $(v)$), one can argue that, for any fixed $n\geq 1$,
\begin{align}
\lim_{N\rightarrow\infty}\e F_{N,n}^{h}(m,\varepsilon_N,\delta_N)
=
\lim_{N\rightarrow\infty}\e F_{N,1}^{h}(m,\varepsilon_N,\delta_N),
\label{eqNoEcost}
\end{align}
which, obviously, will finish the proof. This follows from a standard approximation argument by generic models, exactly as in \cite{SubagFEL}, but, before we sketch it, let us notice that we are in the situation when $0<q=\int\! x^2\,\mu(dx)<1$ and $\int\! x(1-x)\,\mu(dx)>0,$ which implies that:
\begin{enumerate}
\item $(\lambda,\zeta)\to\PP^{\ef_{\zeta_0}}_\mu(\lambda,\zeta)$ is strictly convex (see e.g. (\ref{convexity}) below), 
\item by the equation (\ref{lem2eq1qp}) and (\ref{lem2eq1qp2}) in Lemma \ref{add:lem2}, the infimum $\inf_{\lambda, \zeta}\PP^{\ef_{\zeta_0}}_\mu(\lambda,\zeta)$ is achieved on $\lambda$ that is uniformly bounded, $|\lambda|\leq L$.
\end{enumerate}
For the specific model $\xi_q$ we are considering above, $\ef_{\zeta_0}$ was chosen in an optimal way, so that the minimizer is $(0,\zeta_0).$ However, we will now vary the model $\xi_q$ while keeping $\ef_{\zeta_0}$, so the two items above refer to this case. In particular, by continuity and compactness, these items imply that the minimizer (let us denote it by $(\lambda_{\xi_q},\zeta_{\xi_q})$) is unique and depends continuously on the model $\xi_q$ (see e.g. \cite[Corollary 4.2]{SKbonus}). The arguments in \cite[Section 3.7]{SKmodel} require no modifications to show that, for generic models on the narrow band, the distribution of the overlap converges to some $\zeta^*\in \mathcal{M}_{0,1-q}$ and the limit of the free energy, via the Aizenman-Sims-Starr cavity computation in Section \ref{sec:ParisiOnBand} above, is given by $\inf_\lambda \PP^{\ef_{\zeta_0}}_\mu(\lambda,\zeta^*).$ On the other hand, by the Parisi formula in Theorem \ref{thm1}, this limit equals 
$$
\inf_{\lambda, \zeta}\PP^{\ef_{\zeta_0}}_\mu(\lambda,\zeta)=\PP^{\ef_{\zeta_0}}_\mu(\lambda_{\xi_q},\zeta_{\xi_q}).
$$ 
By uniqueness of the minimizer, $\zeta^*= \zeta_{\xi_q}.$ Moreover, since this is the limiting distribution of the overlap, there can be no free energy cost of constraining the overlaps between $n$ replicas to some fixed value in the support of $\zeta_{\xi_q}.$ On the other hand, in our model above, the external field $\ef_{\zeta_0}$ was chosen in such a way that zero is in the support of the minimizer $\zeta_{0}$, so, when we approximate this model by generic models, by continuity of $\zeta_{\xi_q}$ in the model $\xi_q$, these generic models will have points in the support very close to zero. As a result, for our model above, there can be no free energy cost of constraining the overlaps to be near zero and (\ref{eqNoEcost}) must hold. For more details, see e.g. \cite[Lemma 4.8]{SKbonus}.

In the case when $q=\int\! x^2\,\mu(dx)=0,$ again, we can use the approximation in (\ref{eq:TAPmAg22}) and Theorem \ref{thm1ag2}. The argument here is exactly the same, except we do not need to introduce the external field, because zero is already in the support of the Parisi measure of the original model without external field.\qed

\section{TAP states are ancestral}\label{SecTAPAnc}

In this section we will prove Theorem \ref{Thm1label}, which will follow from the following zero-temperature formula from \cite{CPTAP17}, which is a generalization to soft spins of the zero-temperature result Jagannath-Sen \cite[Theorem 1.2]{JS17} for discrete spins, which itself was derived from the positive temperature formulas with general prior spin distributions \cite{Panchenko2005, Panchenko2015, panchenko2018}. 

Define a functional $\mathcal{P}_q$ on $\mathbb{R}\times \mathcal{M}_{0,q}$ by
\begin{align}\label{pfinfinity}
\mathcal{P}_q(\lambda,\gamma)=\Phi_{\gamma}^\lambda(0,0)-\frac{1}{2}\int_0^qs\xi''(s)\gamma(s) ds,
\end{align}
where, for a given $\lambda,$ $\Phi_{\gamma}^\lambda(s,x)$ is defined as the solution of
\begin{align}
\label{pde2ASTA}
\partial_s \Phi_{\gamma}^\lambda&=-\frac{\xi''(s)}{2}\Bigl(\partial_{xx}\Phi_{\gamma}^\lambda+\gamma(s)\bigl(\partial_x\Phi_{\gamma}^\lambda\bigr)^2\Bigr)
\end{align}
on $[0,q)\times\mathbb{R}$, with the boundary condition 
\begin{align}
\label{eqBoundary1}
\Phi_{\gamma}^\lambda(q,x):=\max_{a\in [-1,1]}\Bigl(a x+\lambda (a^2-q)
+\Lambda_{\zeta_*}(q,a)\Bigr),
\end{align}
where $\Lambda_{\zeta_*}(q,a)$ was defined in (\ref{eqCCPar}) (recall that it is bounded and continuous on $[-1,1]$). Then \cite[Theorem 5]{CPTAP17} implies the following.
\begin{thm} \label{thm1STA} For any $q\in [0,1],$ we have that
	\begin{align}
	\label{thm1:eq1STA}
\lim_{N\rightarrow\infty} \e \max_{\frac{1}{N}\|m\|^2=q}
	\frac{1}{N}\Bigl(
	H_N(m)
	+ \sum_{i=1}^{N}\Lambda_{\zeta_*}(q,m_i)
	\Bigr)
	=\inf_{(\lambda,\gamma)\in\mathbb{R}\times\mathcal{M}_{0,q}}\mathcal{P}_q(\lambda,\gamma).
	\end{align}
\end{thm}

\begin{proof}[Proof of Theorem \ref{Thm1label}]
Using this result, \eqref{eq:TAPzetastar}, Theorem \ref{thm:GenTAP}, and Gaussian concentration, in order to prove Theorem \ref{Thm1label}, it is enough to prove that
\begin{equation}
\inf_{(\lambda,\gamma)\in\mathbb{R}\times\mathcal{M}_{0,q}}\mathcal{P}_q(\lambda,\gamma)
-\frac{1}{2}\int_q^{1}\!s\xi''(s)\zeta_*(s)\,ds
\leq \PP(\zeta_*).
\end{equation}
We will take $\gamma=\zeta_*$ (restricted to $[0,q]$) and take $\lambda=0$. Then the function in the boundary condition (\ref{eqBoundary1}) is $a x+\Lambda_{\zeta_*}(q,a)$ and, since $\Phi_{\zeta_*}(q,x)$ is convex, the definition (\ref{eqCCPar}) implies by conjugation  that 
$$
\max_{a\in [-1,1]} \bigl(a x+\Lambda_{\zeta_*}(q,a)\bigr) = \Phi_{\zeta_*}(q,x).
$$
Since the PDE in (\ref{pde2ASTA}) with $\gamma=\zeta_*$ is the Parisi PDE for the original model, we get that 
$$
\mathcal{P}_q(0,\zeta_*)
=
\Phi_{\zeta_*}(0,0)-\frac{1}{2}\int_0^qs\xi''(s)\zeta_*(s) ds.
$$
Finally,
$$
\mathcal{P}_q(0,\zeta_*) -\frac{1}{2}\int_q^{1}\!s\xi''(s)\zeta_*(s)\,ds
=
\Phi_{\zeta_*}(0,0)-\frac{1}{2}\int_0^1\! s\xi''(s)\zeta_*(s) ds = \PP(\zeta_*),
$$
which finishes the proof.
\end{proof}

\section{Generalized TAP equations}\label{SecTAPGenEE}

In this section, we will prove the formula for the gradient of $\nTAP(\mu_m)$ in Theorem \ref{ThmGTElab} for all $m\in (-1,1)^N$. For $q\in [0,1]$, and $h\in\Reals^N$, define
\begin{equation}
\nTAP(m,q,\zeta,h):=
\frac{1}{N}\sum_{i=1}^{N}\bigl(\Phi_{\zeta}(q,h_i)-m_i h_i\bigr)
-\frac{1}{2}\int_q^{1}\!s\xi''(s)\zeta(s)\,ds.
\label{eqTAPfirstagGEE}
\end{equation}
If $\frac{1}{N}\|m\|^2=q$ then the definition of $\nTAP(\mu)$ in (\ref{eqTAPfirst}) implies that
\begin{equation}
\nTAP(m):=\nTAP(\mu_m) = 
\inf_{\zeta\in \mathcal{M}_{0,1},\,h\in\Reals^N} \nTAP(m,q,\zeta,h).
\label{eqTAPvariZH}
\end{equation}
If $m\in (-1,1)^N$, it is clear that we minimize over $h$ in some cube $[-L,L]^N$, where $L$ depends only on the largest value of $|m_i|$ (we will need this for compactness argument below). It is a standard fact that the functional $\nTAP(m,q,\zeta,h)$ is strictly convex in $h$ and convex in $\zeta$. The reason it is not strictly convex in $\zeta$ is because the functional depends only on the restriction of $\zeta$ to $[q,1]$. For this reason, let us make a convention that, for a given $q$, we minimize over $\zeta$ fixed to be $\zeta(s)=0$ for $s\in [0,q).$ Then the minimizer of the above functional for $m\in (-1,1)^N$ is unique and will be denoted by $\zeta_m, h_m,$ so that, for $q=\frac{1}{N}\|m\|^2,$
\begin{equation}
\nTAP(m)= \nTAP(m,q,\zeta_m,h_m).
\end{equation}
If $m$ converges to $m_0\in (-1,1)^N$ then, by the continuity of $\nTAP(\mu)$ proved in Theorem \ref{lem:TAPnew},
$$
\lim_{m\to m_0}\nTAP(m,q,\zeta_m,h_m) = \nTAP(m_0,q_0,\zeta_{m_0},h_{m_0}).
$$
Since $\mu_m\to \mu_{m_0}$, $q=\frac{1}{N}\|m\|^2\to q_0=\frac{1}{N}\|m_0\|^2$, and any subsequential limit of $\zeta_m$ is equal to zero on $[0,q_0)$ (by our convention above), the uniqueness of the minimizer implies that $\zeta_m\to \zeta_{m_0}$ and $h_m\to h_{m_0}.$ With this observation, in order to compute the gradient of $\nTAP(m)$, we will need two lemmas.
\begin{lem}
Consider a metric space $D$ and a function $f\colon (-\eps,\eps)\times D\to \Reals.$ Suppose that there exists a function $d\colon (-\eps,\eps)\to D$ such that 
\begin{equation}
f(t,d(t))=\inf_{d\in D}f(t,d),
\label{eqLemddtF}
\end{equation}
and suppose that $d(t)$ is continuous at $t=0$. Also, suppose that the right derivative $\partial_t^+ f(t,d)$ exists and is continuous at $(0,d(0))$. Then 
\begin{equation}
\frac{d^+}{dt}f(t,d(t))\Bigr|_{t=0}=\partial_t^+ f(0,d(0)).
\end{equation}
The same statement holds for left derivatives.
\end{lem}
\begin{proof}
Using (\ref{eqLemddtF}), for $\alpha>0$,
$$
\frac{f(\alpha,d(\alpha))-f(0,d(\alpha))}{\alpha}\leq\frac{f(\alpha,d(\alpha))-f(0,d(0))}{\alpha}\leq\frac{f(\alpha,d(0))-f(0,d(0))}{\alpha}
$$
and, therefore,
$$
\min_{s\in [0,\alpha]}\partial_t^+ f(s,d(\alpha))\leq\frac{f(\alpha,d(\alpha))-f(0,d(0))}{\alpha}\leq \max_{s\in [0,\alpha]} \partial_t^+ f(s,d(0)).
$$
Letting $\alpha\downarrow 0$ finishes the proof.
\end{proof}

As we discussed above, $\nTAP(m)$ is obtained by taking infimum over all $\zeta\in \mathcal{M}_{0,1}$ with $\zeta(s)=0$ on $[0,q)$ and $h\in \mathbb{R}^N.$ For any such $\zeta$ and $x\in \mathbb{R}$, define a stochastic process $u_{\zeta,x}$ on $[q,1]$ by 
$$
u_{\zeta,x}(s)=\partial_x\Phi_{\zeta}(s,X_{\zeta,x}(s)),
$$ 
where $(X_{\zeta,x}(s))_{q\leq s\leq 1}$ is the (strong) solution of the SDE
\begin{align*}
dX_{\zeta,x}(s)&=\zeta(s)\xi''(s)\partial_x\Phi_{\zeta}(s,X_{\zeta,x}(s))ds+\xi''(s)^{1/2}dW_s
\end{align*}
with the initial condition $X_{\zeta,x}(q)=x.$
\begin{lem}
	We have that
	\begin{align}\label{lem37:eq1}
	\partial_{xx}\Phi_{\zeta}(q,x)&=1-\int_{[q,1]}\! \e u_{\zeta,x}(l)^2\,d\zeta(l).
	\end{align}
\end{lem}
\begin{proof}
Recall from \cite[Lemma 2]{AC15} that there are two useful identities associated with  the process $X_{\zeta,x}$, namely, for any $q\leq s\leq s'\leq 1,$
	\begin{align}
	\e u^2_{\zeta,x}(s')-\e u^2_{\zeta,x}(s)&=\int_{s}^{s'}\! \xi''(r)\e\bigl(\partial_{xx}\Phi_{\zeta}(r,X_{\zeta,x}(r))\bigr)^2\,dr,
	\label{eqIden1opu}
	\end{align}
	and
	\begin{align}
	\partial_{xx}\Phi_{\zeta}(s',X_{\zeta,x}(s'))-\partial_{xx}\Phi_{\zeta}(s,X_{\zeta,x}(s))
	&=
	-\int_s^{s'}\! \xi''(r)\zeta(r)\bigl(\partial_{xx}\Phi_{\zeta}(r,X_{\zeta,x}(r))\bigr)^2\,dr
	\nonumber
	\\
	&+\int_{s}^{s'}\! \xi''(r)^{1/2}\partial_{xxx}\Phi_{\zeta}(r,X_{\zeta,x}(r))\,dW_r.
		\label{eqIden2opu}
	\end{align}
Using these and the Fubini theorem,
	\begin{align*}
	&\e\partial_{xx}\Phi_{\zeta}(1,X_{\zeta,x}(1))-\partial_{xx}\Phi_{\zeta}(q,x)\\
(\mbox{by (\ref{eqIden2opu})})\,\,	&=-\e\int_q^{1}\! \xi''(r)\zeta(r)\bigl(\partial_{xx}\Phi_{\zeta}(r,X_{\zeta,x}(r))\bigr)^2\,dr\\
	&=-\e\int_q^{1}\! \xi''(r)\bigl(\partial_{xx}\Phi_{\zeta}(r,X_{\zeta,x}(r))\bigr)^2\Bigl(\int_{[q,r]}\zeta (dl)\Bigr)\,dr\\
	&=-\e\int_{[q,1]}\int_{l}^1\! \xi''(r)\bigl(\partial_{xx}\Phi_{\zeta}(r,X_{\zeta,x}(r))\bigr)^2\,dr\,d\zeta(l)\\
(\mbox{by (\ref{eqIden1opu})})\,\,	&=-\int_{[q,1]}\!  \e\bigl(\partial_{x}\Phi_{\zeta}(1,X_{\zeta,x}(1))\bigr)^2-\e\bigl(\partial_{x}\Phi_{\zeta}(l,X_{\zeta,x}(l))\bigr)^2\,d\zeta(l)\\
	&=-\e\bigl(\partial_{x}\Phi_{\zeta}(1,X_{\zeta,x}(1))\bigr)^2+\int_{[q,1]}\e\bigl(\partial_{x}\Phi_{\zeta}(l,X_{\zeta,x}(l))\bigr)^2\,d\zeta(l).
	\end{align*}
	Noting that $\partial_{xx}\Phi_{\zeta}(1,x)=1-\tanh^2(x)=1-(\partial_x\Phi_{\zeta}(1,x))^2,$ we get
	\begin{align*}
	\e\partial_{xx}\Phi_{\zeta}(1,X_{\zeta,x}(1))&=1-\e\bigl(\partial_{x}\Phi_{\zeta}(1,X_{\zeta,x}(1))\bigr)^2.
	\end{align*}
	Combining these together completes our proof.
\end{proof}

\begin{proof}[Proof of Theorem \ref{ThmGTElab}]
Take $m_0\in (-1,1)^N$ and consider the path $m_t=m_0+tv$ for some $v\in \Reals^N$, which lies in $(-1,1)^N$ for $t\in (-\eps,\eps)$ for some small $\eps>0,$ and denote $q_t=\frac{1}{N}\|m_t\|^2.$  Let $D=\mathcal{M}_{0,1} \times \Reals^N$ and, with $d=(\zeta,h)$, let 
$
f(t,d) := \nTAP(m_t,q_t,\zeta,h).
$
Let $d(t):=(\zeta_t,h_t)=(\zeta_{m_t},h_{m_t})$ be the minimizers defined above, so that $\nTAP(m_t)=f(t,d(t)).$ Since $\nabla \nTAP(m_0)\cdot v$ is the derivative of $f(t,d(t))$ at $t=0$, we can apply the above lemma once its assumptions are verified. The continuity of $d(t)$ at $t=0$ follows from the discussion above. To compute $\partial_t^+ \nTAP(m_t,q_t,\zeta,h)$, we need to compute the partial derivatives of $\nTAP(m,q,\zeta,h)$ with respect to all $m_i$ and $q$, which are the only parameters that depend on $t$ for fixed $(\zeta,h)$. We will take the derivatives of $m_t,q_t$ in $t$ only at the end, using that, for $t=0,$ we have $(m,q,\zeta,h)=(m_0,q_0,\zeta_0,h_0)$. First, right and left derivatives in $q$ are equal to
$$
\partial_q^\pm \nTAP(m,q,\zeta,h) =
\frac{1}{N}\sum_{i=1}^{N}\partial_q^\pm\Phi_{\zeta}(q,h_i)
+\frac{1}{2}q\xi''(q)\zeta(q\pm0),
$$
where $\zeta(q\pm0)$ are the one-sided limits of $\zeta(q)$.
(Here, we consider both derivatives, because $q_t=\frac{1}{N}\|m_t\|^2$ may be increasing or decreasing with $t$.) Using the Parisi PDE (\ref{ParisiPDEOrig}) for the first term, we can rewrite
\begin{align*}
\partial_q^\pm \nTAP(m,q,\zeta,h) 
&=
-\frac{1}{N}\sum_{i=1}^{N}\frac{\xi''(q)}{2}\partial_{xx} \Phi_{\zeta}(q,h_i)
\\
&\quad\,\, +
\frac{1}{2}\xi''(q)\zeta(q\pm0)\Bigl(
q - \frac{1}{N}\sum_{i=1}^{N}\bigl(\partial_x \Phi_{\zeta}(q,h_i)\bigr)^2
\Bigr).
\end{align*}
The only possible discontinuity on the right hand side is in the c.d.f. $\zeta(q)$. However, at $t=0$,\begin{equation}
\partial_x \Phi_{\zeta}(q,h_i)\bigr|_{t=0}= \partial_x \Phi_{\zeta_0}(q_0,h_{0,i})=m_{0,i},
\label{eqpdxPimax}
\end{equation}
because $h_{0,i}$ is the minimizer and critical point of $\Phi_{\zeta_0}(q_0,h)-m_{0,i} h$ and, therefore,
$$
\frac{1}{N}\sum_{i=1}^{N}\bigl( \partial_x \Phi_{\zeta_0}(q_0,h_{0,i})\bigr)^2
=
\frac{1}{N}\sum_{i=1}^{N} m_{0,i}^2 = q_0.
$$
This means that the one-sides derivatives above are continuous at $(m_0,q_0,\zeta_0,h_0)$ and
\begin{align}
\partial_q^\pm \nTAP(m,q,\zeta,h)\bigr|_{t=0} 
&= 
-\frac{1}{N}\sum_{i=1}^{N}\frac{\xi''(q_0)}{2}\partial_{xx} \Phi_{\zeta_0}\bigl(q_0,\oPsi(q_0, m_{0,i},\zeta_0)\bigr),
\label{eqParDerq0}
\end{align}
since the minimizer $h_{0,i}=\oPsi(q_0, m_{0,i},\zeta_0)$ was defined in (\ref{eqDefPsiefO}).
The derivative $\nTAP(m,q,\zeta,h)$ in $m_i$ equals $-h_i$, which is continuous and
\begin{equation}
\partial_{m_i} \nTAP(m,q,\zeta,h)\bigr|_{t=0} =
-\frac{1}{N}h_{0,i}=-\frac{1}{N}\oPsi(q_0, m_{0,i},\zeta_0).
\end{equation}
If we denote
$$
C(m_0):=\xi''(q_0)\frac{1}{N}\sum_{i=1}^{N}
\partial_{xx} \Phi_{\zeta_0}\bigl(q_0,\oPsi(q_0, m_{0,i},\zeta_0)\bigr)
=\xi''(q_0)\frac{1}{N}\sum_{i=1}^{N}
\partial_{xx} \Phi_{\zeta_0}\bigl(q_0,h_{0,i}\bigr)
$$
and use that $\frac{d}{dt}q|_{t=0}=2m_0\cdot v$, combining the above we have that
\begin{align*}
\nabla \nTAP(m_0)\cdot v & = \frac{d}{dt} q\cdot \partial_q \nTAP(m,q,\zeta,h)\bigr|_{t=0}+ \sum_{i=1}^{N} \frac{d}{dt}m_i\cdot  \partial_{m_i} \nTAP(m,q,\zeta,h)\bigr|_{t=0}\\
& = -\frac{1}{N}\Bigl(\oPsi(q_0, m_{0,i},\zeta_0)+C(m_0)m_{0,i}
\Bigr)_{i\leq N}\cdot v
\end{align*}
for all $v\in\Reals^N.$ To complete the proof, it remains to verify that 
\begin{align}
C(m_0)&=\xi''(q_0)\! \int_{q_0}^1\zeta_{0}(s)\,ds.
\label{eqCm0}
\end{align}
For a fixed $m_0\in (-1,1)^N$ with $q_0=\frac{1}{N}\|m_0\|^2$, let us find the minimizers $\zeta_0$ and $h_0$ of the strictly convex variational problem (\ref{eqTAPvariZH}), with the above convention that we optimize over $\zeta$ fixed to be $\zeta(s)=0$ for $s\in [0,q).$ Then the directional derivative with respect to $\zeta$ at $\zeta_0$ in the direction of any other distributions must be non-negative. If, for a given $\zeta\in \mathcal{M}_{0,1}$, we consider the path $\zeta_b=\zeta_0+b(\zeta-\zeta_0)$ parametrized by $b\in [0,1],$ one can compute the directional derivative in a standard way (see \cite[Theorem 2]{C17} or \cite[Lemma 4.11]{SKbonus}) to get
\begin{align*}
\frac{d}{db}\nTAP(m_0,q_0,\zeta_b,h_0)\Big|_{b=0^+}&=\frac{1}{2}\int_q^1\! \xi''(s)\Bigl(\frac{1}{N}\sum_{i=1}^N\e u_{\zeta_0,h_{0,i}}(s)^2-s\Bigr)\bigl(\zeta(s)-\zeta_0(s)\bigr)\,ds\geq 0.
\end{align*}
Now, if we vary over all possible $\zeta$, the minimality of $\zeta_0$ implies that 
whenever $s\in [q,1]$ is in the support of $\zeta_0$, we must have that (see, e.g. \cite{C17,JT17})
$$
\frac{1}{N}\sum_{i=1}^N\e u_{\zeta_0,h_{0,i}}(s)^2=s.
$$
Plugging this into \eqref{lem37:eq1} with $\zeta=\zeta_0$ and initial condition $X_{\zeta,x}(q)=x=h_{0,i}$ and averaging over $i$,
\begin{align*}
\frac{1}{N}\sum_{i=1}^N\partial_{xx}\Phi_{\zeta_0}(q,h_{0,i})&=1-\int_{[q,1]}\! l\,d\zeta(l).
\end{align*}
Using that
\begin{align*}
\int_{[q,1]}l\,d\zeta(l)&=\int_{[q,1]}\Bigl(\int_q^l\! dw\Bigr)\,d\zeta(l)+q=\int_q^1\int_{[w,1]}\!d\zeta (l)\,dw+q
\\
&=\int_q^1\!(1-\zeta(w^-))\,dw+q=1-\int_q^1\!\zeta(w)\,dw
\end{align*}
finishes the proof of (\ref{eqCm0}).
\end{proof}

\section{Classical TAP correction}\label{SecTAPclass}

We proceed to establish Propositions \ref{add:prop2}  and \ref{add:prop4} and Corollary \ref{add:cor1}. 
\begin{proof}[Proof of Proposition \ref{add:prop2}]
	Assume that $(0,\delta_0)$ is the minimizer to $\oP_\mu^{\ef_{\RS}}.$ Let $\zeta\in \mathcal{M}_{0,1-q}$ be fixed. For $b\in [0,1],$ define $$
	\zeta_b=(1-b)\delta_0+b \zeta.$$ Then from the minimality of $\delta_0,$ one gets that (see Remark \ref{rmk2})
	\begin{align*}
	\frac{d}{db}\oP_\mu^{\ef_{\RS}}(0,\zeta_b)\Big|_{b=0^+}=\frac{1}{2}\int_0^{1-q}\xi_q''(s)
	(\gamma_\mu(s)-s)(\zeta(s)-\delta_0(s))ds\geq 0.
	\end{align*}
Next, by Fubini's theorem, write
	\begin{align}
	\begin{split}\label{add:eq9}
	0\leq &\int_0^{1-q}\xi_q''(s)
	(\gamma_\mu(s)-s)(\zeta (s)-\delta_0(s))ds\\
	&=\int_0^{1-q}\xi_q''(s)(\gamma_\mu(s)-s)\Bigl(\int_0^s(\zeta (dr)-\delta_0(dr))\Bigr)ds\\
	&=\int_0^{1-q}\int_{r}^{1-q}\xi_q''(s)(\gamma_\mu(s)-s)ds(\zeta (dr)-\delta_0(dr))\\
	&=\int_0^{1-q}\int_{r}^{1-q}\xi_q''(s)(\gamma_\mu(s)-s)ds\zeta (dr)
	-\int_0^{1-q}\xi_q''(s)(\gamma_\mu(s)-s)ds\\
	&=-\int_0^{1-q}\Gamma_\mu(r)\zeta (dr).
	\end{split}
	\end{align}
	Since this inequality holds for all $\zeta\in \mathcal{M}_{0,1-q}$, it follows that $\Gamma_{\mu}(r)\leq 0$ for all $r\in [0,1-q].$
	Conversely, if $\Gamma_\mu(r)\leq 0$ for all $0\leq r\leq 1-q,$ one can reverse the argument to get that $$
	\frac{d}{db}\oP_\mu^{\ef_{\RS}}(0,\zeta _b)\Big|_{b=0^+}\geq 0
	$$
	for all $\zeta\in \mathcal{M}_{0,1-q}.$ Since $\oP_\mu^{\ef_{\RS}}(0,\cdot)$ is a strictly convex functional (using the assumption $\mu\neq \delta_1$), this implies that $\delta_0$ is the unique minimizer of $\oP_\mu^{\ef_{\RS}}(0,\cdot)$. In order to show that $(\lambda_\mu^{\ef_{\RS}},\zeta_\mu^{\ef_{\RS}})=(0,\delta_0)$ we split our discussion into two cases: the support of $\mu$ contains a point in $(0,1)$ and $\mu$ is supported only on $\{0,1\}$ with $\mu(\{0\})>0.$ In the latter case, recall that we defined $\lambda_{\mu}^{\ef_{\RS}}=0$ and $\zeta_{\mu}^{\ef_{\RS}}$ to be the (unique) minimizer of $\oP_\mu^{\ef_{\RS}}(0,\cdot)$. Hence, $(\lambda_{\mu}^{\ef_{\RS}},\zeta_{\mu}^{\ef_{\RS}})=(0,\delta_0).$ In the former case, Remark \ref{rmk1} shows that there exists a unique minimizer $(\lambda_{\mu}^{\ef_{\RS}},\zeta_{\mu}^{\ef_{\RS}})$ of $\oP_\mu^{\ef_{\RS}}(\cdot,\cdot)$.
Since by definition ${\ef_{\RS}}$ satisfies $\Phi_{a,\delta_0}(0,{\ef_{\RS}}(a))=0$ for all $a\in [0,1)$, this implies that $\partial_\lambda\mathcal{P}_\mu^{\ef_{\RS}}(0,\delta_0)=0.$ This together with the fact that $\delta_0$ is the minimizer of $\oP_\mu^{\ef_{\RS}}(0,\cdot)$ implies that $(\lambda_{\mu}^{\ef_{\RS}},\zeta_{\mu}^{\ef_{\RS}})=(0,\delta_0).$ This establishes the equivalence conditions in the statement of Proposition \ref{add:prop2}.
	
	Finally, we compute $\oP_\mu^{\ef_{\RS}}(0,\delta_0)$ assuming that either of the conditions in the equivalence holds. Recall the explicit expression of $\Phi_{a,\delta_0}$ from \eqref{add:eq81} and the particular choice of $\ef_{\RS}$ from \eqref{add:eq8}. From these, it can be checked that
	\begin{align*}
	\Phi_{a,\delta_0}(0,\ef_{\RS}(a))&=\frac{1-a^2}{2}\xi_q'(1-q)-a\th^{-1}(a)+\log 2\ch(\th^{-1}(a))\\
	&=\frac{1-a^2}{2}\xi_q'(1-q)-I(a).
	\end{align*}
	From this,
	\begin{align*}
	\int_{[0,1)}\Phi_{a,\delta_0}(0,\ef_{\RS}(a))d\mu(a)&=\frac{1}{2}\xi_q'(1-q)\int_{[0,1)}(1-a^2)dd\mu(a)-\int_{[0,1)}I(a)d\mu(a)\\
	&=\frac{1}{2}\xi_q'(1-q)\int(1-a^2)da-\int I(a)d\mu(a)\\
	&=\frac{1}{2}\xi_q'(1-q)(1-q)-\int I(a)d\mu(a),
	\end{align*}
	where we have used that $I(a)$ is continuous at $1$ and $I(1)=0.$ In addition, using integration by parts
	\begin{align*}
	\int_0^{1-q}\xi_q''(s)sds&=\xi_q'(1-q)(1-q)-\xi_q(1-q).
	\end{align*}
	Therefore,
	\begin{align*}
	\oP_\mu^{\ef_{\RS}}(0,\delta_0)
	&=\int_{[0,1)}\Phi_{a,\delta_0}(0,\ef_{\RS}(a))d\mu(a)-\frac{1}{2}\int_0^{1-q}\xi_q''(s)sds\\
	&=\frac{1}{2}\xi_q'(1-q)(1-q)-\int I(a)d\mu(a)-\frac{1}{2}\xi_q'(1-q)(1-q)+\frac{1}{2}\xi_q(1-q)\\
	&=-\int  I(a)d\mu(a)+\frac{1}{2}\bigl(\xi(1)-\xi(q)-\xi'(q)(1-q)\bigr),
	\end{align*}
	where we used $\xi_q(1-q)=\xi(1)-\xi(q)-\xi'(q)(1-q).$
	This establishes the desired formula for $\oP_\mu^{\ef_{\RS}}(0,\delta_0).$ Finally, since $\mu\neq \delta_1$ and $(0,\delta_0)$ is the minimizer to $\oP_\mu^{\ef_{\RS}}$, it follows from Theorem \ref{lem:TAPnew} that $\TAP(\mu)=\oTAP(\mu)=\oP_\mu^{\ef_{\RS}}.$
\end{proof}

\begin{proof}[Proof of Corollary \ref{add:cor1}] For $m\in [-1,1]^N$ and $\mu\in \MM_{0,1},$ denote $q_m=\|m\|_2^2/N$ and $q=\int a^2d\mu(a).$ 
	Recall from \cite{CPTAP17} that if $q_{\EA}$ is the largest point in the support of the original Parisi measure of the Parisi formula for $F_N$, then almost surely
	\begin{align}\label{add:cor1:proof:eq2}
	\lim_{N\rightarrow\infty}F_N&=\lim_{\varepsilon\downarrow 0}\lim_{N\rightarrow\infty}\sup_{m\in [-1,1]^N:|q_m-q_{\EA}|\leq \varepsilon}\Bigl(\frac{H_N(m)}{N}-\int  I(a)d\mu_m(a)+C(q_m)\Bigr).
	\end{align}
	Also recall from Theorem \ref{thm:GenTAP} that
	\begin{align}\label{add:cor1:proof:eq6}
	\lim_{N\rightarrow\infty}F_N&=\lim_{\varepsilon\downarrow 0}\lim_{N\rightarrow\infty}\sup_{m\in [-1,1]^N:|q_m-q_{\EA}|\leq \varepsilon}\Bigl(\frac{H_N(m)}{N}+\oTAP(\mu_m)\Bigr).
	\end{align}
	Since
	\begin{align}\label{add:cor1:proof:eq7}
	\oTAP(\mu_m)&\leq \mybar{P}_{\mu_m}^{\ef_{\RS}}(0,\delta_0)=-\int  I(a)d\mu_m(a)+C(q_m),\,\,\forall m\in [-1,1]^N,
	\end{align}
	we see that
	\begin{align}\label{add:eq7}
	\lim_{\varepsilon\downarrow 0}\lim_{N\rightarrow\infty}\sup_{m\in [-1,1]^N:|q_m-q_{\EA}|\leq \varepsilon}\Bigl(-\int  I(a)d\mu_m(a)+C(q_m)-\oTAP(\mu_m)\Bigr)=0.
	\end{align}
	We claim that for any $\varepsilon_N\downarrow 0,$ 
	\begin{align*}
	\limsup_{N\rightarrow\infty}\sup_{m\in [-1,1]^N:|q_m-q_{\EA}|\leq \varepsilon_N}\sup_{0\leq s\leq 1-q_m}\Gamma_{\mu_m}(s)\leq  0.
	\end{align*}
	If this is not true, then there exists some $\delta>0$ and a sequence $m^N$ with $|q_{m^N}-q_{\EA}|\leq \varepsilon_N$ such that $\mu_{m^N}$ converges to certain $\mu$ weakly and 
	$$
	\sup_{0\leq s\leq 1-q_{m^N}}\Gamma_{\mu_{m^N}}(s)\geq \delta.
	$$
	From these, passing to the limit gives
	\begin{align*}
	\int a^2d\mu(a)=q_{\EA}\,\,\mbox{and}\,\,\sup_{0\leq s\leq 1-q_{\EA}}\Gamma_{\mu}(s)\geq \delta.
	\end{align*}
	Note that $q_{\EA}<1$ implies $\mu\neq \delta_1$. From this, the above display, and Proposition~\ref{add:prop2}, we arrive at 
	$$\oTAP(\mu)< \oP_\mu^{\ef_{\RS}}(0,\delta_0)=-\int  I(a)d\mu(a)+C(q).$$ Since these two sides are uniformly continuous functions of $\mu$, we see that 
	\begin{align*}
	\liminf_{N\rightarrow\infty}\Bigl(-\int  I(a)d\mu_m(a)+C(q_{m^N})-\oTAP(\mu_{m^N})\Bigr)>0,
	\end{align*}
	which contradicts \eqref{add:eq7}. This establish our claim. 
	
	Now from the above claim, for any $\eta>0$, there exists an $\varepsilon>0$ such that as long as $N$ is large enough, if $m\in [-1,1]^N$ satisfies $|q_m-q_{\EA}|\leq \varepsilon$, then
	\begin{align*}
	\sup_{0\leq s\leq 1-q_m}\Gamma_{\mu_m}(s)\leq \eta.
	\end{align*}
	From this, \eqref{add:cor1:proof:eq6}, and Theorem \ref{thm:GenTAP},
	\begin{align*}
	\lim_{N\rightarrow\infty}F_N&=\lim_{\eta\downarrow 0}\lim_{N\rightarrow\infty}\sup_{m\in [-1,1]^N:\sup_{0\leq s\leq 1-q_m}\Gamma_{\mu_m}(s)\leq \eta}
	\Bigl(\frac{H_N(m)}{N}+\oTAP(\mu_{\mu_m})\Bigr).
	\end{align*}
	Recall \eqref{add:cor1:proof:eq7}. Our proof will be completed if it is established that
	\begin{align}\label{add:cor1:proof:eq8}
	\lim_{\eta\downarrow 0}\lim_{N\rightarrow\infty}\sup_{m\in [-1,1]^N:\sup_{0\leq s\leq 1-q_m}\Gamma_{\mu_m}(s)\leq \eta}\Bigl(-\int  I(a)d\mu_m(a)+C(q_m)-\oTAP(\mu_m)\Bigr)=0.
	\end{align}
	The argument of proving this is essentially the same as the above claim. Assume on the contrary that there exist $\eta_N\downarrow 0$, $\delta>0$ and $m^N\in [-1,1]^N$ with $$
	\sup_{0\leq s\leq 1-q_{m^N}}\Gamma_{\mu_{m^N}}(s)\leq \eta_N$$ 
	such that $\mu_{m^N}$ weakly converges to some $\mu$ and
	\begin{align*}
	-\int  I(a)\mu_{m^N}(da)+C(q_{m^N})-\oTAP(\mu_{m^N})\geq \delta,\,\,\forall N\geq 1.
	\end{align*}
	From these, we see that by passing to the limit, $$
	\sup_{0\leq s\leq 1-q}\Gamma_{\mu}(s)\leq 0
	$$ and 
	\begin{align*}
	-\int  I(a)\mu (da)+C(q )-\oTAP(\mu )\geq \delta.
	\end{align*}
	If $\mu\neq \delta_1$, then these contradict Proposition \ref{add:prop2}. If $\mu=\delta_1,$ then $q=1$ and in this case, it can be clearly checked that $$
	\oTAP(\mu)=0=-\int I(a)d\mu(a)+C(q),$$
	which again contradict to the above inequality.  Hence \eqref{add:cor1:proof:eq8} must be valid.
\end{proof}

\begin{proof}[Proof of Proposition \ref{add:prop4}]
	Note that $\xi_q(s)=\beta^2s^2/2$ and 
	\begin{align*}
	\Gamma_\mu(s)=\beta^2\int_0^s(\gamma_\mu(s)-s)ds.
	\end{align*}
	Assume that $(\lambda_\mu^{\ef_{\RS}},\zeta_\mu^{\ef_{\RS}})=(0,\delta_0).$  From Proposition \ref{add:prop2}, it can be seen that $\Gamma_\mu(s)$ attains the global maximum at $0$. On the other hand, it can also be checked that $\frac{d}{ds}\Gamma_\mu(0)=\beta^2\gamma_\mu(0)=0$ by a direct computation. From these, it follows that the second derivative of $\Gamma_\mu$ at zero is not positive. Now, following the same computation as \cite[Proposition 3]{AC15.1}, this second derivative can be computed as
	$$\frac{d^2}{ds^2}\Gamma_\mu(0)=\beta^2\Bigl(\beta^2\int \E\partial_{xx}\Phi_{a,\delta_0}(0,\ef_{\RS}(a))^2d\mu(a)-1\Bigr)\leq 0.$$
	Consequently,
	\begin{align*}
	\beta^2\int \E\partial_{xx}\Phi_{a,\delta_0}(0,\ef_{\RS}(a))^2d\mu(a)\leq 1.
	\end{align*}
	Finally, since
	\begin{align*}
	\partial_{xx}\Phi_{a,\delta_0}(0,\ef_{\RS}(a))&=1-\th^2(\ef_{\RS}(a)-at^2(0))=1-a^2,\forall a\in [0,1),
	\end{align*}
	we arrive at the Plefka condition by plugging this equation into the above inequality.
\end{proof}


\section{Analytical results}\label{SecAnalytical}

This section is devoted to handling some basic properties of the effective field $\Psi(a,\zeta)$ and the Parisi functional $\oP_\mu^{\ef}(\lambda,\zeta).$

\subsection{Basic facts about $\Phi_{a,\zeta}$}\label{subphi}
Recall  that the PDE solutions $\Phi_{a,\zeta}$ and $\Phi_\zeta$ are connected by \eqref{add:eq2}. In order to state several useful properties of $\Phi_{a,\zeta}$ in the present paper, we first recall some well-known properties of the PDE solution $\Phi_{\zeta}.$ First of all, $\Phi_{\zeta}$ satisfies $\sup_{\zeta\in \mathcal{M}_{0,1}}\|\partial_{x^k}\Phi_{\zeta}\|_\infty<\infty$ for all $k\geq 1,2,3$. Second, $\Phi_{\zeta}$ can be written as a stochastic optimization problem. Third, for any $t\in [0,1]$, $(\zeta,x)\in \mathcal{M}_{0,1}\times\mathbb{R}\mapsto \partial_x^k\Phi_\zeta(t,x)$ for $k=0,1,2,3$ is Lipschitz and $(\zeta,x)\in \mathcal{M}_{0,1}\times\mathbb{R}\mapsto \Phi_\zeta(t,x)$ is strictly convex. Lastly, the directional derivative of $\Phi_{\zeta}(t,x)$ in $(\zeta,x)$ exists and admits an explicit formula in terms of the optimal process appearing in the stochastic control representation of $\Phi_\zeta$. See \cite{AC15.1,AC15,C17,JT16} for these results.
Due to the equation \eqref{add:eq2}, it can be checked immediately that the following statements are valid for any $\mu\in \MM_{0,1}$ with $q=\int a^2d\mu(a):$

\noindent{\bf (I) Regularity:} We have that
\begin{align}
\begin{split}\label{subphi:eq1}
\sup_{(a,\zeta,t,x)\in [0,1]\times \mathcal{M}_{0,1-q}\times [1-q]\times\mathbb{R}}&\Bigl\{|\partial_{x}\Phi_{a,\zeta}(t,x)|,|\partial_{xx}\Phi_{a,\zeta}(t,x)|,\\
&\quad   |\partial_{xx}\Phi_{a,\zeta}(t,x)|,|\partial_{ax}\Phi_{a,\zeta}(t,x)|\Bigr\}<\infty.
\end{split}
\end{align}

\noindent {\bf (II) Stochastic optimal control:}
The quantity $\Phi_{a,\zeta}(0,x)$ can be expressed as a stochastic optimal control problem, which states that
\begin{align}
\begin{split}\label{control}
\Phi_{a,\zeta}(0,x)&=\sup_{u}\Bigl(\e f\Bigl(a,x+\int_0^{1-q}\xi_q''(s)\zeta(s)u(s)ds+\int_0^{1-q}\xi_q''(s)^{1/2}dW_s\Bigr)\\
&\qquad\qquad\qquad-\frac{1}{2}\int_0^{1-q}\xi_q''(s)\zeta(s)\e u(s)^2ds\Bigr),
\end{split}
\end{align} 
where the supremum is over all progressively measurable processes 
$u=(u(s))_{0\leq s\leq 1-q}$ with respect to the standard Brownian motion $W=(W_s)_{0\leq s\leq 1-q}$ and with 
$$\sup_{s\in [0,1-q]}|u(s)|\leq 2.$$ Here the optimal process is attained by $u_{a,\zeta,x}(s)=\partial_x\Phi_{a,\zeta}(s,X_{a,\zeta,x}(s))$, where $X_{a,\zeta,x}$ is the solution to the following SDE with initial condition $X_{a,\zeta,x}(0)=x,$
\begin{align}\label{ex:eq2}
dX_{a,\zeta,x}(s)&=\xi_q''(s)\zeta(s)\partial_x\Phi_{a,\zeta}(s,X_{a,\zeta,x}(s))ds+\xi_q''(s)^{1/2}dW_s,\,\,\forall 0\leq s\leq 1-q.
\end{align}

\noindent {\bf (III) Lipschitz property:} For any $(m,n)\in \{0,1\}\times \{0,1,2\},$ $(a,\zeta,x)\in[0,1]\times \mathcal{M}_{0,1-q}\times\mathbb{R}\mapsto \partial_{a^mx^n}\Phi_{a,\zeta}(t,x)$ is Lipschitz in the sense that
\begin{align}
\begin{split}\label{ex:lip}
&\bigl|\partial_{a^mx^n}\Phi_{a,\zeta}(0,x)-\partial_{a^mx^n}\Phi_{a',\zeta'}(0,x')\bigr|\\
&\leq C\Bigl(|a-a'|+|x-x'|+\int_0^{1-q}|\zeta-\zeta'|ds\Bigr)
\end{split}
\end{align} 
for some universal constant $C>0$ depending only on $\xi.$

\noindent {\bf (IV) Convexity:} For any $\zeta_0,\zeta_1\in \mathcal{M}_{0,1-q}$ and $x_0,x_1\in \mathbb{R}$, define 
\begin{align}
\begin{split}\label{xb}
\zeta_b&=(1-b)\zeta_0+b\zeta_1,\\
x_b&=(1-b)x_0+bx_1
\end{split}
\end{align}
for $b\in [-1,1]$. For any $a\in [0,1]$, we have that
\begin{align}\label{convexity0}
\Phi_{a,\zeta_b}(0,x_b)&\leq (1-b)\Phi_{a,\zeta_0}(0,x_0)+b\Phi_{a,\zeta_1}(0,x_1).
\end{align}
Furthermore, whenever $(\zeta_0,x_0)\neq (\zeta_1,x_1)$ and $b\in (0,1)$, this inequality is strict,
\begin{align}\label{convexity}
\Phi_{a,\zeta_b}(0,x_b)&< (1-b)\Phi_{a,\zeta_0}(0,x_0)+b\Phi_{a,\zeta_1}(0,x_1).
\end{align}

\noindent {\bf (V) Directional derivative:} Recall the convex combination in \eqref{xb}. The derivative of $\Phi_{a,\cdot}(0,\cdot)$ exists and is equal to
\begin{align}
\begin{split}\label{diff}
\frac{d}{db}\Phi_{a,\zeta_b}(0,x_b)&=\frac{1}{2}\int_0^{1-q}\xi_q''(s)(\zeta_1(s)-\zeta_0(s))\e u_{a,\zeta_b,x_b}(s)^2ds\\
&\qquad+(x_1-x_0)\partial_x\Phi_{a,\zeta_b}(0,x_b)
\end{split}
\end{align}
for any $a\in [0,1]$ and $b\in (0,1),$ where $u_{a,\zeta_b,x_b}$ is the optimal process of \eqref{control} with $x=x_b$ and $\zeta=\zeta_b$. Furthermore, the right derivative of $\Phi_{a,\zeta_b}(0,x_b)$ also exists at $b=0$ and it is equal to the above formula with $b=0.$ 

\subsection{Basic properties of $\Psi(a,\zeta)$}\label{SubsectionPsi} 

Now we use the properties listed in the above subsection to study a number of key features of $\Psi(a,\zeta)$ defined in (\ref{eqDefPsief}).

\begin{lem}\label{ex:lem1}
	For any $q\in [0,1]$, $\Psi$ is continuous on $[0,1)\times \mathcal{M}_{0,1-q}.$ 
\end{lem}

\begin{proof}
	Let $a_0\in [0,1)$, $\zeta_0\in\mathcal{M}_{0,1-q}$ be fixed. By definition of $\Psi(a_0,\zeta_0)$, for any $\eps>0,$
	\begin{align*}
	\delta:=\min_{x\in I(\eps)^c}|\partial_x\Phi_{a_0,\zeta_0}(0,x)|>0,
	\end{align*}
	where $I(\eps)^c$ is the complement of an open interval defined by
	$$I(\eps)=(\Psi(a_0,\zeta_0)-\eps,\Psi(a_0,\zeta_0)+\eps).$$ From the  Lipschitz property \eqref{ex:lip} with $(m,n)=(0,1)$, we see that if $$|a-a_0|+\int_{0}^{1-q}|\zeta(s)-\zeta_0(s)|ds<\frac{\delta}{2C},$$
	then for any $x\in I(\eps)^c$, we have
	\begin{align*}
	|\partial_x\Phi_{a,\zeta}(0,x)|&\geq |\partial_x\Phi_{a_0,\zeta_0}(0,x)|-C\Bigl(|a-a_0|+\int_{0}^{1-q}|\zeta(s)-\zeta_0(s)|ds\Bigr)\geq\frac{\delta}{2}.
	\end{align*}
	This implies that
	\begin{align*}
	\min_{x\in I(\eps)^c}|\partial_x\Phi_{a,\zeta}(0,x)|\geq \frac{\delta}{2}.
	\end{align*}
	Since $\partial_x\Phi_{a,\zeta}(0,\Psi(a,\zeta))=0,$ it follows that $\Psi(a,\zeta)\in I(\eps)$. Hence, $\Psi(a,\zeta)$ is continuous on $[0,1)\times\mathcal{M}_{0,1-q}.$ 
\end{proof}

Recall the definition of the effective field $\ef_\zeta=\Psi(a,\zeta)$ from \eqref{add:eq12}.

\begin{lem}\label{add:lem}
	Let $q\in [0,1].$ For any $\zeta\in \mathcal{M}_{0,1-q}$, $\ef_\zeta$ is a well-defined strictly increasing function with $\ef_\zeta(0)=0$.
	In addition, for any $l\in (0,1),$
	\begin{align*}
	\sup_{(a,\zeta)\in [0,l]\times \mathcal{M}_{0,1-q}}\partial_a\ef_\zeta(a)<\infty.
	\end{align*}

\end{lem}

\begin{proof}
	Let $q\in [0,1]$ be fixed.
	Recall from \eqref{add:eq2} that if $\zeta\in \mathcal{M}_{0,1-q}$ and $\zeta_0\in \mathcal{M}_{0,1}$ satisfy $\zeta(t)=\zeta_0(q+t)$ for $t\in [0,1-q],$ then for any $a\in [-1,1]$ and $x\in \mathbb{R}$,
	\begin{align}\label{add:eq6}
	\Phi_{a,\zeta}(0,x)&=\Phi_{\zeta}\Bigl(q,x-a\int_q^1\xi''(s)\zeta_0(s)ds\Bigr)-ax+\frac{a^2}{2}\int_q^1\xi''(s)\zeta_0(s)ds.
	\end{align}
	Since $\lim_{|x|\to \infty}\Phi_{\zeta_0}(q,x)=\infty$, we see that $\lim_{|x|\to \infty}\Phi_{a,\zeta}(0,x)=\infty$ for all $a\in [0,1)$ and $\zeta\in\mathcal{M}_{0,1-q}$.
	On the other hand, we also know that $\Phi_{a,\zeta}(0,\cdot)$ is a strictly convex function, by \eqref{convexity}. These imply that for any $a\in [0,1)$, $\Phi_{a,\zeta}(0,\cdot)$ has only one critical point, so $\ef_\zeta$ is well-defined. In particular, when $a=0,$ $\Phi_{a,\zeta}(0,\cdot)$ is an even function so that $\ef_\zeta(0)=0.$ 
	
	Next, we show that $\ef_\zeta$ is strictly increasing. First, note that $\partial_{xx}\Phi_{\zeta_0}(q,x)>0$ for all $x\in \mathbb{R}$ since $\Phi_{\zeta_0}(q,\cdot)$ is strictly convex. From this and \eqref{add:eq6}, for any $(a,\zeta,x)\in [0,1]\times \mathcal{M}_{0,1-q}\times \mathbb{R}$, 
	\begin{equation}
	\begin{split}\label{theaboveclaim}
	\partial_{ax}\Phi_{a,\zeta}(0,x)&=-\Bigl(\int_0^1\xi''(s)\zeta_0(s)ds \Bigr)\partial_{xx}\Phi_{\zeta_0}\Bigl(0,x-a\int_0^1\xi''(s)\zeta_0(s)ds\Bigr)<0, \\
	\partial_{xx}\Phi_{a,\zeta}(0,x)&=\partial_{xx}\Phi_{\zeta_0}\Bigl(0,x-a\int_0^1\xi''(s)\zeta_0(s)ds\Bigr)>0.
	\end{split}
	\end{equation}
	Consequently, a direct differentiation of $\partial_x\Phi_{a,\zeta}(0,\ef_\zeta(a))=0$ in $a\in [0,1)$ and using (\ref{theaboveclaim}) yield that
	\begin{align}\label{add:eq-4}
	\partial_a\ef_{\zeta}(a)&=-\frac{\partial_{ax}\Phi_{a,\zeta}(0,x)}{\partial_{xx}\Phi_{a,\zeta}(0,x)}\Big|_{x=\ef_{\zeta}(a)}> 0.
	\end{align}
	Hence, $\ef_{\zeta}$ is strictly increasing on $[0,1).$ 
	
	Finally, we prove the uniform upper bound for $\partial_a\ef_\zeta(a).$ Fix $l\in (0,1).$ By  Lemma \ref{ex:lem1}, $\ef_\zeta(a)\leq L$ for all $a\leq l$ and all $\zeta$, for some large enough $L$. Therefore,
	\begin{align*}
	\inf_{(a,\zeta)\in [0,l]\times \mathcal{M}_{0,1-q}}\partial_{xx}\Phi_{a,\zeta}(0,\ef_\zeta(a))
	\geq
	\inf_{(a,\zeta,x)\in [0,1]\times \mathcal{M}_{0,1-q}\times[0,L]}\partial_{xx}\Phi_{a,\zeta}(0,x)
	>0,
	\end{align*}
	by (\ref{theaboveclaim}), continuity \eqref{ex:lip} and compactness. On the other hand, from \eqref{subphi:eq1}, 
	\begin{align*}
	\sup_{(a,\zeta,x)\in [-1,1]\times \mathcal{M}_{0,1-q}\times\mathbb{R}}\bigl|\partial_{ax}\Phi_{a,\zeta}(0,x)\bigr|<\infty.
	\end{align*}
	From these inequalities and \eqref{add:eq-4}, $\sup_{(a,\zeta)\in [0,l]\times \mathcal{M}_{0,1-q}}\partial_a\ef_\zeta(a)<\infty$.
\end{proof}

\begin{lem}\label{ex:lem2} 
	There exist positive constants $c_1,c_1'\in \mathbb{R}$ and $c_2,c_2'>0$ such that
	\begin{align}\label{ex:lem2:eq1}
	c_1'+c_2'\tanh^{-1} (a)\leq \Psi(a,\zeta)\leq c_1+c_2\tanh^{-1}(a)
	\end{align}
	for all $a\in [0,1)$, $q\in [0,1]$ and $\zeta\in \mathcal{M}_{0,1-q}.$
\end{lem}

\begin{proof}
	Note that $f(t,x):=\partial_x\Phi_{a,\zeta}(t,x)$ satisfies the equation
	\begin{align*}
	\partial_tf&=-\frac{\xi_q''(t)}{2}\bigl(\partial_{xx}f+2\zeta(t) (\partial_xf)(\partial_x\Phi_{a,\zeta})\bigr)
	\end{align*}
	with $f(1-q,x)=-a+\tanh{x}.$ Using the Feyman-Kac formula, 	\begin{align}
	\partial_x\Phi_{a,\zeta}(0,x)&=-a+\e \tanh X_{a,\zeta,x}(1-q),
	\label{eqAbove14}
	\end{align}
	where $X_{a,\zeta,x}$ is defined through \eqref{ex:eq2}. Note that since $|u(s)|\leq 2$ and $|\zeta|\leq 1$, we have that $x+z_-\leq X_{a,\zeta,x}(1-q)\leq x+z_+,$	where
	\begin{align*}
	z_-&=-2\xi_q'(1-q)+\int_0^{1-q}\xi_q''(s)^{1/2}dW_s,\\
	z_+&=2\xi_q'(1-q)+\int_0^{1-q}\xi_q''(s)^{1/2}dW_s.
	\end{align*}
	Using that $\xi_q'(s)=\xi'(s+q)-\xi'(q)$ and $\xi_q''(s)=\xi''(s+q),$ we get	\begin{align*}
	z_-&\stackrel{d}{=}-2(\xi'(1)-\xi'(q))+\int_q^{1}\xi''(s)^{1/2}dW_s\stackrel{d}{=}-2\sigma^2(q)+\sigma(q)g,\\
	z_+&\stackrel{d}{=}2(\xi'(1)-\xi'(q))+\int_q^{1}\xi''(s)^{1/2}dW_s\stackrel{d}{=}2\sigma^2(q)+\sigma(q)g,
	\end{align*}	
	where $\sigma(q)=(\xi'(1)-\xi'(q))^{1/2}$ and $g\sim N(0,1)$. Since $\partial_x\Phi_{a,\zeta}(0,\Psi(a,\zeta))=0$, the equation (\ref{eqAbove14}) with $x=\Psi(a,\zeta)$ implies that	
	\begin{align}\label{add:eq-3}
	\e\tanh \bigl(\Psi(a,\zeta)+z_-)\leq a\leq \e\tanh \bigl(\Psi(a,\zeta)+z_+).
	\end{align}
	Note that $\tanh(x)$ is nondecreasing on $\mathbb{R}$ and is concave on $[0,\infty)$, so
	\begin{align*}
	a\leq \e\tanh \bigl(\Psi(a,\zeta)+z_+)\leq \e\tanh \bigl(\Psi(a,\zeta)+|z_+|)\leq \tanh \bigl(\Psi(a,\zeta)+\e |z_+|),
	\end{align*}
	which clearly gives the desired lower bound in \eqref{ex:lem2:eq1}.
	
	The upper bound of \eqref{ex:lem2:eq1} requires a bit more work. From the left-hand side of \eqref{add:eq-3}, for any $M>0,$
	\begin{align*}
	&\tanh(\Psi(a,\zeta)-2\sigma^2(0)-\sigma(0)M)\p(|g|\leq M)-\p(|g|>M)\\
	&\leq \e \tanh(\Psi(a,\zeta)-2\sigma^2(0)-\sigma(0)|g|)I(|g|\leq M)\\
	&\quad +\e \tanh(\Psi(a,\zeta)-2\sigma^2(0)-\sigma(0)|g|)I(|g|> M)\\
	&=\e \tanh(\Psi(a,\zeta)-2\sigma^2(0)-\sigma(0)|g|)\leq a.
	\end{align*}
	Hence,
	\begin{align}\label{ex:lem2:eq2} 
	\tanh(\Psi(a,\zeta)-2\sigma^2(0)-\sigma(0)M)&\leq \frac{a+\p(|g|>M)}{\p(|g|\leq M)}.
	\end{align}
	From now on, we choose
	$$
	M=\max(8/\sqrt{2\pi},\sqrt{-2\log (1-a)}\bigr).
	$$
	First, note that from L'H\^opital's rule,
	\begin{align*}
	\lim_{a\uparrow 1}\frac{\tanh^{-1}a}{\sqrt{-2\log (1-a)}}&=\lim_{a\uparrow 1}\frac{\frac{1}{1-a^2}}{\frac{1}{(1-a)\sqrt{-2\log(1-a)}}}=\lim_{a\uparrow 1}\frac{\sqrt{-2\log(1-a)}}{1+a}=\infty.
	\end{align*}
	This means that there exists a constant $c>0$ such that
	\begin{align}\label{ex:lem2:eq3} 
	c\tanh^{-1}(a)\geq \sqrt{-2\log (1-a)},\,\,\forall a\in [0,1).
	\end{align}
	Second, this choice of $M$ also implies that
	\begin{align*}
	\p(|g|>M)\leq \frac{2}{M\sqrt{2\pi}}e^{-M^2/2}\leq \frac{1-a}{4},
	\end{align*} 
	where the first inequality is the usual tail bound for the Gaussian random variable. From this, 
	$$
	\frac{a+\p(|g|>M)}{\p(|g|\leq M)}\leq \frac{1+3a}{3+a}<1
	$$
	so that from \eqref{ex:lem2:eq2} and \eqref{ex:lem2:eq3},
	\begin{align}
	\begin{split}\label{ex:lem2:eq4}
	\Psi(a,\zeta)&\leq \tanh^{-1}\Bigl(\frac{a+\p(|g|>M)}{\p(|g|\leq M)}\Bigr)+2\sigma^2(0)+\sigma(0)M\\
	&\leq \tanh^{-1}\Bigl(\frac{1+3a}{3+a}\Bigr)+2\sigma^2(0)+\sigma(0)M\\
	&\leq \tanh^{-1}\Bigl(\frac{1+3a}{3+a}\Bigr)+2\sigma^2(0)+\sigma(0)\Bigl(c\tanh^{-1}(a)+\frac{8}{\sqrt{2}}\Bigr).
	\end{split}
	\end{align}
	To finish our proof, it remains to control the first term. Note that for any $0\leq x\leq y<1,$ 
	\begin{align*}
	\tanh^{-1}(y)-\tanh^{-1}(x)&=\int_x^{y}\frac{1}{1-w^2}dw\\
	&\leq \int_x^{y}\frac{1}{1-w}dw=-\log\frac{1-y}{1-x}.
	\end{align*}
	In particular, if we take $x=a$ and $y=(1+3a)/(3+a),$ then $0\leq x\leq y<1$ and
	\begin{align*}
	\tanh^{-1}\Bigl(\frac{1+3a}{3+a}\Bigr)-\tanh^{-1}(a)&\leq -\log\frac{1-y}{1-x}\\
	&=-\log \frac{\frac{3+a-1-3a}{3+a}}{1-a}=-\log \frac{2}{3+a}\leq \log 2.
	\end{align*}
	This and \eqref{ex:lem2:eq4} together complete our proof.
\end{proof}

\subsection{Convexity and directional derivative of $\oP_\mu^{\ef}(\lambda,\zeta)$} 
Finally, we establish two key properties of the functional $(\ef,\lambda,\zeta) \to \oP_\mu^{v}(\lambda,\zeta)$. Let $\mu\in \MM_{0,1}$ be fixed and let $q=\int a^2\mu(da)$. For any $\zeta_0,\zeta_1\in \mathcal{M}_{0,1-q}$, $\ef_0,\ef_1\in \oV$, and $\lambda_0,\lambda_1\in \mathbb{R}$, denote, for $b\in [0,1],$
\begin{align*}
\zeta_b&=(1-b)\zeta_0+b\zeta_1,\\
\ef_b&=(1-b)\ef_0+b\ef_1,\\
\lambda_b&=(1-b)\lambda_0+b\lambda_1.
\end{align*}
First, we show that $(\ef,\lambda,\zeta) \to \oP_\mu^{v}(\lambda,\zeta)$ is convex.
\begin{lem}\label{lem29}
	We have that 
	\begin{align*}
	\oP_\mu^{\ef_b}(\lambda_b,\zeta_b)&\leq (1-b)\oP_\mu^{\ef_0}(\lambda_0,\zeta_0)+b\oP_\mu^{\ef_1}(\lambda_1,\zeta_1).
	\end{align*}
	Moreover, if $b\in (0,1)$, this inequality is strict if either $\zeta_0\neq \zeta_1$ or $\lambda_0a+\ef_0(a)\not\equiv \lambda_1a+\ef_1(a)$ on $\supp(\mu)\cap (0,1)$.
\end{lem}
\begin{proof}
	This follows immediately from \eqref{convexity0} and \eqref{convexity}.
\end{proof}

\begin{remark}\rm
	\label{rmk1}
	From this lemma and the Lipschitz property \eqref{ex:lip}, if the support of $\mu$ contains a point in $(0,1)$ then, for any fixed $\ef\in \oV$, the functional $(\lambda,\zeta)\to \oP_\mu^{\ef}(\lambda,\zeta)$ is a strictly convex and continuous on $\mathbb{R}\times\mathcal{M}_{0,1-q}$. This guarantees the existence of the unique minimizer $(\lambda_\mu^\ef,\zeta_\mu^\ef)$ in $\oP_\mu^{\ef}=\inf_{(\lambda,\zeta)\in \mathbb{R}\times \mathcal{M}_{0,1-q}}\oP_\mu^\ef(\lambda,\zeta).$
\end{remark}

Next, we show that the directional derivative of $(\ef,\zeta) \to \oP_\mu^{v}(0,\zeta)$ exists and write down an explicit expression under a certain assumption. Recall the optimal process $u_{a,\zeta,x}$ from \eqref{ex:eq2}.

\begin{lem}\label{lem:dd}
	Let $\zeta_0,\zeta_1\in \mathcal{M}_{0,1-q}$ and $\ef_0,\ef_1\in \oV$. Assume that  $\ef_0=\ef_1$ on $[\delta,1)$ for some $\delta\in (0,1).$ Then we have that
	\begin{align}
	\frac{d}{db}\oP_{\mu}^{\ef_b}(0,\zeta_b)\Bigl|_{b=0^+}
	&=\frac{1}{2}\int_0^{1-q}\xi_q''(s)(\zeta_1(s)-\zeta_0(s))\Bigl(\int_{[0,1)}\e u_{a,\zeta_0,\ef_0(a)}(s)^2\mu(da)-s\Bigr) ds
	\nonumber
	\\
	&\quad+\int_{[0,1)}(\ef_1(a)-\ef_0(a))\partial_x\Phi_{a,\zeta_b}(0,\ef_0(a))\mu(da).
	\label{lem:dd:eq1}
	\end{align}
\end{lem}	

\begin{proof}
	Recall from \eqref{diff} that, for any $a\in [0,1)$ and $b\in [0,1),$
	\begin{align*}
	\begin{split}
	\frac{d}{db}\Phi_{a,\zeta_b}(0,\ef_b(a))&=\frac{1}{2}\int_0^{1-q}\xi_q''(s)\bigl(\zeta_1(s)-\zeta_0(s)\bigr)\e u_{a,\zeta_b,\ef_b(a)}(s)^2ds\\
	&\qquad+(\ef_1(a)-\ef_0(a))\partial_x\Phi_{a,\zeta_0}(0,\ef_b(a)),
	\end{split}
	\end{align*}
	where $u_{a,\zeta_b,\ef_b(a)}$ is the process in \eqref{control} with $x=\ef_b(a)$ and $\zeta=\zeta_b$. Furthermore, the right derivative of $\Phi_{a,\zeta_b}(0,\ef_b(a))$ also exists at $b=0$ and is equal to the right hand side of the above equation with $b=0.$ From the assumption $\ef_0=\ef_1$ on $[\delta,1),$ we have that from \eqref{subphi:eq1} and  \eqref{ex:lem2:eq1},
	\begin{align*}
	&\sup_{a\in [0,1),b\in [0,1]}\bigl|(\ef_1(a)-\ef_0(a))\partial_x\Phi_{a,\zeta_b}(0,\ef_b(a))\Bigr|<\infty.
	\end{align*}
	From this uniform upper bound and the bounded convergence theorem, it follows that
	\begin{align*}
	\frac{d}{db}\oP_{\mu}^{\ef_b}(0,\zeta_b)\Bigl|_{b=0^+}
	&=\lim_{b\downarrow 0}\int_{[0,1)}\frac{\Phi_{a,\zeta_b}(0,\ef_b(a))-\Phi_{a,\zeta_0}(0,\ef_0(a))}{b}\mu(da)\\
	&\qquad-\frac{1}{2}\int_0^{1-q}\xi_q''(s)s(\zeta_1(s)-\zeta_0(s))ds\\
	&=\frac{1}{2}\int_{[0,1)}\Bigl(\int_0^{1-q}\xi_q''(s)(\zeta_1(s)-\zeta_0(s))\e u_{a,\zeta_0,\ef_0(a)}(s)^2ds\Bigr)\mu(da)\\
	&\qquad+\int_{[0,1)}(\ef_1(a)-\ef_0(a))\partial_x\Phi_{a,\zeta_0}(0,\ef_0(a))\mu(da)\\
	&\qquad-\frac{1}{2}\int_0^{1-q}\xi_q''(s)s(\zeta_1(s)-\zeta_0(s))ds.
	\end{align*}
	Finally, combining the first and third equations together by using Fubini's theorem, the above uniform upper bound, and \eqref{ex:lem2:eq1} completes our proof.
\end{proof}

\begin{remark}\rm
	\label{rmk2}
	Suppose that $\zeta_0$ is a minimizer of the variational formula $\inf_{\zeta\in \mathcal{M}_{0,1-q}}\oP_\mu^{\ef_{\zeta_0}}(0,\cdot)$. In a standard manner as \cite[Proposition 1]{C17}, \cite[Proposition 1.1]{JT17}, or \cite[Lemma 4.13]{SKbonus}, by using Fubini's theorem and the above uniform upper bounds of $\partial_x\Phi_{a,\zeta}$ and $\partial_{xx}\Phi_{a,\zeta}$ in \eqref{subphi:eq1},  
	the directional derivative \eqref{lem:dd} and the minimality of $\zeta_0$ together yield that we must have
	\begin{align*}
	\int_{[0,1)}\e\bigl(\partial_x\Phi_{a,\zeta_0}(s,X_{a,\zeta_0,\ef_{\zeta_0}(a)}(s))\bigr)^2\mu(da)&=s,\\
	\xi_q''(s)\int_{[0,1)}\e\bigl(\partial_{xx}\Phi_{a,\zeta_0}(s,X_{a,\zeta_0,\ef_{\zeta_0}(a)}(s))\bigr)^2\mu(da)&\leq 1,
	\end{align*} 
	for any $s$ in the support of $\zeta_0$.
\end{remark}
 
\bibliographystyle{plain}
\bibliography{bibliography}

\end{document}